\documentclass[reqno]{amsart}

\usepackage[latin1]{inputenc}
\usepackage[T1]{fontenc}
\usepackage[ngerman,english]{babel}
\usepackage{latexsym}
\usepackage{amsthm,amssymb,amsmath,amscd}
\usepackage{dsfont}
\usepackage{enumerate,enumitem}
\usepackage{ifthen}
\usepackage[super]{nth}
\usepackage{mathtools}
\usepackage{color}

\newcommand{\Aa}{\mathcal{A}}

\newcommand{\Dd}{\mathcal{D}}
\newcommand{\Gg}{\mathcal{G}}

\newcommand{\Bb}{\mathcal{B}}
\newcommand{\Cc}{\mathcal{C}}

\newcommand{\Rr}{\mathcal{R}}
\newcommand{\Hh}{\mathcal{H}}
\newcommand{\Ww}{\mathcal{W}}

\newcommand{\Ll}{\mathcal{L}}

\newcommand{\Ff}{\mathcal{F}}

\newcommand{\Kk}{\mathcal{K}}

\newcommand{\R}{\mathbb{R}}
\newcommand{\N}{\mathbb{N}}

\newcommand{\C}{\mathbb{C}}

\newcommand{\forms}{\mathbf{s}}
\newcommand{\formb}{\mathbf{b}}
\newcommand{\formc}{\mathbf{c}}

\newcommand{\Cn}{\C^n}
\newcommand{\Cnn}{\C^{n\times n}}

\newcommand{\Rn}{\R^n}

\newcommand{\Loneloc}{L^1_{\operatorname{loc}}}
\newcommand{\dprod}[2]{\left( #1,#2 \right)}
\newcommand{\iprod}[2]{\langle #1,#2 \rangle}
\newcommand{\Coo}{C_0^\infty}

\newcommand{\norm}[1]{\lVert#1\rVert}

\newcommand{\abs}[1]{\lvert#1\rvert}

\newcommand{\set}[2]{\{#1:#2\}}

\newcommand{\ls}{\operatorname{span}}
\newcommand{\sgn}{\operatorname{sgn}}
\newcommand{\ran}{\operatorname{ran}}
\newcommand{\dom}{\operatorname{dom}}

\newcommand{\diag}{\operatorname{diag}}
\newcommand{\dist}{\operatorname{dist}}

\newcommand{\re}{\operatorname{Re}}
\newcommand{\im}{\operatorname{Im}}
\newcommand{\e}{\operatorname{e}}
\renewcommand{\d}{\,{\rm d}}
\renewcommand{\i}{\operatorname{i}}
\newcommand{\coker}{\operatorname{coker}}

\newcommand{\essran}{\operatorname{ess} \ran}
\newcommand{\boldpi}{\boldsymbol\pi}
\DeclareMathOperator*{\esssup}{ess\,sup}

\newcommand{\eps}{\varepsilon}
\newcommand{\la}{\lambda}

\newcommand{\pspec}{\sigma_{\operatorname{p}}}
\newcommand{\essspec}{\sigma_{\operatorname{e2}}}

\newcommand{\defeq}{\vcentcolon=}
\newcommand{\eqdef}{=\vcentcolon}

\newtheorem{thm}{Theorem}[section]
\newtheorem{cor}[thm]{Corollary}
\newtheorem{lem}[thm]{Lemma}
\newtheorem{prop}[thm]{Proposition}

\theoremstyle{definition}
\newtheorem{defi}[thm]{Definition}
\newtheorem{exple}[thm]{Example}

\newtheorem{asm}[thm]{Assumption}

\theoremstyle{remark}
\newtheorem{rem}[thm]{Remark}

\numberwithin{equation}{section}

\setlist{leftmargin=9mm}

\mathtoolsset{showonlyrefs}

\begin{document}
	
\title{Schur complement dominant operator matrices}

\author{Borbala Gerhat}

\address{
	Mathematical Institute, University of Bern, Sidlerstrasse 5, 3012 Bern, CH}
\curraddr{
Department of Mathematics, Faculty of Nuclear Sciences and Physical Engineering, Czech Technical University in Prague, Trojanova 13, 120 00 Prague, Czech Republic}

\email{borbala.gerhat@fjfi.cvut.cz}



\begin{abstract}

We propose a method for the spectral analysis of unbounded operator matrices in a general setting which fully abstains from standard perturbative arguments. Rather than requiring the matrix to act in a Hilbert space $\mathcal{H}$, we extend its action to a suitable distributional triple $\Dd \subset \Hh \subset \Dd_-$ and restrict it to its maximal domain in $\Hh$. The crucial point in our approach is the choice of the spaces $\Dd$ and $\Dd_-$ which are essentially determined by the Schur complement of the matrix. We show spectral equivalence between the resulting operator matrix in $\Hh$ and its Schur complement, which allows to pass from a suitable representation of the Schur complement (e.g.~by generalised form methods) to a representation of the operator matrix. We thereby generalise classical spectral equivalence results imposing standard dominance patterns.

The  abstract results are applied to damped wave equations with possibly unbounded and/or singular damping, to Dirac operators with Coulomb-type potentials, as well as to generic second order matrix differential operators. By means of our methods, previous regularity assumptions can be weakened substantially.

\end{abstract}

\thanks{The author gratefully acknowledges the support of the \emph{Swiss National Science Foundation}, grant No.~169104. This work was also partially supported by the EXPRO grant No.~20-17749X of the \emph{Czech Science Foundation}. The author would like to thank the \emph{SEMP} student exchange program for their support, as well as the \emph{Queen's University Belfast} for their hospitality. She expresses her gratitude for fruitful mathematical discussions with \emph{O. Ibrogimov} and \emph{P. Siegl}, Graz University of Technology. Finally, she thanks \emph{J.-C. Cuenin}, Loughborough University, and her former PhD supervisor \emph{C. Tretter}, University of Bern, for drawing her attention to the relevant references \cite{Esteban-Loss-2007,Esteban-Loss-2008} and~\cite{Schimmer-Solovej-Tokus-2020}, respectively}.

\date{23 May 2021}
\maketitle

\section{Introduction}

Motivated by a wide range of applications, unbounded (non-self-adjoint) operator matrices emerge from coupled systems of linear partial differential equations and have been of considerable interest, see e.g.\ the monograph \cite{Tretter-2008} and the references therein.
%
%
Typically, the spectral analysis of such problems is rather challenging, starting with the non-trivial task of determining a suitable domain of definition on which the resulting operator matrix is closed and has non-empty resolvent set. An often fruitful approach is to establish a certain spectral correspondence between the operator matrix
\begin{equation*}
	\Aa = \left(
	\begin{array}{cc}
		A & B \\
		C & D
	\end{array}
	\right): \Hh \supset \dom \Aa \to \Hh, \qquad \Hh = \Hh_1 \oplus \Hh_2,
\end{equation*}
and one of its (two) so-called Schur complements $S (\cdot)$, the (scalar) operator family
\begin{equation}
	S(\lambda)=A-\lambda-B(D-\lambda)^{-1}C : \Hh_1 \supset \dom S(\la) \to \Hh_1, \qquad \lambda \in \rho (D).
\end{equation}
Although the latter gives rise to a non-linear spectral problem, in general, more methods are available for its analysis.
Up to now, the spectral correspondence described above has mainly been achieved by taking advantage of certain patterns of relative boundedness between the entries of the operator matrix, see e.g.~the pioneering work~\cite{Nagel-1989} or the more recent article~\cite{Langer-Strauss-2017}. There are several results in \cite{Esteban-Loss-2007, Esteban-Loss-2008, Freitas-Siegl-Tretter-2018, Ibrogimov-2017, Ibrogimov-Siegl-Tretter-2016, Ibrogimov-Tretter-2017, Schimmer-Solovej-Tokus-2020}, however, which seem to abstain from this type of perturbative argument. Aiming to capture these phenomena, we propose a more general framework for operator matrices and allow a systematic approach to the spectral analysis of a wider class of problems. We point out that even though our approach was inspired by the conceptual observations in \cite{Freitas-Siegl-Tretter-2018, Ibrogimov-2017, Ibrogimov-Siegl-Tretter-2016, Ibrogimov-Tretter-2017}, its scope goes beyond the latter; not only are our results due to their abstract nature much more versatile, but even applied to particular problems in the mentioned references, they allow much weaker natural and in some sense minimal (even distributional) regularity of the coefficients, see the applications in Sections~\ref{sec:DWE} and~\ref{sec:singular.matrix.DE}.
%

%
Our method combines a distributional technique with the assumption that, in a certain sense, the Schur complement dominates all other terms in the Frobenius-Schur factorisation of the resolvent. Said distributional approach consists of extending the operator matrix to a suitable triplet of Hilbert spaces
\begin{equation*}
	\Dd \subset \Hh \subset \Dd_-,
\end{equation*}
where each of the above inclusions represents a continuous embedding with dense range. More precisely, we define the action of $\Aa$ on a space of test functions $\Dd$ with values in a space of distributions $\Dd_-$ and consider its restriction $\Aa_0$ to the maximal domain
\begin{equation}
	\dom \Aa_0 = \set{\mathbf x \in \Dd}{\Aa \mathbf x \in \Hh}.
\end{equation}
This method has been employed successfully in the past; however, except in~\cite{Esteban-Loss-2007, Esteban-Loss-2008, Schimmer-Solovej-Tokus-2020}, it seems that the spaces of test functions and distributions have consistently been determined by the underlying patterns of relative boundedness within the operator matrix, e.g.\ as form domain of some entry and its dual space in \cite[Chap.\ 1.2.1]{Ammari-Nicaise-2015} or \cite{Jacob-Tretter-Trunk-Vogt-2018}.
The key novelty in our approach is to choose the spaces $\Dd$ and $\Dd_-$ in a way that the Schur complement $S (\cdot)$ consists of bounded and boundedly invertible operators between their first components $\Dd_S$ and $\Dd_{-S}$. This choice guarantees the required dominance of the Schur complement and allows us to relate invertibility and semi-Fredholmness of $\Aa_0$ to invertibility and semi-Fredholmness of $S_0 (\cdot)$ defined as family of maximal operators in $\Hh_1$. We thus obtain equivalence of their (point and essential) spectra, which in applications might eventually lead to desired properties like closedness and non-empty resolvent set of the operator matrix and its semigroup generation. Our technique can be viewed as a strategy to relate a well-behaved representation of the Schur complement, which can be implemented e.g.~by the generalised form methods in~\cite{Almog-Helffer-2015}, to a well-behaved representation of the operator matrix. We thereby extend classical results on spectral equivalence due to standard dominance patterns as e.g.~diagonal dominance in~\cite{Nagel-1989}, see Remark~\ref{rem:diag.dom}.

Although not closely related to our framework, we mention another non-standard approach in~\cite{Nagel-1990} towards the spectral analysis of operator matrices. The setting therein, however, covers problems of different type and essentially aims at incorporating mixing boundary conditions. From a structural point of view,~\cite{Nagel-1990} is more restrictive than our approach since it requires diagonal dominance of the underlying operator matrix; the dominance assumption therein was relaxed to more general patterns in \cite{Batkai-Binding-Dijksma-Hryniv-Langer-2005}.

We apply our abstract results to linearly damped wave equations on $\Omega \subset \Rn$ with non-negative damping $a$ and potential $q$, which give rise to the operator matrix
\begin{equation}
\label{eq:intro.DWE}
	\Aa = \left(
	\begin{array}{cc}
	0 & I \\
	\Delta - q & -2a
	\end{array}
	\right)
\end{equation}
in a suitable Hilbert space. Unlike in most of the previously existing results, see e.g.\ \cite{Ammari-Nicaise-2015, Gesztesy-Goldstein-Holden-Teschl-2012, Jacob-Tretter-Trunk-Vogt-2018}, we allow the damping and potential to be singular and/or unbounded at infinity and do not require the damping to be relatively bounded with respect to $\Delta - q$. To the best of our knowledge, this case has only been covered in the works \cite{Freitas-Siegl-Tretter-2018, Ikehata-Takeda-2020} so far.  In \cite{Freitas-Siegl-Tretter-2018}, essentially assuming that $a \in W^{1,\infty}_{\operatorname{loc}} (\overline \Omega)$ and that for every $\eps >0 $ there exists $C_\eps \ge 0$ with $\abs{\nabla a} \le \eps a^\frac{3}{2} + C_\eps$, see~\cite[Asm.~I]{Freitas-Siegl-Tretter-2018} for the precise more general assumptions, the spectral equivalence of $\Aa$ to its second Schur complement
\begin{equation}
	S (\lambda) = - \frac{1}{\lambda} (- \Delta + q + 2 \lambda a  + \lambda^2), \qquad \la \in \C \setminus \{0\},
\end{equation}
as an operator family in $L^2 (\Omega)$ is established. This leads to the generation of a contraction semigroup and thus to the existence and uniqueness of the solutions to the underlying equation. In \cite{Ikehata-Takeda-2020} on the other hand, an approximation procedure is performed to construct a unique weak solution, whose norm and total energy are shown to decay polynomially in time.
\emph{Merely assuming} $a, q \in \Loneloc (\Omega)$, our methods provide spectral correspondence between matrix and Schur complement, as well as the generation of a strongly continuous contraction semigroup. We thereby significantly generalise \cite{Freitas-Siegl-Tretter-2018}, where technical assumptions on growth and regularity of the damping are needed in order to describe the operator domain of the Schur complement. Notice that under the latter more restrictive assumptions, our realisation of the operator matrix \eqref{eq:intro.DWE} coincides with the one defined in \cite{Freitas-Siegl-Tretter-2018}, see Remark~\ref{rem:DWE.same.operator}. Moreover, we point out that our method can equally be employed to realise more general distributional dampings including the ones considered in~\cite{Ammari-Nicaise-2015, Krejcirik-Royer-2022-arxiv, Krejcirik-Kurimaiova-2020}, see Remark~\ref{rem:DWE.distr.damping}. 
As another application of our results, we present second order matrix differential operators of the form
\begin{equation}
	\label{eq:singular.DE.matrix}
	\Aa = \left(
	\begin{array}{cc}
		- \Delta + q & \nabla \cdot \formb \\
		\formc \cdot \nabla & d
	\end{array}
	\right)
\end{equation}
on $\Omega \subset \Rn$ with low regularity coefficients. Problems of this type arise in areas like magnetohydrodynamics or astrophysics and have been previously studied in e.g.\ \cite{Boegli-Marletta-2020, Ibrogimov-2017, Ibrogimov-Siegl-Tretter-2016, Ibrogimov-Tretter-2017, Konstantinov-1998, Kurasov-Lelyavin-Naboko-2008}. Our methods allow us to avoid typical technical assumptions like $q \in C (\Omega)$, $\formb, \formc \in C^1 (\Omega)^n$ and $d \in C^1 (\Omega)$, see e.g\ \cite{Ibrogimov-Siegl-Tretter-2016}. Under certain \emph{natural weak regularity conditions} on the coefficients, see Assumption~\ref{asm:IV}, we are able to define a closed realisation of the operator matrix \eqref{eq:singular.DE.matrix} in the space $L^2 (\Omega) \oplus L^2 (\Omega)$ with non-empty resolvent set. We show spectral equivalence to its first Schur complement and, imposing additional assumptions on the structure of the problem, that the operator matrix generates a  strongly continuous contraction semigroup.

Based on the application to Dirac operators subject to Coulomb type potentials in three dimensions, we compare our work to the results in~\cite{Esteban-Loss-2007,Esteban-Loss-2008,Schimmer-Solovej-Tokus-2020}; the latter came to our attention only after developing the method at hand and are placed in a more restrictive self-adjoint setting. More precisely, we construct a self-adjoint realisation of the symmetric matrix differential expression
\begin{equation}\label{eq:intro.Hardy}
	\Aa = \begin{pmatrix}
		V + 1 & - \i \sigma \cdot \nabla \\
		- \i \sigma \cdot \nabla & V - 1
	\end{pmatrix}
\end{equation}
in $L^2 (\R^3)^2 \oplus L^2 (\R^3)^2$; see Section~\ref{sec:Dirac} for details e.g.~on the precise definition of $\sigma \cdot \nabla$. Besides assuming that the real-valued potential $V$ is square integrable and bounded above, we suppose that it satisfies a Hardy-Dirac inequality, i.e.~that there exists a constant $\Lambda > \sup_{\R^3} V - 1$ such that
\begin{equation}
	\label{eq:intro.HD}
	\int_{\R^3} \frac{|\sigma \cdot \nabla f|^2}{1 + \Lambda - V} \d x + \int_{\R^3} (V + 1 - \Lambda) |f|^2 \d x \ge 0
\end{equation}
for $f\in\Coo(\R^3)^2$. It was shown in~\cite{Dolbeault-Esteban-Loss-Vega-2004,Dolbeault-Esteban-Sere-2000} that the latter holds for a class of potentials which in particular include Coulombic ones. Inequality~\eqref{eq:intro.HD} is crucial and translates into coercivity of the first Schur complement on its form domain for certain spectral parameters, which eventually leads to the desired self-adjointness by establishing a spectral gap.

In the innovative works~\cite{Esteban-Loss-2007,Esteban-Loss-2008}, the authors take a similar approach and construct the same self-adjoint Dirac operator by means of~\eqref{eq:intro.HD}. Even though they engage both a dominant Schur complement and a distributional setting, comparing to our result they require an additional assumption, see~\eqref{eq:EL.grad} below, which ensures that smooth compactly supported functions lie in the operator domain of the Schur complement. Since the Coulomb potentials $V(x) = - \nu / |x|$ satisfy latter condition, \cite{Esteban-Loss-2007, Esteban-Loss-2008} provides a distinguished self-adjoint extension of the minimal operator for the full range of parameters $\nu \in (0,1]$. In~\cite{Esteban-Lewin-Sere-2019,Esteban-Lewin-Sere-2021}, the form domain $\Dd_S$ was further described in order to obtain min-max principles for the eigenvalues in the spectral gap. Approaching the problem from a slightly different more abstract perspective, in~\cite{Schimmer-Solovej-Tokus-2020} a robust method was developed to obtain distinguished self-adjoint extensions of symmetric operators with gap (as well as to establish suitable min-max principles for their eigenvalues in the gap). Applied to the case of Dirac operators with Coulomb potentials, the technique therein allows to omit the requirement~\eqref{eq:EL.grad} in~\cite{Esteban-Loss-2007,Esteban-Loss-2008} and recovers the same set of assumptions as our method.  We point out that, while the constructions in~\cite{Esteban-Loss-2008,Schimmer-Solovej-Tokus-2020} are implemented on a fairly general abstract level, unlike our results  they strongly rely on the symmetric structure of the underlying problem.


We next study Klein-Gordon equations with purely imaginary, merely locally square integrable potentials. The physically relevant case of real potentials, in which the problem exhibits an indefinite structure, was studied using several different methods, see e.g.~\cite{Langer-Najman-Tretter-2006,Langer-Najman-Tretter-2008} where Krein spaces were employed or~\cite{Veselic-1991} where for linear potentials a different approach via oscillatory integrals was taken. The case of (purely) imaginary potentials, however, seems to have not been considered so far. In fact, the problem is then equivalent to a wave equation with suitable damping and potential and can be treated by applying our previously obtained results. We show spectral equivalence between the matrix and its second Schur complement and conclude that the resulting operator matrix is densely defined and boundedly invertible. In the particular case of the potential $V(x) = \i x$ in one dimension, we prove that the spectrum of the matrix is empty; this is not surprising considering similar results related to the complex Airy operator in e.g.~\cite{Almog-2008,Almog-Grebenkov-Helffer-2019}.

Finally, we show that our method is not limited to applications where the spaces $\Dd_S$ and $\Dd_{-S}$ are the form domain of the Schur complement and its dual space. We demonstrate this fact based on a constant coefficient differential operator matrix where $\Dd_S = H^1 (\Rn)$ and $\Dd_{-S} = H^{-2}(\Rn)$, while the natural form domain of one of the  Schur complements
\begin{equation}
	S(\lambda) = \Delta - \lambda + \Delta^2 (\Delta - \lambda)^{-1} \sqrt{- \Delta}, \qquad \lambda \in \C \setminus (-\infty, 0],
\end{equation}
in $L^2(\R^n)$ is the fractional Sobolev space $H^\frac32 (\Rn)$. We stress that this example serves an illustrative purpose only and is kept as simple as possible; relevant problems with similar structure can be found in \cite{Ibrogimov-2017,Ibrogimov-Siegl-Tretter-2016,Ibrogimov-Tretter-2017}.
This paper is organised as follows. In Section~\ref{sec:abstract.results}, we present our main abstract results and lay the ground for the spectral equivalence between $\Aa_0$ and $S_0$; the main theorems are Theorem~\ref{thm:first.implication} and~\ref{thm:second.implication}, where we show that invertibility or semi-Fredholmness of $\Aa_0$ and $S_0$, respectively, imply invertibility or semi-Fredholmness of $S_0$ and $\Aa_0$. In Section~\ref{sec:matrix.Schur.complement}, we translate the results of Section~\ref{sec:abstract.results} into several corollaries providing spectral equivalence between the operator matrix and its Schur complement as an operator family. In Section~\ref{sec:DWE}, we apply the established spectral equivalence to the damped wave equation with unbounded and/or singular damping and potential; in particular, the main Theorem~\ref{thm:DWE} states the generation of a contraction semigroup for the underlying problem. 
Section~\ref{sec:singular.matrix.DE} contains the application of our results to second order matrix differential operators with low regularity coefficients; the main results are Theorem~\ref{thm:matrix.DE}, which states spectral equivalence between matrix and Schur complement, and Theorem~\ref{thm:matrix.DE.accretive}, which shows the generation of a contraction semigroup under additional structural assumptions. Section~\ref{sec:Dirac} is devoted to Dirac operators with Coulomb type potentials; in the main Theorem~\ref{thm:Dirac}, we construct a self-adjoint realisation by establishing a spectral gap. While Section~\ref{sec:KG} concerns the Klein-Gordon equation with purely imaginary potential, Section~\ref{sec:Laplace.powers} further illustrates the nature of our results based on a constant coefficient matrix differential operator.

\subsection{Notation and preliminaries}

Let $n \in \N$ be the spatial dimension, let $\iprod{\cdot}{\cdot}_{\Cn}$ be the inner product and $\abs{\cdot}$ the euclidean norm on $\Cn$. We denote the real scalar product and tensor product on $\Cn$, respectively, by
\begin{equation}
	\xi \cdot \eta  = \sum_{j = 1}^{n} \xi_j \eta_j, \qquad \xi \otimes \eta = (\xi_j \eta_k)_{j,k} \in \Cnn, \qquad \xi, \eta \in \Cn.
\end{equation}

For $\Omega \subset \Rn$ and a measurable function $m: \Omega \to \C$, we denote by $\dom m$ the maximal domain of the corresponding multiplication operator in $L^2 (\Omega)$.
%
%
Moreover, recall the definition of the essential range
\begin{equation}
	\essran m = \set{z \in \C}{\lambda^n (m^{-1}(B_\varepsilon (z)))>0 ~ \textnormal{for any} ~ \varepsilon > 0},
\end{equation}
where $\lambda^n$ is the $n$-dimensional Lebesgue measure, and note that $\essran m$ is closed.
The space of matrix-valued locally integrable functions will be denoted
\begin{equation}
	\Loneloc (\Omega)^{n \times n} = \set{M = (M_{jk})_{j,k}: \Omega \to \Cnn}{M_{jk} \in \Loneloc (\Omega), ~1 \le j,k \le n};
\end{equation}
note that this is equivalent to $\norm{M} \in \Loneloc (\Omega)$ for any norm $\norm{\cdot}$ on $\Cnn$. The spaces $L^\infty (\Omega)^n$ and $L^2_{\operatorname{loc}} (\Omega)^n$ shall be defined analogously. Moreover, $L^2 (\Omega; \omega)$ denotes the weighted $L^2$-space on $\Omega$ with a non-negative measurable weight $\omega$.

Let $\Gg_1$, $\Gg_2$ and $\Gg$ be Hilbert spaces. The set of closed linear operators from $\Gg_1$ to $\Gg_2$ will be denoted by $\Cc(\Gg_1, \Gg_2)$, the space of everywhere defined and bounded operators by $\Bb (\Gg_1, \Gg_2)$ and the space of compact operators by $\Kk(\Gg_1, \Gg_2)$. Moreover, as usual, we write $\Cc (\Gg) = \Cc(\Gg, \Gg)$, $\Bb (\Gg) = \Bb(\Gg, \Gg)$ and $\Kk (\Gg) = \Kk (\Gg,\Gg)$.
We denote the duality pairing and anti-dual space of $\Gg$ by
\begin{equation}
	\dprod{\phi}{f}_{\Gg^* \times \Gg} = \phi (f), \quad f \in \Gg, \quad \phi \in \Gg^* = \set{\phi : \Gg \to \C}{\phi ~ \textnormal{antilinear, bounded}};
\end{equation}
note that we hereby adopt the conventions in \cite{Edmunds-Evans-1987} and work with \emph{antilinear} functionals, which allows us to identify bounded sesquilinear forms on $\Gg$ with bounded linear operators $\Gg \to \Gg^*$. This convention includes defining the space of distributions on $\Omega\subset \Rn$, denoted $\Dd'(\Omega)$, as \emph{antilinear} continuous functionals on the space of test functions $\Coo(\Omega)$.
The resolvent set, spectrum and point spectrum of a linear operator $T$ in $\Gg$ will be denoted by $\rho (T)$, $\sigma (T)$ and $\pspec (T)$, respectively. We use the following (one of five in general non-equivalent) definitions of the essential spectrum
\begin{equation}
	\essspec (T) = \set{\lambda \in \C}{T - \lambda \notin \Ff_+ (\Gg)};
\end{equation}
here the semi-Fredholm operators in $\Gg$ with finite nullity and deficiency, respectively, are defined in the following way
\begin{equation}
	\begin{aligned}
		\Ff_+ (\Gg) & = \set{T \in \Cc (\Gg)}{\ran T ~ \textnormal{closed}, ~ \dim \ker T < \infty}, \\
		\Ff_- (\Gg) & = \set{T \in \Cc (\Gg)}{\ran T ~ \textnormal{closed}, ~ \dim \coker T < \infty},
	\end{aligned}
\end{equation}
and the co-kernel of $T$ is the closed subspace $\coker T = (\ran T)^\perp$. If $T \in \Cc (\Gg)$ and $\ran T$ is closed, one can define a generalised inverse $T^\# \in \Bb (\Gg)$ of $T$ as
\begin{equation}
	\label{eq:gen.inv}
	T^\# x = \begin{cases}
		(T\vert_{\dom T \cap (\ker T)^\perp})^{-1} x & \quad x \in \ran T, \\
		0 & \quad x \in \coker T;
	\end{cases}
\end{equation}
it follows from the closedness of $T$ and $\ran T$ that the above indeed defines a bounded operator. With this definition, if $P$ and $Q$, respectively, denote the orthogonal projections on $\ker T$ and $\coker T$, then
\begin{equation}
	\label{eq:gen.inv.comp}
	\ran T^\# \subset \dom T, \qquad T T^\# = I - Q, \qquad T^\# T = I-P\vert_{\dom T}.
\end{equation}
Note that, since $T$ is closed, so is $\ker T$ and both projections above are bounded.

Throughout the entire paper, $\Hh_1$ and $\Hh_2$ denote complex Hilbert spaces and $\Hh = \Hh_1 \oplus \Hh_2$ their orthogonal sum. The canonical projections of $\Hh$ on $\Hh_1$ and $\Hh_2$, respectively, are denoted by $\pi_1$ and $\pi_2$; their adjoints $\pi_1^*$ and $\pi_2^*$ are the corresponding canonical embeddings of $\Hh_1$ and $\Hh_2$ in $\Hh$.

Finally, we write $\lesssim$ or $\gtrsim$ if the respective inequalities hold with a multiplicative constant depending only on quantities which are fixed and the dependence on which is thus irrelevant.

\section{Main abstract results}
\label{sec:abstract.results}

We present a new abstract setting for the spectral analysis of operator matrices. Under fairly general assumptions, our approach allows to establish a correspondence between semi-Fredholmness and invertibility of an operator matrix and one of its Schur complements; the latter provides a relation between (point and essential) spectra of matrix and Schur complement as an operator family, see Section~\ref{sec:matrix.Schur.complement}.

Our strategy is the following. Rather than directly defining an operator matrix in the product space $\Hh = \Hh_1 \oplus \Hh_2$, we consider its entries acting in suitable triplets \eqref{eq:triplets} and restrict the matrix action to the maximal domain in $\Hh$, see Definition~\ref{def:matrix.Schur.compl}. A crucial point of our construction is the choice of spaces $\Dd_S$ and $\Dd_{-S}$, such that the Schur complement $S: \Dd_S \to \Dd_{-S}$ is bounded and boundedly invertible. Notice also that we construct a generalised Schur complement, which unlike the classical definition does not require the entry $D$ to be boundedly invertible but merely to have a generalised inverse, see Assumption~\ref{asm:I}~(iii) and \eqref{eq:gen.inv}, \eqref{eq:gen.inv.comp}.

Proofs of the results in this section can be found in Section \ref{sec:proofs}.

\subsection{Assumptions and definitions}

We work exclusively with the first Schur complement; clearly, all results in this section hold under analogous assumptions for the second Schur complement, see Remark~\ref{rem:second.Schur.compl}.

\begin{asm} \label{asm:I}
	\begin{enumerate}
		\item Let $\Dd_S$, $\Dd_2$, $\Dd_{-S}$, $\Dd_{-1}$ and $\Dd_{-2}$ be complex Hilbert spaces. Assume that
		\begin{equation}\label{eq:triplets}
			\Dd_S \subset \Hh_1 \subset \Dd_{-S}, \qquad \Dd_2 \subset \Hh_2 \subset \Dd_{-2},
		\end{equation}
		where the corresponding canonical embeddings are continuous and have dense ranges. Moreover, let $\Dd_{-S} \subset \Dd_{-1}$ be continuously embedded.
		\item Assume that the operators $A$, $B$ and $C$ are bounded between the spaces
		\begin{equation}
			A \in \Bb(\Dd_S, \Dd_{-1}), \qquad B \in \Bb (\Dd_2, \Dd_{-1}), \qquad C \in \Bb (\Dd_S, \Dd_{-2}).
		\end{equation}
		\item Let $D_0 \in \Cc (\Hh_2)$ have closed range in $\Hh_2$, let $\dom D_0 \subset \Dd_2$ be dense in $\Dd_2$ and assume that there exist extensions
		\begin{equation}
			D_0 \subset D \in \Bb (\Dd_2, \Dd_{-2}), \qquad D_0^\# \subset D^\ddagger \in \Bb (\Dd_{-2}, \Dd_2). \tag*{//}
		\end{equation}
	\end{enumerate}
\end{asm}

\begin{rem}
		We point out that $\Dd_{-S} = \Dd_{-1}$ in the applications in Sections~\ref{sec:DWE},~\ref{sec:KG} and~\ref{sec:Laplace.powers}. However, allowing $A$ and $B$ to map to a larger space widens the range of applicability of our assumptions, see e.g.\ Sections~\ref{sec:singular.matrix.DE} and~\ref{sec:Dirac} where $\Dd_{-S} \subsetneq \Dd_{-1}$. While $\Dd_{-S}$ plays an essential role and is determined by the Schur complement, the auxiliary space $\Dd_{-1}$ merely provides an environment for the operations needed in order to define the latter; notice that in Sections~\ref{sec:singular.matrix.DE} and~\ref{sec:Dirac} it is clear from the construction that $\Dd_{-1}$ can even be chosen arbitrarily large in a suitable sense. \hfill //
\end{rem}

The operator matrix $\Aa_0$ and its first Schur complement $S_0$ are defined as the following maximal operators in the underlying Hilbert spaces. We emphasise that, although their extensions $\Aa$ and $S$ are assumed to be bounded between suitable spaces, $\Aa_0$ and $S_0$ are in general unbounded in $\Hh$ and $\Hh_1$, respectively.

\begin{defi}
	\label{def:matrix.Schur.compl}
	Let Assumption~\ref{asm:I} be satisfied. We define
	\begin{equation}
	\label{eq:def.Aa.S}
		\Aa \defeq \left(
		\begin{array}{cc}
			A & B \\
			C & D
		\end{array}
		\right) \in \Bb (\Dd, \Dd_-), \qquad S \defeq A - B D^\ddagger C \in \Bb (\Dd_S, \Dd_{-1}),
	\end{equation}
	where $\Dd \defeq \Dd_S \oplus \Dd_2$ and $\Dd_- \defeq \Dd_{-1} \oplus \Dd_{-2}$. 
	Moreover, let the corresponding maximal operators $\Aa_0 \defeq \Aa\vert_{\dom\Aa_0}$ in $\Hh$ and $S_0 \defeq S\vert_{\dom S_0}$ in $\Hh_1$ be defined on their respective domains
	\begin{equation}
		\dom \Aa_0 \defeq \set{(f,g) \in \Dd}{\Aa (f,g) \in \Hh}, \qquad \dom S_0 \defeq \set{f \in \Dd_S}{Sf \in \Hh_1}. \tag*{//}
	\end{equation}
\end{defi}

Notice that if $0 \in \rho(D_0)$, then $D^\ddagger = D^{-1}$ and the definition of $S$ reduces to the standard formula for the Schur complement, see Lemma~\ref{lem:T.inverse} below.

\begin{rem}
	\label{rem:second.Schur.compl}
	All results in the present section hold analogously for the second Schur complement. The assumptions have to be translated in a straightforward way as follows. One assumes the following inclusions between the Hilbert spaces
	\begin{equation}
		\Dd_1 \subset \Hh_1 \subset \Dd_{-1}, \qquad \Dd_S \subset \Hh_2 \subset \Dd_{-S} \subset \Dd_{-2},
	\end{equation}
	where the corresponding canonical embeddings are continuous and all except the embedding $\Dd_{-S} \hookrightarrow \Dd_{-2}$ are assumed to have dense range. Moreover,
	\begin{equation}
		B \in \Bb (\Dd_S, \Dd_{-1}), \qquad C \in \Bb (\Dd_1, \Dd_{-2}), \qquad D \in \Bb (\Dd_S, \Dd_{-2}).
	\end{equation}
	The operator $A_0 \in \Cc (\Hh_1)$ has closed range in $\Hh_1$, $\dom A_0$ is dense in $\Dd_1$ and
	\begin{equation}
		A_0 \subset A \in \Bb (\Dd_1, \Dd_{-1}), \qquad A_0^\# \subset A^\ddagger \in \Bb (\Dd_{-1}, \Dd_1).
	\end{equation}
	Then the matrix $\Aa$ is defined as in \eqref{eq:def.Aa.S} with $\Dd = \Dd_1 \oplus \Dd_S$ and the second Schur complement is given by
	\begin{equation}
		S = D - C A^\ddagger B \in \Bb (\Dd_S, \Dd_{-2}).
	\end{equation}
	With domains analogous to the ones in Definition~\ref{def:matrix.Schur.compl}, $\Aa_0$ and $S_0$, respectively, are the corresponding maximal operators in $\Hh$ and $\Hh_2$. \hfill //
\end{rem}

\subsection{Dense domain and point spectrum}

We start our ana\-ly\-sis of the relation between $\Aa_0$ and $S_0$ by providing sufficient conditions for the matrix $\Aa_0$ to be densely defined in $\Hh$.

\begin{prop}
	\label{prop:Aa.dom.dense}
	Let Assumption~{\rm\ref{asm:I}} be satisfied.
	
	\begin{enumerate}
		\item Let $\coker D_0 = \{0\}$, let $\dom S_0$ be dense in $\Hh_1$ and assume with
		\begin{equation}
			\label{eq:B.dom}
			\dom B_0 \defeq \set{f \in \Dd_2}{B f \in \Hh_1}
		\end{equation}
		that $\dom B_0 \cap \dom D_0$ is dense in $\Hh_2$. Then $\dom \Aa_0$ is dense in $\Hh$.
		
		\item Let $\Dd_{-S} = \Dd_{-1}$, let $0\in\rho(D_0) \cap \rho(S_0)$ and let $\dom S_0$ be dense in $\Dd_S$. Moreover, assume there exists $z \in \C$ such that $\ran (S_0 - z)$ is closed in $\Hh_1$ with $\coker (S_0 - z) = \{0\}$ and such that there exists an extension
		\begin{equation}\label{eq:S.gen.inv.shift.ext}
			(S_0 - z)^\# \subset S^\ddagger_z \in \Bb (\Dd_{-S}, \Dd_S).
		\end{equation}
		Then $\dom \Aa_0$ is dense in both $\Dd$ and $\Hh$.
	\end{enumerate} 
\end{prop}

If $D_0$ is boundedly invertible, the kernels and thus point spectra of $\Aa_0$ and $S_0$ are related as follows.

\begin{prop}
	\label{prop:point.spectrum}
	Let Assumption~{\rm\ref{asm:I}} be satisfied and $0 \in \rho (D_0)$. Then
	\begin{equation}
		0 \in \pspec (\Aa_0) \iff 0 \in \pspec (S_0)
	\end{equation}
	and the following identities hold
	\begin{equation}
		\ker \Aa_0 = \set{(f,-D^{-1} Cf)}{f \in \ker S_0}, \qquad \dim \ker \Aa_0 = \dim \ker S_0.
	\end{equation}
\end{prop}

\subsection{Semi-Fredholmness and bounded in\-ver\-ti\-bi\-li\-ty}

We proceed by establishing a relation between bounded in\-ver\-ti\-bi\-li\-ty/semi-Fred\-holm\-ness of the matrix and bounded in\-ver\-ti\-bi\-li\-ty/semi-Fred\-holm\-ness of its Schur complement. This in turn provides a connection between the (essential) spectra of $\Aa_0$ and $S_0 (\cdot)$ as an operator family, see Section~\ref{sec:matrix.Schur.complement}. Proofs of the following results can be found in Section~\ref{sec:proofs}.

\begin{thm}
	\label{thm:first.implication}
	Let Assumption~{\rm\ref{asm:I}} be satisfied.
	\begin{enumerate}
		\item Let $0 \in \rho (D_0)$. Then
			\begin{equation}
				0 \in \rho (\Aa_0) \implies 0 \in \rho (S_0).
			\end{equation}
		\item Let $S_0 \in \Cc (\Hh_1)$ and assume either that $\coker D_0 = \{0\}$ or that $D_0 \in \Ff_- (\Hh_2)$ with $\coker D_0 \subset \dom C_0^*$ and $\dom C_0$ is dense in $\Dd_S$, where
		\begin{equation}
			C_0 \defeq C\vert_{\dom C_0}, \qquad \dom C_0 \defeq \set{f \in \Dd_S}{Cf \in \Hh_2}.
		\end{equation}
		Then
		\begin{equation}
			\Aa_0 \in \Ff_+ (\Hh) \implies S_0 \in \Ff_+ (\Hh_1).
		\end{equation}
	\end{enumerate}
\end{thm}

We establish the reverse implication, i.e.\ bounded invertibility/semi-Fred\-holm\-ness of $S_0$ implying bounded invertibility/semi-Fred\-holm\-ness of $\Aa_0$.

\begin{thm}
	\label{thm:second.implication}
	Let Assumption~{\rm\ref{asm:I}} be satisfied and assume that $S \in \Bb (\Dd_S, \Dd_{-S})$ and $B (\dom D_0) \subset \Dd_{-S}$.
	\begin{enumerate}
		\item Let  $0\in\rho(D_0)$. If $S^{-1} \in \Bb (\Dd_{-S}, \Dd_S)$ {\rm(}which by the continuity of the embeddings $\Dd_S \hookrightarrow \Hh_1 \hookrightarrow \Dd_{-S}$ implies that $0 \in \rho(S_0)${\rm)}, then $0 \in \rho(\Aa_0)$.
		\item Let $\dom S_0$ be dense in $\Dd_S$, let $D_0 \in \Ff_+ (\Hh_2)$ and $S_0 \in \Ff_+ (\Hh_1)$ such that there exists an extension $S^\ddagger \supset S^\#_0$ with $S^\ddagger \in \Bb(\Dd_{-S}, \Dd_S)$. Then $\Aa_0 \in \Ff_+ (\Hh)$.
	\end{enumerate}
\end{thm}

Finally, we provide a sufficient condition for the existence of the extension $S^\ddagger$ above; a corollary of Theorem~\ref{thm:second.implication} then reads as follows.

\begin{cor}
	\label{cor:second.implication}
	Let Assumption~{\rm\ref{asm:I}} be satisfied, let $S \in \Bb (\Dd_S, \Dd_{-S})$, let $\dom S_0$ be dense in $\Dd_S$ and $B (\dom D_0) \subset \Dd_{-S}$. Moreover, assume there exists $z \in \C$ such that $\ran (S_0 - z)$ is closed in $\Hh_1$ with $\coker (S_0 - z) = \{0\}$ and such that there exists an extension $S^\ddagger_z$ as in \eqref{eq:S.gen.inv.shift.ext}.
	\begin{enumerate}
		\item Let $0\in\rho(D_0)$. Then
			\begin{equation}
				0 \in \rho(S_0) \implies 0 \in \rho(\Aa_0).
			\end{equation}
		\item Let $D_0 \in \Ff_+ (\Hh_2)$. Then
		\begin{equation}
			S_0 \in \Ff_+ (\Hh_1) \implies \Aa_0 \in \Ff_+ (\Hh).
		\end{equation}
	\end{enumerate}
\end{cor}

\begin{rem}
	The claims of both Proposition~\ref{prop:Aa.dom.dense}~(ii) and Corollary~\ref{cor:second.implication} remain true for finite dimensional $\coker (S_0 - z) \neq \{0\}$ if the orthogonal projection on $\coker (S_0 - z)$ in $\Hh_1$ has an extension in $\Bb (\Dd_{-S}, \Hh_1)$. Analogously, Proposition~\ref{prop:Aa.dom.dense}~(i) still holds if $\coker D_0 \neq \{0\}$ is finite dimensional and the orhtogonal projection on $\coker D_0$ in $\Hh_2$ has a bounded extension in $\Bb (\Dd_{-2},\Hh_2)$. \hfill //
\end{rem}

\begin{rem}\label{rem:diag.dom}
		Our results generalise the spectral equivalence in the diagonally dominant case established in the classical work~\cite{Nagel-1989}. More precisely, this concerns linear operator matrices
		\begin{equation}
			\Aa = \begin{pmatrix}
				A & B \\
				C & D
			\end{pmatrix} : \Hh \supset \dom A \times \dom D \to \Hh
		\end{equation}
	 	where $A$ and $D$, respectively, are densely defined with non-empty resolvent set in $\Hh_1$ and $\Hh_2$, and where $C$ and $B$, respectively, are linear operators (from $\Hh_1$ to $\Hh_2$ and vice versa) which are relatively bounded with respect to $A$ and $D$. The equivalence between~\cite[Thm.~2.4]{Nagel-1989}~(a) and~(b) can then be recovered from Theorems~\ref{thm:first.implication}~(i) and~\ref{thm:second.implication}~(i) applied to the matrix $\Aa - \la = \Aa_0 - \la$ and its Schur complement $S(\la) = S_0 (\la)$ in a straightforward way with 
	 	\begin{equation}
	 		\Dd_S := \dom A, \quad \Dd_{-S} := \Dd_{-1} := \Hh_1, \quad \Dd_2 := \dom D, \quad \Dd_{-2} := \Hh_2.
	 	\end{equation}
	 	Here the Hilbert spaces $\Dd_S$ and $\Dd_2$ shall be equipped with the scalar products associated to the graph norms of $A$ and $D$, respectively. Notice that in this particular case, we have
	 	\begin{equation}
	 		0 \in \rho (S(\la)) \implies S(\la)^{-1} \in \Bb (\Hh_1, \dom A),
	 	\end{equation}
 		cf.~Lemma~\ref{lem:closed.graph}. The equivalence between~\cite[Thm.~2.4]{Nagel-1989}~(a) and~(c) can be seen analogously using the second Schur complement.
	 	
	 	We mention that, while in~\cite{Nagel-1989} only spectral parameters in $\rho(A) \cap \rho(D)$ are considered, this extends the result to $\rho(D)$ or $\rho(A)$, respectively, when employing the first or second Schur complement; cf.~Corollaries~\ref{cor:family.first.implication}~(i) and~\ref{cor:family.second.implication}~(i) below.
	 	\hfill //
\end{rem}

\subsection{Technical lemmas}

We start by stating and proving some auxiliary results.
\begin{lem}
	\label{lem:closed.graph}
	Let $\Gg_1$, $\Gg_2$ and $\Gg_3$ be Hilbert spaces such that $\Gg_2 \subset \Gg_3$ is continuously embedded. If $T \in \Bb (\Gg_1, \Gg_3)$ and $\ran T \subset \Gg_2$, then $T \in \Bb (\Gg_1, \Gg_2)$.
\end{lem}

\begin{proof}
	Since $T \in \Cc (\Gg_1, \Gg_3)$ and $\Gg_2$ is continuously embedded in $\Gg_3$, we also have $T \in \Cc (\Gg_1, \Gg_2)$. The claim now follows from the closed graph theorem.
\end{proof}

\begin{lem}\label{lem:T.inverse}
	Let $\Gg_1$, $\Gg$ and $\Gg_{-1}$ be Hilbert spaces such that
	\begin{equation}
		\Gg_1 \subset \Gg \subset \Gg_{-1}
	\end{equation}
	 and the corresponding canonical embeddings are continuous. Let $T_0 \in \Cc (\Gg)$ have closed range in $\Gg$, let $\dom T_0 \subset \Gg_1$ and assume there exist extensions
	\begin{equation}
		T_0 \subset T \in \Bb (\Gg_1, \Gg_{-1}), \qquad T_0^\# \subset T^\ddagger \in \Bb (\Gg_{-1}, \Gg_1).
	\end{equation}
	Let $P$ and $Q$, respectively, denote the orthogonal projections on $\ker T_0$ and $\coker T_0$ in $\Gg$. Then
	\begin{equation}
		TT^\ddagger\vert_\Gg = I - Q
	\end{equation}
	and, if $\Gg$ is dense in $\Gg_{-1}$, then $\coker T_0 = \{0\}$ implies $TT^\ddagger = I$. If $\dom T_0$ is dense $\Gg_1$, then
	\begin{equation}
		T^\ddagger T = I - P\vert_{\Gg_1}
	\end{equation}
	and $\ker T_0 = \{0\}$ implies $T^\ddagger T = I$. In particular, if $0 \in \rho(T_0)$, $\Gg$ is dense in $\Gg_{-1}$ and $\dom T_0$ is dense in $\Gg_1$, then $T^\ddagger = T^{-1}$. 
\end{lem}

\begin{proof}
	Using relation \eqref{eq:gen.inv.comp} and since $\ran T_0^\# \subset \dom T_0$, we obtain
		\begin{equation}
			TT^\ddagger\vert_\Gg = T_0 T_0^\# = I - Q.
		\end{equation}
	If $\Gg$ is dense in $\Gg_{-1}$ and $\coker T_0 = \{0\}$, then $TT^\ddagger = I$ follows from the above identity and from $TT^\ddagger \in \Bb (\Gg_{-1})$.
	
	Let $\dom T_0$ be dense in $\Gg_1$. From Lemma~\ref{lem:closed.graph}, we obtain $P \in \Bb (\Gg, \Gg_1)$. Consequently, again using \eqref{eq:gen.inv.comp}, we see that $T^\ddagger T \in \Bb (\Gg_1)$ and $I - P\vert_{\Gg_1} \in \Bb (\Gg_1)$ satisfy
	\begin{equation}
		\label{eq:T.gen.inv}
		T^\ddagger T \vert_{\dom T_0} = T_0^\# T_0 = I - P\vert_{\dom T_0}.
	\end{equation}
	Hence, $T^\ddagger T = I - P\vert_{\Gg_1}$ is a consequence of the density of $\dom T_0$ in $\Gg_1$ and clearly, $T^\ddagger T = I$ holds if $\ker T_0 = \{0\}$. The last claim is now immediate.
\end{proof}

\begin{lem}
	\label{lem:D.ker.ran}
	If Assumption~{\rm\ref{asm:I}} holds, then $\ker D = \ker D_0$ and $\ran D^\ddagger \perp_{\Hh_2} \ker D_0$.
\end{lem}

\begin{proof}
	We first show $\ker D = \ker D_0$. Let $f \in \ker D$. Since $\dom D_0$ is dense in $\Dd_2$, there exists a sequence $\{f_m\}_m \subset \dom D_0$ such that $f_m \to f$ in $\Dd_2$ as $m \to \infty$. Consequently, from $D_0^\# D_0 \subset D^\ddagger D \in \Bb (\Dd_2)$ and $Df = 0$, with $g_m \defeq (I - D_0^\# D_0) f_m$ we conclude
	\begin{equation}
		\norm{g_m - f}_{\Hh_2} \lesssim \norm{g_m - f}_{\Dd_2} \le \norm{f_m - f}_{\Dd_2} + \norm{D^\ddagger D f_m}_{\Dd_2} \to 0, \quad m \to \infty.
	\end{equation}
	Note that $I - D_0^\# D_0$ is the orthogonal projection on $\ker D_0$ in $\Hh_2$, see \eqref{eq:gen.inv.comp}. Hence, the closedness of $D_0$ and $g_m \in \ker D_0$, $m \in \N$, imply $f \in \dom D_0$ and $D_0 f = 0$.
	It remains to show $\ran D^\ddagger \perp_{\Hh_2} \ker D_0$. However, since $\Hh_2$ is dense in $\Dd_{-2}$ and $D^\ddagger : \Dd_{-2} \to \Dd_2$ is continuous, we have
	\begin{equation}\label{eq:lem.D.dagger.ran}
		\ran D^\ddagger = D^\ddagger \left(\overline{\Hh_2}^{\Dd_{-2}}\right) \subset \overline{D^\ddagger (\Hh_2)}^{\Dd_2}.
	\end{equation}
	Moreover, from $\Dd_2$ being continuously embedded in $\Hh_2$ and \eqref{eq:gen.inv}, we conclude
	\begin{equation}\label{eq:lem.D0.ker.ran}
		\overline{\ran D_0^\#}^{\Dd_2} \subset \overline{\ran D_0^\#}^{\Hh_2} \perp_{\Hh_2} \ker D_0.
	\end{equation}
	Considering $D^\ddagger \vert_{\Hh_2} = D_0^\#$, the claim now follows from \eqref{eq:lem.D.dagger.ran} and \eqref{eq:lem.D0.ker.ran}.
\end{proof}

\subsection{Proofs} \label{sec:proofs}

We start by proving Proposition~\ref{prop:Aa.dom.dense} (i); for technical reasons, the proof of (ii) is given at the end of this section.

\begin{proof}[Proof of Proposition~{\rm\ref{prop:Aa.dom.dense} (i)}]
	By Lemma~\ref{lem:T.inverse}, we have that $DD^\ddagger = I$. Moreover, it is straightforward to check the following inclusion
	\begin{equation}
		\label{eq:subdom.Aa}
		\set{(f, h - D^\ddagger Cf)}{f \in \dom S_0, \, h \in \dom B_0 \cap \dom D_0} \subset \dom \Aa_0.
	\end{equation}
	Let $(u, v) \in \Hh$ with $(u, v) \perp_\Hh \dom \Aa_0$; we need to show $u = v = 0$. By \eqref{eq:subdom.Aa},
	\begin{equation}
		\label{eq:subdom.Aa.ortho}
		\iprod{f}{u}_{\Hh_1} + \iprod{h - D^\ddagger Cf}{v}_{\Hh_2} = 0, \quad f \in \dom S_0, \quad h \in \dom B_0 \cap \dom D_0.
	\end{equation}
	Setting $f = 0$ in \eqref{eq:subdom.Aa.ortho}, from the density of $\dom B_0 \cap \dom D_0$ in $\Hh_2$ it follows that $v = 0$. Hence,~\eqref{eq:subdom.Aa.ortho} implies $u \perp \dom S_0$ and we obtain $u = 0$ from the density of $\dom S_0$ in $\Hh_1$.
\end{proof}

\begin{proof}[Proof of Proposition~{\rm\ref{prop:point.spectrum}}]
	By Lemma~\ref{lem:T.inverse}, $0 \in \rho (D_0)$ implies $D^\ddagger = D^{-1}$. Let $(f,g) \in \ker \Aa_0$, i.e.\ 
	\begin{equation}
		Af + Bg = 0, \qquad Cf + Dg = 0,
	\end{equation}
	where $f \in \Dd_S$ and $g\in \Dd_2$. Applying $D^{-1}$ to the second equation and inserting it in the first one, we obtain
	\begin{equation}
		g = -D^{-1} Cf, \qquad 0 = Af - B D^{-1} Cf = Sf = S_0 f.
	\end{equation}
	Conversely, $f \in \ker S_0$ and $g = -D^{-1} Cf$ clearly imply $0 = \Aa (f, g) = \Aa_0 (f, g)$. The equality of dimensions now follows immediately.
\end{proof}

\begin{proof}[Proof of Theorem {\rm~\ref{thm:first.implication}}]
	(i) We show that $S_0^{-1} = L \defeq \pi_1 \Aa_0^{-1} \pi_1^* \in \Bb(\Hh_1)$, thus implying $0 \in \rho(S_0)$; recall that $\pi_1$ denotes the canonical projection from $\Hh$ onto~$\Hh_1$. Let therefore $f \in \dom S_0$. Since $D^\ddagger = D^{-1}$ by $0 \in \rho(D_0)$ and Lemma~\ref{lem:T.inverse}, we have $(f, -D^{-1} Cf) \in \dom \Aa_0$ and
	\begin{equation}
		\Aa_0 (f, -D^{-1} Cf) = (S_0f, 0) = \pi_1^* S_0 f.
	\end{equation}
	By applying $\pi_1 \Aa_0^{-1}$ to the above identity, one obtains $f = LS_0 f$.
	Conversely, let $f \in \Hh_1$ be arbitrary. Then
	\begin{equation}
	\label{eq:S.inverse}
		(f, 0) = \Aa_0 \Aa_0^{-1} \pi_1^* f = (ALf + B \pi_2 \Aa_0^{-1} \pi_1^* f, CLf + D \pi_2 \Aa_0^{-1} \pi_1^* f).
	\end{equation}
	Applying $D^{-1}$ to the second component of \eqref{eq:S.inverse}, we obtain
	\begin{equation}
		\pi_2 \Aa_0^{-1} \pi_1^* f = -D^{-1} CLf.
	\end{equation}
	We insert this identity in the first component of \eqref{eq:S.inverse} to conclude
	\begin{equation}
	\Hh_1 \ni f = ALf - B D^{-1} CLf = SLf,
	\end{equation}
	which in turn implies that $Lf \in \dom S_0$ and $S_0 Lf = f$.

	(ii) Let us proceed by contraposition, i.e.\ let us show that if $S_0 \notin \Ff_+ (\Hh_1)$, then $\Aa_0 \notin \Ff_+ (\Hh)$. Assume that $\Aa_0$ is closed, since otherwise $\Aa_0 \notin \Ff_+ (\Hh)$ by definition. Since $S_0$ is closed, there exists a sequence $\{f_m\}_m \subset \dom S_0$ with $\norm{f_m}_{\Hh_1} = 1$, $m \in \N$, such that
	\begin{equation}
	\label{eq:S.sing.sequ}
		S_0 f_m \overset{\Hh_1}{\longrightarrow} 0, \qquad f_m \overset{w}{\longrightarrow} 0 ~ ~ {\rm in} ~ ~ \Hh_1, \qquad m \to \infty,
	\end{equation}
	see \cite[Thm.\ IX.1.3 (i)]{Edmunds-Evans-1987} and the preceding paragraph (note that the density of the domain is only needed for part (ii) of the Theorem).
	We construct a singular sequence for $\Aa_0$ and conclude $\Aa_0 \notin \Ff_+ (\Hh)$ by \cite[Thm.\ IX.1.3 (i)]{Edmunds-Evans-1987}. Define the sequence $\{\mathbf{x}_m \defeq (f_m, -D^\ddagger Cf_m)\}_m \subset \Dd$, then
	\begin{equation} \label{eq:Aa.sing.sequ.cand}
		\norm{\mathbf{x}_m}_\Hh \ge \norm{f_m}_{\Hh_1} = 1, \qquad \Aa \mathbf{x}_m = (S_0 f_m, (I - D D^\ddagger) Cf_m), \qquad m \in \N.
	\end{equation}
	We show that $\mathbf x_m \in \dom \Aa_0$ and $\Aa_0 \mathbf{x}_m \to 0$ in $\Hh$ as $m \to \infty$. If $\coker D_0 = \{0\}$, then Lemma~\ref{lem:T.inverse} implies $D D^\ddagger = I$ and consequently,
	\begin{equation}
	\Aa \mathbf{x}_m = (S_0 f_m, 0) = \Aa_0 \mathbf x_m \overset{\Hh}{\longrightarrow} 0, \qquad m \to \infty.
	\end{equation}
	Consider the case that
	\begin{equation}
		\{0\} \neq \coker D_0 = \ls \{\phi_j\}_{1 \le j \le k} \subset \dom C_0^*;
	\end{equation}
	recall that $\dim \coker D_0 < \infty$ as $D_0 \in \Ff_-(\Hh_2)$.
	Then, using that $Cf \in \Hh_2$ for $f \in \dom C_0 \subset \Dd_S$, that $DD^\ddagger \vert_{\Hh_2} = D_0 D_0^\#$ and relation \eqref{eq:gen.inv.comp}, we obtain the identity
	\begin{equation}
	\label{eq:QC.formula}
		(I - D D^\ddagger) Cf = \sum_{j=1}^{k} \iprod{Cf}{\phi_j}_{\Hh_2} \phi_j = \sum_{j=1}^{k} \iprod{f}{C_0^*\phi_j}_{\Hh_1} \phi_j, \quad f \in \dom C_0.
	\end{equation}
	By density of $\dom C_0$ in $\Dd_S$ and since both left and right hand side of \eqref{eq:QC.formula} are continuous in $f$ with respect to $\norm{\cdot}_S$ (with values in $\Dd_{-2}$), formula \eqref{eq:QC.formula} is valid for all $f \in \Dd_S$. We can thus extend
	\begin{equation}
		(I - D D^\ddagger) C \subset K \defeq \sum_{j=1}^{k} \iprod{\cdot}{C_0^*\phi_j}_{\Hh_1} \phi_j \in \Kk (\Hh_1, \Hh_2),
	\end{equation}
	hence from \eqref{eq:Aa.sing.sequ.cand} we see that $\mathbf x_m \in \dom \Aa_0$ and \eqref{eq:S.sing.sequ} implies that $\Aa_0 \mathbf{x}_m = (S_0 f_m, Kf_m) \to 0$ strongly in $\Hh$ as $m \to \infty$. Since $\norm{\mathbf{x}_m}_\Hh \ge 1$, in both cases $\coker D_0 = \{0\}$ and $\coker D_0 \neq \{0\}$, we obtain an $\Hh$-normalised sequence
	\begin{equation}
	\label{eq:Aa.sing.sequ}
		\widetilde{\mathbf{x}}_m \defeq \frac{\mathbf{x}_m}{\norm{\mathbf{x}_m}_\Hh} \in \dom \Aa_0, \qquad \Aa_0 \widetilde{\mathbf{x}}_m \overset{\Hh}{\to} 0, \qquad m \to \infty.
	\end{equation}
	It remains to show that $\{\widetilde{\mathbf{x}}_m\}_m$ has no convergent subsequence in $\Hh$. We assume the opposite, i.e.\ that there exists $(f, g) \in \Hh$ and a subsequence (again denoted by $\{\widetilde{\mathbf{x}}_m\}_m$) such that $\widetilde{\mathbf{x}}_m \to (f, g)$ in $\Hh$ as $m \to \infty$. Consequently, from $\norm{\mathbf{x}_m}_\Hh \ge 1$ and~\eqref{eq:S.sing.sequ} it follows for arbitrary $u \in \Hh_1$ that
	\begin{equation}
		\abs{\iprod{f}{u}_{\Hh_1}} = \lim_{m \to \infty} \frac{\abs{\iprod{f_m}{u}_{\Hh_1}}}{\norm{\mathbf{x}_m}_\Hh} \le \lim_{m \to \infty} \abs{\iprod{f_m}{u}_{\Hh_1}} = 0,
	\end{equation}
	thus $f = 0$. Since $\Aa_0$ is closed, \eqref{eq:Aa.sing.sequ} and $\widetilde{\mathbf{x}}_m \to (0, g)$ imply
	\begin{equation}
		(0, g) \in \dom \Aa_0, \qquad \Aa_0 (0, g) = (Bg, Dg) = 0,
	\end{equation}
	in particular $g \in \ker D = \ker D_0$ by Lemma~\ref{lem:D.ker.ran}. Since by construction $\pi_2 \widetilde{\mathbf{x}}_m \in \ran D^\ddagger$, Lemma~\ref{lem:D.ker.ran} ultimately leads to the contradiction
	\begin{equation*}
		1 = \lim_{m \to \infty} \norm{\widetilde{\mathbf{x}}_m}_\Hh^2 = \norm{g}_{\Hh_2}^2 = \lim_{m \to \infty} \iprod{\pi_2 \widetilde{\mathbf{x}}_m}{g}_{\Hh_2} = 0. \qedhere
	\end{equation*}
\end{proof}

The proof of Theorem~\ref{thm:second.implication} is based on constructing a (left approximate) inverse for $\Aa_0$. Its main ingredient is the following proposition.

\begin{prop}
	\label{prop:left.inverse}
	Let the assumptions of Theorem {\rm\ref{thm:second.implication} (ii)} be satisfied. Then
	\begin{equation}\label{eq:L.left.inverse}
		\Ll \defeq \left(
		\begin{array}{cc}
			S_0^\# & -S^\ddagger B D_0^\# \\
			-D^\ddagger C S_0^\# & D_0^\# + D^\ddagger C S^\ddagger B D_0^\#
		\end{array}
		\right)
		\in \Bb(\Hh, \Dd)
	\end{equation}
	satisfies $\Ll \Aa_0 = I + \Kk\vert_{\dom\Aa_0}$, where
	\begin{equation}
	\label{eq:K.left.inverse}
		\Kk \defeq \left(
		\begin{array}{cc}
			-P_S & S^\ddagger B P_D \\
			D^\ddagger C P_S & -P_D - D^\ddagger C S^\ddagger B P_D
		\end{array}
		\right)
		\in \Kk(\Hh, \Dd)
	\end{equation}
	has finite rank. Here $P_S$ and $P_D$ denote the orthogonal projections on $\ker S_0$ in $\Hh_1$ and on $\ker D_0$ in $\Hh_2$, respectively.
\end{prop}

\begin{proof}
	From Lemma~\ref{lem:closed.graph}, it follows that $P_S \in \Bb (\Hh_1, \Dd_S)$ and $P_D \in \Bb (\Hh_2, \Dd_2)$. Moreover, $B (\dom D_0) \subset \Dd_{-S}$ implies $\ran B P_D \subset \Dd_{-S}$ and again by Lemma~\ref{lem:closed.graph} we obtain $B P_D \in \Bb (\Hh_2, \Dd_{-S})$. Altogether, we have $\Kk \in \Bb (\Hh, \Dd)$. Since $\ker S_0$ and $\ker D_0$ are finite dimensional, $\Kk$ has finite rank and is thus compact.
	We next show that $\Ll$ is well defined and in $\Bb(\Hh,\Dd)$. Notice therefore that
	\begin{equation}
		D_0^\# = D^\ddagger\vert_{\Hh_2} \in \Bb (\Hh_2, \Dd_2), \qquad S_0^\# = S^\ddagger\vert_{\Hh_1} \in \Bb (\Hh_1, \Dd_S).
	\end{equation}
	Hence, from $B (\dom D_0) \subset \Dd_{-S}$, we conclude that $\ran BD_0^\# \subset \Dd_{-S}$ and, applying Lemma~\ref{lem:closed.graph}, that $BD_0^\# \in \Bb (\Hh_2, \Dd_{-S})$. This shows the claimed boundedness.
	It remains to prove that
	\begin{equation}
		\Ll \Aa_0 (f, g) = (f, g) + \Kk (f, g), \qquad (f, g) \in \dom \Aa_0.
	\end{equation}
	Since $S_0^\# \subset S^\ddagger$ and $D_0^\# \subset D^\ddagger : \Dd_{-2} \to \Dd_2$, and since $Cf, Dg \in \Dd_{-2}$, we can write
	\begin{equation}
	\label{eq:left.approx.inverse.proof}
		\pi_1 \Ll \Aa_0 (f, g) = S^\ddagger (Af + Bg - B D^\ddagger (Cf + Dg)) = S^\ddagger (Sf + Bg - B D^\ddagger Dg).
	\end{equation}
	Applying Lemma~\ref{lem:T.inverse} to $S_0$ and $D_0$, we obtain the identities
	\begin{equation}
		S^\ddagger S = I - P_S\vert_{\Dd_S}, \qquad D^\ddagger D = I - P_D\vert_{\Dd_2}.
	\end{equation}
	Using this,~\eqref{eq:left.approx.inverse.proof}, $S^\ddagger : \Dd_{-S} \to \Dd_S$ and  $Sf, BP_D g \in \Dd_{-S}$ gives
	\begin{equation}
		\pi_1 \Ll \Aa_0 (f, g) = S^\ddagger (Sf + B P_D g) = f - P_S f + S^\ddagger B P_D g = f + \pi_1 \Kk (f, g).
	\end{equation}
	The proof of $\pi_2 \Ll \Aa_0 (f, g) = g + \pi_2 \Kk (f, g)$ is analogous.
\end{proof}

Employing the above proposition, we are now able to prove Theorem~\ref{thm:second.implication}.

\begin{proof}[Proof of Theorem {\rm\ref{thm:second.implication}}]
	(i) We apply Proposition~\ref{prop:left.inverse}. Note that all the assumptions are satisfied with $S^\ddagger = S^{-1}$; the density of $\dom S_0$ in $\Dd_S$ follows from the density of $\Hh_1$ in $\Dd_{-S}$ and $S^{-1} \in \Bb (\Dd_{-S}, \Dd_S)$.
	From $0 \in \rho(D_0)$, we conclude $D_0^\# = D_0^{-1}$, $P_D = 0$ and $D^\ddagger = D^{-1}$ by Lemma~\ref{lem:T.inverse}. Moreover, Lemma~\ref{lem:T.inverse} gives $S_0^\# = S_0^{-1}$ and $P_S = 0$ since $0 \in \rho (S_0)$. Hence $\Kk = 0$, i.e.\ the operator matrix
	\begin{equation}
		\Ll = \left(
		\begin{array}{cc}
			S_0^{-1} & -S^{-1} B D_0^{-1} \\
			-D^{-1} C S_0^{-1} & D_0^{-1} + D^{-1} C S^{-1} B D_0^{-1}
		\end{array}
		\right)
	\end{equation}
	from Proposition~\ref{prop:left.inverse} is a left inverse for $\Aa_0$. We show that $\Ll  \in \Bb (\Hh, \Dd)$ is also a right inverse for $\Aa_0$; as $\Dd \subset \Hh$ is continuously embedded, $0 \in \rho(\Aa_0)$ then follows from
		\begin{equation}
			\Aa_0^{-1} = \Ll \in \Bb (\Hh).
		\end{equation}
	
	Let $(f, g) \in \Hh$, then we have  $\Ll (f, g) \in \Dd$ and, using $D_0^{-1} \subset D^{-1}$ and $S_0^{-1} \subset S^{-1}$, we can write
	\begin{equation}
	\begin{aligned}
		\pi_1 \Aa \Ll (f, g) & = A (S^{-1}f -S^{-1} B D^{-1} g) \\
		& \qquad \qquad \qquad  + B (-D^{-1} C S^{-1}f + D^{-1} g + D^{-1} C S^{-1} B D^{-1}g). \\
	\end{aligned}
	\end{equation}
	Since $A :\Dd_S \to \Dd_{-1}$ and $B : \Dd_2 \to \Dd_{-1}$, we can break the parentheses in the above identity and obtain
	\begin{equation}
		\pi_1 \Aa \Ll (f, g) = (A - B D^{-1} C) S^{-1}f + ( I - A S^{-1} + B D^{-1} C S^{-1} ) B D^{-1} g
	\end{equation}
	and further conclude
	\begin{equation}
		\pi_1 \Aa \Ll (f, g) = S S^{-1} f + (I - S S^{-1})  B D^{-1} g = f.
	\end{equation}
	In the same way, one shows that $\pi_2 \Aa \Ll (f, g) = g$. Hence, $\Aa \Ll (f, g) = (f, g)$, which implies $\ran \Ll \subset \dom \Aa_0$ and $\Aa_0 \Ll = I$.
	(ii) By Proposition~\ref{prop:left.inverse}, the identity
	\begin{equation}
	\label{eq:left.approx.inverse}
		\Ll \Aa_0 = I + \Kk\vert_{\dom\Aa_0}
	\end{equation}
	holds true with $\Ll \in \Bb (\Hh, \Dd)$ and $\Kk \in \Bb (\Hh, \Dd)$ as in \eqref{eq:L.left.inverse} and \eqref{eq:K.left.inverse}. We first show that $\Aa_0$ is closed in $\Hh$. Let therefore $\{(f_m, g_m)\}_m \subset \dom \Aa_0$ and $(f, g), (u, v) \in \Hh$ such that $(f_m, g_m) \to (f, g)$ and $\Aa_0 (f_m, g_m) \to (u, v)$ in $\Hh$ as $m \to \infty$. The continuity of $\Ll$ and $\Kk$ imply
	\begin{equation}
	\label{eq:Aa.closed}
		(f_m, g_m) = \Ll \Aa_0 (f_m, g_m) - \Kk (f_m, g_m) \overset{\Dd}{\longrightarrow} \Ll (u, v) - \Kk (f, g), \quad m \to \infty,
	\end{equation}
	and so the sequence $\{(f_m, g_m)\}_m$ is convergent in both $\Hh$ and $\Dd$. Since $\Dd \subset \Hh$ is con\-ti\-nu\-ous\-ly embedded, the limits must coincide, thus $(f_m, g_m) \to (f,g)$ in $\Dd$ as $m \to \infty$. Consequently, $\Aa \in \Bb (\Dd, \Dd_-)$ implies
	\begin{equation}
		\Aa_0 (f_m, g_m) = \Aa (f_m, g_m) \overset{\Dd_-}{\longrightarrow} \Aa (f, g), \quad m \to \infty.
	\end{equation}
	Hence, $\{\Aa_0 (f_m,g_m)\}_m$ is convergent in both $\Hh$ and $\Dd_-$, which, since $\Hh\subset\Dd_-$ is continuously embedded, gives the equality $\Aa (f, g) = (u, v) \in \Hh$. This in turn implies $(f, g) \in \dom \Aa_0$ and $\Aa_0 (f, g) = (u, v)$.
	By the continuity of the embedding $\Dd \subset \Hh$, we have $\Ll, \Kk \in \Bb (\Hh)$. Moreover, it follows from Proposition~\ref{prop:left.inverse} that $\Kk$ has finite rank, thus $\Kk \in \Kk (\Hh)$ and from \eqref{eq:left.approx.inverse} we see that $\Ll$ is a bounded left approximate inverse for $\Aa_0$. Finally, from the closedness of $\Aa_0$, \eqref{eq:left.approx.inverse} and \cite[Thm.\ I.3.13]{Edmunds-Evans-1987}, we conclude $\Aa_0 \in \Ff_+ (\Hh)$.
\end{proof}

\begin{proof}[Proof of Corollary {\rm\ref{cor:second.implication}}]
	Let $S_0 \in \Ff_+ (\Hh_1)$; note that this in particular holds if $0 \in \rho(S_0)$. We start by deriving a resolvent type identity for $S_0^\#$ in order to construct an extension
	\begin{equation}
			S_0^\# \subset S^\ddagger \in \Bb (\Dd_{-S}, \Dd_S).
	\end{equation}
	Let $P$ denote the orthogonal projection on $\ker S_0$ in $\Hh_1$. From \eqref{eq:gen.inv.comp}, we conclude (notice that $\coker S_0 = \{0\}$ and thus $Q=0$)
	\begin{equation}
		(S_0 - z)^\# - S_0^\# = (S_0^\# S_0 + P) (S_0 - z)^\# - S_0^\# (S_0 - z) (S_0 - z)^\# = (z S_0^\# + P)(S_0 - z)^\#.
	\end{equation}
	We can thus define an extension
	\begin{equation}
		S_0^\# \subset S^\ddagger_z - (z S_0^\# + P) S^\ddagger_z \eqdef S^\ddagger.
	\end{equation}
	By Lemma~\ref{lem:closed.graph}, one has $zS_0^\# + P \in \Bb (\Hh_1, \Dd_S)$. Moreover, $S^\ddagger_z \in \Bb (\Dd_{-S}, \Hh_1)$ since $\Dd_S$ is continuously embedded in $\Hh_1$. Altogether, $S^\ddagger \in \Bb (\Dd_{-S}, \Dd_S)$ is shown.
	
	The statements (i) and (ii), respectively, now immediately follow from Theorem~\ref{thm:second.implication} (i) and (ii); notice that in (i) the assumption $0 \in \rho (S_0)$ together with Lemma~\ref{lem:T.inverse} implies $S^{-1} = S^\ddagger$.
\end{proof}

It remains to prove Proposition~\ref{prop:Aa.dom.dense} (ii).

\begin{proof}[Proof of Proposition {\rm \ref{prop:Aa.dom.dense} (ii)}]
	Notice first that the assumptions of Theorem~\ref{thm:second.implication}~(i) are satisfied, in particular $S^{-1} \in \Bb (\Dd_{-S},\Dd_S)$, see the proof of Corollary~\ref{cor:second.implication} (i). Moreover, since $\Dd_{-S} = \Dd_{-1}$, the inverse $\Ll$ in the proof of Theorem~\ref{thm:second.implication}~(i) has the following bounded extension
	\begin{equation}
		\widetilde \Ll \defeq \begin{pmatrix}
			S^{-1} & -S^{-1} B D^{-1} \\
			-D^{-1} C S^{-1} & D^{-1} + D^{-1} C S^{-1} B D^{-1}
		\end{pmatrix} \in \Bb (\Dd_-,\Dd).
	\end{equation}
	It is straightforward to check that $\widetilde \Ll = \Aa^{-1}$, which by the density of $\Hh$ in $\Dd_-$ immediately implies the density of $\dom \Aa_0$ in $\Dd$. The density of $\dom \Aa_0$ in $\Hh$ then follows from the density of $\Dd$ in $\Hh$ and the continuity of its canonical embedding.
\end{proof}

\section{Spectral equivalence between operator matrix and Schur complement}
\label{sec:matrix.Schur.complement}

We apply the results from Section~\ref{sec:abstract.results} in order to obtain spectral correspondence between the matrix $\Aa_0$ and its Schur complement $S_0 (\cdot)$ as an operator family. More precisely, we assume that $\Theta \subset \C$ is a set such that, for a fixed $\lambda \in \Theta$, Assumption~\ref{asm:I} is satisfied with $A - \lambda$, $D_0 - \lambda$ and $D - \lambda$ instead of $A$, $D_0$ and $D$, respectively. Our main theorems (which correspond to the case $\lambda = 0$) then apply to the matrix $\Aa - \lambda$ and its (generalised) Schur complement $S (\lambda)$, see \eqref{eq:Schur.complement.family}. This relates semi-Fredholmness/bounded invertibility of $\Aa_0-\lambda$ to semi-Fredholmness/bounded invertibility of $S_0 (\lambda)$ and thereby provides an equivalence of the type
\begin{equation}
	\lambda \in \sigma (\Aa_0) \iff 0 \in \sigma (S_0 (\lambda)) \iff \lambda \in \sigma (S_0 (\cdot)),
\end{equation}
and analogously with $\sigma_{\rm p}$ and $\sigma_{\operatorname{e2}}$ instead of $\sigma$. Recall that the spectrum of the operator family $\set{S_0 (\lambda)}{\lambda \in \Theta}$ is defined as
\begin{equation}
	\sigma (S_0 (\cdot)) = \set{\lambda \in \Theta}{0 \in \sigma (S_0 (\lambda))};
\end{equation}
its resolvent set, essential spectrum and point spectrum are defined analogously.

\begin{asm} \label{asm:II}
	Throughout this section, we assume the following.
	\begin{enumerate}
		\item Let Assumption~\ref{asm:I}~(i) and (ii) be satisfied.
		\item Let $D_0 \in \Cc (\Hh_2)$ such that $\dom D_0 \subset \Dd_2$ is dense in $\Dd_2$ and assume that there exists an extension
		\begin{equation}
			D_0 \subset D \in \Bb (\Dd_2, \Dd_{-2}). \tag*{//}
		\end{equation}
	\end{enumerate}
\end{asm}

The assumptions above allow us to introduce the operator matrix $\Aa$ and its maximal restriction $\Aa_0$ in $\Hh$ analogously to Definition~\ref{def:matrix.Schur.compl}. Moreover, we can define the Schur complement $S(\lambda)$ for a subset $\Theta$ of spectral parameters $\lambda \in \C$ such that $D_0 - \lambda$ and its extension $D - \lambda$ satisfy Assumption~\ref{asm:I}~(iii). Notice that, in line with Assumption~\ref{asm:I}~(iii), we do not require $\la \in \rho(D_0)$ but allow $\lambda \in \C$ such that $D_0 - \lambda$ merely has a generalised inverse, see \eqref{eq:gen.inv} and \eqref{eq:gen.inv.comp}.

\begin{defi}
	\label{def:matrix.Schur.compl.family}
	Let Assumption~\ref{asm:II} be satisfied and define $\Aa \in \Bb (\Dd, \Dd_-)$, $\dom \Aa_0$ and $\Aa_0$ as in Definition~\ref{def:matrix.Schur.compl}. Moreover, define the following family of operators
	\begin{equation}
	\label{eq:Schur.complement.family}
		S (\lambda) \defeq A - \lambda - B D^\ddagger (\lambda) C \in \Bb (\Dd_S, \Dd_{-1}), \qquad \lambda \in \Theta,
	\end{equation}
	where the set $\Theta \subset \C$ satisfies
	\begin{equation}\label{eq:Theta}
		\begin{aligned}
			\Theta & ~\subset~ \set{\lambda \in \C}{\ran (D_0 - \lambda) ~ \textnormal{closed in} ~ \Hh_2, ~ \\
			& \hspace{3cm} \exists ~ \textnormal{extension} ~ (D_0 - \lambda)^\# \subset D^\ddagger (\lambda) \in \Bb (\Dd_{-2}, \Dd_2)}.
		\end{aligned}
	\end{equation}
	Let the corresponding family of maximal operators in $\Hh_1$ be defined as
	\begin{equation}
		S_0 (\lambda) \defeq S (\lambda)\vert_{\dom S_0 (\lambda)}, \quad \dom S_0 (\lambda) \defeq \set{f \in \Dd_S}{S (\lambda) f \in \Hh_1}, \quad \lambda \in \Theta. \tag*{//}
	\end{equation}
\end{defi} 

\begin{rem}
	\label{rem:families.variable.domain}
	The spaces in Assumption~\ref{asm:II}~(i) are chosen such that, independently of $\lambda \in \Theta$, the results of Section~\ref{sec:abstract.results} apply to $\Aa - \lambda$ and $S(\lambda)$. However, if the spaces and operators in Assumption~\ref{asm:II} may depend on $\lambda \in \Theta$, our method defines an operator matrix valued family
	\begin{equation}
		\Ll(\lambda) = \begin{pmatrix}
			A(\lambda) & B(\lambda) \\
			C(\lambda) & D(\lambda)
		\end{pmatrix} \in \Bb (\Dd (\lambda), \Dd_- (\lambda)), \quad \lambda \in \Theta,
	\end{equation}
	whose domain $\dom \Ll_0 (\lambda) \subset \Hh$ in general depends on the spectral parameter. \hfill //
\end{rem}

Let us state the first of four corollaries, which are a translation of the results in Section~\ref{sec:abstract.results} to the present setting. Namely, if the assumptions of Proposition~\ref{prop:Aa.dom.dense} are satisfied in one point $\lambda_0 \in \Theta$, then $\Aa_0$ is densely defined in $\Hh$ and/or $\Dd$.

\begin{cor}
	\label{cor:family.dense.domain}
	Let Assumption~{\rm\ref{asm:II}} be satisfied and let $\Aa_0$, $S_0$ and $\Theta$ be as in Definition {\rm\ref{def:matrix.Schur.compl.family}}.
	
	\begin{enumerate}
		\item  Let $\dom B_0 \cap \dom D_0$ {\rm(}with $\dom B_0$ defined as in \eqref{eq:B.dom}{\rm)} be dense in $\Hh_2$. Assume there exists $\lambda_0 \in \Theta$ such that $\coker (D_0-\lambda_0) = \{0\}$ and such that $\dom S_0 (\lambda_0)$ is dense in $\Hh_1$. Then $\dom \Aa_0$ is dense in $\Hh$.
		
		\item Let $\Dd_{-S} = \Dd_{-1}$ and assume that there exists $\lambda_0 \in \Theta$ such that $D_0 - \lambda_0$ and $S_0 (\lambda_0)$ satisfy the assumptions of Proposition~{\rm\ref{prop:Aa.dom.dense} (ii)}. Then $\dom \Aa_0$ is dense in both $\Dd$ and $\Hh$.
	\end{enumerate}
\end{cor}

\begin{proof}
	The assumptions imply that Assumption~\ref{asm:I} is satisfied with $A - \lambda_0$, $D_0 - \lambda_0$ and $D - \lambda_0$ instead of $A$, $D_0$ and $D$, respectively. Hence, the claims in (i) and (ii) follow from Proposition~\ref{prop:Aa.dom.dense} (i) and (ii) applied to $\Aa_0 - \lambda_0$ and $S_0 (\lambda_0)$ and from
	\begin{equation*}
		\dom \Aa_0 = \dom (\Aa_0 - \lambda_0), \qquad \dom D_0 = \dom (D_0 - \lambda_0). \qedhere
	\end{equation*}
\end{proof}

The point spectra of $\Aa_0$ and $S_0 (\cdot)$ correspond on $\Theta \cap \rho (D_0)$.

\begin{cor}
	\label{cor:point.spectrum.family}
	Let Assumption~{\rm\ref{asm:II}} be satisfied and let $\Aa_0$, $S_0$ and $\Theta$ be as in Definition {\rm\ref{def:matrix.Schur.compl.family}}. Then
	\begin{equation}
		\pspec (\Aa_0) \cap \rho (D_0) \cap \Theta = \pspec (S_0 (\cdot)) \cap \rho (D_0).
	\end{equation}
\end{cor}

\begin{proof}
	Let $\lambda \in \Theta \cap \rho (D_0)$. By Proposition~\ref{prop:point.spectrum} applied to $\Aa_0 - \lambda$ and $S_0 (\lambda)$, we conclude that $0 \in \ker(\Aa_0-\lambda)$ if and only if $0 \in \ker S_0 (\lambda)$. Hence,
	\begin{equation*}
		\pspec (\Aa_0) \cap \rho (D_0) \cap \Theta = \pspec (S_0 (\cdot)) \cap \rho (D_0) \cap \Theta = \pspec (S_0 (\cdot)) \cap \rho (D_0). \qedhere
	\end{equation*}
\end{proof}

On subsets of $\Theta$ where the assumptions of Theorem~\ref{thm:first.implication} are satisfied, we obtain an inclusion of the (essential) spectrum of $S_0 (\cdot)$ in the (essential) spectrum of $\Aa_0$.

\begin{cor}
	\label{cor:family.first.implication}
	Let Assumption~{\rm\ref{asm:II}} be satisfied and let $\Aa_0$, $S_0$ and $\Theta$ be defined as in Definition {\rm\ref{def:matrix.Schur.compl.family}}.
	\begin{enumerate}
		\item The following inclusion holds
		\begin{equation}
			\sigma (S_0 (\cdot)) \cap \rho (D_0) \subset \sigma (\Aa_0) \cap \rho (D_0) \cap \Theta.
		\end{equation}
		\item If $\Sigma \subset \Theta$ is such that, for all $\lambda \in \Sigma$, the assumptions of Theorem~{\rm\ref{thm:first.implication}~(ii)} are satisfied with $\Aa_0 - \lambda$ and $S_0 (\lambda)$, then
		\begin{equation}
		 	\essspec (S_0 (\cdot)) \cap \Sigma \subset \essspec (\Aa_0) \cap \Sigma.
		\end{equation}
	\end{enumerate}
\end{cor}

\begin{proof}
	The statements in (i) and (ii), respectively, follow similarly to the proof of Corollary~\ref{cor:point.spectrum.family} by applying Theorem~\ref{thm:first.implication} (i) and (ii) to $\Aa_0 - \lambda$ and $S_0 (\lambda)$ (with fixed $\lambda \in \rho (D_0) \cap \Theta$  in (i) and $\lambda \in \Sigma$ in (ii)).
\end{proof}

The reverse inclusions are obtained from Corollary~\ref{cor:second.implication}.

\begin{cor}
	\label{cor:family.second.implication}
	Let Assumption~{\rm\ref{asm:II}} be satisfied and let $\Aa_0$, $S$, $S_0$ and $\Theta$ be defined as in Definition {\rm\ref{def:matrix.Schur.compl.family}}.
	\begin{enumerate}
		\item If $\Sigma \subset \Theta$ is such that, for all $\lambda \in \Sigma$, the assumptions of Corollary~{\rm\ref{cor:second.implication}~(i)} are satisfied with $\Aa_0 - \lambda$, $S (\lambda)$ and $S_0 (\lambda)$, then
		\begin{equation}
			\sigma (\Aa_0) \cap \Sigma \subset \sigma (S_0 (\cdot)) \cap \Sigma.
		\end{equation} 
		\item  If $\Sigma \subset \Theta$ is such that, for all $\lambda \in \Sigma$, the assumptions of Corollary~{\rm\ref{cor:second.implication}~(ii)} are satisfied with $\Aa_0 - \lambda$, $S (\lambda)$ and $S_0 (\lambda)$, then
		\begin{equation}
			\essspec (\Aa_0) \cap \Sigma \subset \essspec (S_0 (\cdot)) \cap \Sigma.
		\end{equation}
	\end{enumerate}
\end{cor}

\begin{proof}
	The statements follow similarly to the proof of Corollary~\ref{cor:point.spectrum.family} by applying Corollary~\ref{cor:second.implication} to $\Aa_0 - \lambda$ and $S_0 (\lambda)$ (with fixed $\lambda \in \Sigma$).
\end{proof}

\section{Damped wave equation with irregular damping and potential}
\label{sec:DWE}

As an application of our results, we consider the linearly damped wave equation
\begin{equation}
\label{eq:DWE}
	\partial^2_t u(t, x) + 2 a(x) \partial_t u(t, x) = (\Delta_x - q (x)) u (t, x), \qquad t>0, \quad x \in \Omega,
\end{equation}
on $\Omega \subset \Rn$ with non-negative and possibly singular and/or unbounded damping $a$ and potential $q$. Equation \eqref{eq:DWE} can be written as the following first order Cauchy problem 
\begin{equation}
	\label{eq:DWE.Cauchy.problem}
	\partial_t \left(
	\begin{array}{c}
		u_1 (t,x) \\
		u_2 (t,x)
	\end{array} \right) = \left(
	\begin{array}{cc}
		0 &  1 \\
		\Delta_x - q(x) & - 2a(x)
	\end{array} \right)\left(
	\begin{array}{c}
		u_1 (t,x) \\
		u_2 (t,x)
	\end{array} \right).
\end{equation}

Spectral properties of the operator matrix above determine existence, uniqueness and behaviour of the solutions to \eqref{eq:DWE} and have been studied extensively. While the majority of results relies on relative boundedness of $a$ with respect to $\Delta - q$, see e.g.\ \cite{Ammari-Nicaise-2015,Gesztesy-Goldstein-Holden-Teschl-2012,Jacob-Tretter-Trunk-Vogt-2018}, the recent results \cite{Freitas-Siegl-Tretter-2018,Ikehata-Takeda-2020} do not follow standard patterns and allow stronger damping. Combining a distributional approach similar to \cite{Ammari-Nicaise-2015,Gesztesy-Goldstein-Holden-Teschl-2012,Jacob-Tretter-Trunk-Vogt-2018} with structural observations from \cite{Freitas-Siegl-Tretter-2018}, our method enables us to not only omit the assumption on relative boundedness of the damping but also to substantially lower the regularity of the coefficients. We thereby close a gap which was left open between the technical assumptions in \cite{Freitas-Siegl-Tretter-2018} and the minimal ones suggested by the applied sesquilinear form techniques therein, see Remark~\ref{rem:DWE.same.operator}.

The main result of this section is Theorem~\ref{thm:DWE}; assuming only that $a,q \in \Loneloc (\Omega)$, we define an m-accretive realisation of the operator matrix on the right hand side of \eqref{eq:DWE.Cauchy.problem} in a suitable (standard choice) Hilbert space and show spectral equivalence to its second Schur complement. By semigroup theory, this guarantees existence and uniqueness of the solutions to the underlying Cauchy problem. We thereby cover the essential part of \cite{Freitas-Siegl-Tretter-2018}, where the generation of a strongly continuous contraction semigroup was shown under more restrictive assumptions, see Remark~\ref{rem:DWE.same.operator}. Moreover, we point out that our method can also be employed to realise distributional dampings as considered e.g.\ in \cite{Ammari-Nicaise-2015, Krejcirik-Royer-2022-arxiv, Krejcirik-Kurimaiova-2020}, see Remark~\ref{rem:DWE.distr.damping}.

\subsection{Assumptions and main result}

We make the following natural \emph{low regularity assumptions}.

\begin{asm} \label{asm:III}
	Let $\Omega \subset \Rn$ be open and let $a, q \in \Loneloc (\Omega)$ such that $a, q \ge 0$ almost everywhere in $\Omega$. \hfill //
\end{asm}

In the following, we denote $\norm{\cdot} \defeq \norm{\cdot}_{L^2 (\Omega)}$ and $\iprod{\cdot}{\cdot} \defeq \iprod{\cdot}{\cdot}_{L^2 (\Omega)}$ and use the same notation for norm and inner product in $L^2 (\Omega)^n$ if no confusion can arise.

\subsubsection{Spectral correspondence and semigroup generation}

	We establish the operator theoretic framework behind the Cauchy problem \eqref{eq:DWE.Cauchy.problem} under Assumption~\ref{asm:III}. Let $\Ww(\Omega)$ be the Hilbert space completion of $\Coo (\Omega)$ with respect to the inner product
	\begin{equation}\label{eq:DWE.norm.Ww.Omega}
		\iprod{f}{g}_\Ww \defeq \int_{\Omega} \nabla f \cdot \overline{\nabla g} \d x + \int_{\Omega} q f \overline g \d x, \qquad f, \, g \in \Coo (\Omega);
	\end{equation}
 	recall that $q \ge 0$ and $a \ge 0$ a.e.\ in $\Omega$. Moreover, define
 	\begin{equation}\label{eq:DWE.Dd.S}
 		\Dd_S \defeq H_0^1 (\Omega) \cap \dom q^\frac12 \cap \dom a^\frac12.
 	\end{equation}
	Let $\Aa_0$ and $S_0 (\cdot)$ be the (family of) maximal operators in $\Ww(\Omega) \oplus L^2 (\Omega)$ and $L^2 (\Omega)$, respectively, corresponding to the differential expressions
	\begin{equation}\label{eq:DWE.def.Aa.S}
		\Aa \defeq \begin{pmatrix}
			0 & I \\
			\Delta - q & -2a
		\end{pmatrix}, \quad S(\lambda) \defeq -\frac{1}{\lambda} (- \Delta + q + 2 \lambda a + \lambda^2), \quad \lambda \in \C \setminus \{0\},
	\end{equation}
	on their respective maximal domains
	\begin{equation}\label{eq:DWE.dom.Aa0.S0}
		\begin{aligned}
			\dom \Aa_0 & \defeq \set{(f, g) \in \Ww (\Omega) \times \Dd_S}{(\Delta - q) f - 2 a g \in L^2(\Omega)}, \\
			\dom S_0 (\lambda) & \defeq \set{f \in \Dd_S}{(\Delta - q - 2 \lambda a)f \in L^2 (\Omega)}.
		\end{aligned}
	\end{equation}
	Here the above operations are understood in a standard (antilinear) distributional sense; see Definitions~{\rm\ref{def:DWE.spaces}}, {\rm\ref{def:DWE.operators}}, {\rm\ref{def:DWE.Aa.S}} and Remark~\ref{rem:DWE.distr.act} below for details. Notice that the Schur complement $S_0 (\cdot)$ coincides with the usual definition by means of its quadratic form.
	
	The main result of this section reads as follows and is proven in Section~\ref{sec:DWE.proof}.

\begin{thm}
	\label{thm:DWE}
	Let Assumption~{\rm \ref{asm:III}} be satisfied and let $\Aa_0$ and $S_0(\cdot)$ be as in \eqref{eq:DWE.def.Aa.S} and \eqref{eq:DWE.dom.Aa0.S0}. Then $-\Aa_0$ is m-accretive, thus $\Aa_0$ generates a strong\-ly continuous contraction semigroup on $\Ww(\Omega) \oplus L^2 (\Omega)$ and its domain is dense both in $\Ww(\Omega) \oplus \Dd_S$ and in $\Ww(\Omega) \oplus L^2(\Omega)$. Moreover, the {\rm(}point and essential{\rm)} spectra of $\Aa_0$ and $S_0 (\cdot)$ are equivalent on $\C \setminus (-\infty, 0]$,
	\begin{equation}
		\label{eq:DWE.spectra.equiv} 
		\sigma (\Aa_0) \setminus (-\infty, 0] = \sigma (S_0 (\cdot)) \setminus (-\infty, 0],
	\end{equation}
	and analogously for $\sigma_{\operatorname{p}}$ and $\sigma_{\operatorname{e2}}$ instead of $\sigma$. Moreover, the following relations hold on $(-\infty,0)$
	\begin{align}
		\label{eq:DWE.spec.incl} \sigma (\Aa_0) \cap (-\infty,0) & \supset \sigma (S_0 (\cdot)) \cap (-\infty,0), \\
		\label{eq:DWE.pspec.equiv}\pspec (\Aa_0) \cap (-\infty,0) & = \pspec (S_0 (\cdot)) \cap (-\infty,0).
	\end{align}
\end{thm}

\begin{rem}
	\label{rem:DWE.same.operator}
	If the assumptions on damping and potential in \cite{Freitas-Siegl-Tretter-2018} are satisfied, the operator $G$ introduced therein coincides with $\Aa_0$. Assuming essentially that for every $\eps>0$ there exists $C_\eps \ge 0$ such that
	\begin{equation}
		a \in W^{1,\infty}_{\rm loc} (\overline{\Omega}), \qquad |\nabla a| \le \eps a^\frac32 + C_\eps,
	\end{equation}
	see \cite[Asm.\ I]{Freitas-Siegl-Tretter-2018} for the precise more general assumptions, the authors define $G$ as the closure of the operator matrix $G_0$ in $\Ww (\Omega) \oplus L^2 (\Omega)$ given by
	\begin{equation}
		G_0 = \begin{pmatrix}
			0 & I \\
			\Delta -q & -2a
		\end{pmatrix}, \qquad \dom G_0 = (\dom (-\Delta +q)_{L^2 (\Omega)} \cap \dom a)^2;
	\end{equation}
	here $(-\Delta +q)_{L^2 (\Omega)}$ denotes the Friedrichs extension in $L^2(\Omega)$ of $-\Delta +q$ on $\Coo (\Omega)$. They show that, independently of $\lambda \in \C \setminus (- \infty, 0]$,
	\begin{equation}
		\dom S_0 (\lambda) = \dom (-\Delta +q)_{L^2 (\Omega)} \cap \dom a \subset \Ww (\Omega) \cap \Dd_S.
	\end{equation}
	From this it follows easily that $\dom G_0 \subset \dom \Aa_0$ and their actions clearly coincide. Since both $-G$ and $-\Aa_0$ are m-accretive, see \cite[Thm.\ 2.2]{Freitas-Siegl-Tretter-2018}, their equality already follows from the inclusion $G \subset \Aa_0$, see e.g.\ \cite[\S V.3.10]{Kato-1995}. \hfill //
\end{rem}

\begin{rem}
	\label{rem:DWE.distr.damping}
	Our setting is more general than the perturbative framework in e.g.\ \cite{Ammari-Nicaise-2015, Jacob-Tretter-Trunk-Vogt-2018} and can equally be employed to cover distributional dampings studied therein. In fact, it is clear from the proofs below that any non-negative sesquilinear form on $\Coo (\Omega)$ can be implemented as damping and/or potential. This in particular includes Dirac delta dampings $a (x) = \delta (x-x_0)$ considered in~\cite[Chap.\ 4]{Ammari-Nicaise-2015} on bounded intervals or  in~\cite[Rem.~1]{Krejcirik-Kurimaiova-2020} on the real line; see also the recent work ~\cite{Krejcirik-Royer-2022-arxiv} where the more general case of non-compact star graphs is discussed.  \hfill //
\end{rem}

\subsection{Realisation of matrix and Schur complement}

In line with Sections~\ref{sec:abstract.results} and~\ref{sec:matrix.Schur.complement}, we present our approach to the underlying spectral problem. Note that since we use the second Schur complement, the roles of the spaces $\Hh_1$ and $\Hh_2$, as well as the matrix entries acting in them, are reversed correspondingly, see Remark~\ref{rem:second.Schur.compl}.

\subsubsection{Definition of spaces}

We introduce the spaces for Assumption~\ref{asm:II}~(i).

\begin{defi}
	\label{def:DWE.spaces}
	Let Assumption~\ref{asm:III} hold, let $\Hh_1 \defeq \Ww (\Omega)$ with $\Ww (\Omega)$ as in Theorem~\ref{thm:DWE} and $\Hh_2 \defeq L^2 (\Omega)$. Let the space $\Dd_S$ defined in \eqref{eq:DWE.Dd.S} be equipped with the inner product
	\begin{equation}\label{eq:DWE.norm.S}
		\iprod{f}{g}_S \defeq \int_{\Omega} \nabla f \cdot \overline{\nabla g} \d x + \int_{\Omega} q f \overline g \d x + \int_{\Omega} a f \overline g \d x  + \int_{\Omega} f \overline g \d x, \quad f, \, g \in \Dd_S;
	\end{equation}	
	notice that $q, a \ge 0$ a.e.\  in $\Omega$. Moreover, define
	\begin{equation}
		\Dd_1 = \Dd_{-1} \defeq \Ww (\Omega), \qquad \Dd_{-2} = \Dd_{-S} \defeq \Dd_S^*. \tag*{//}
	\end{equation}
\end{defi}

\begin{prop}
	\label{prop:DWE.spaces}
	Under Assumption~{\rm \ref{asm:III}}, the spaces in Definition {\rm\ref{def:DWE.spaces}} satisfy Assumption~{\rm \ref{asm:II}~(i)} according to Remark~{\rm\ref{rem:second.Schur.compl}}. Moreover, $\Coo(\Omega)$ is dense in $\Dd_S$ and one can embed $\Dd_S \subset \Ww(\Omega)$ continuously.
\end{prop}

\begin{proof}
	Since $q\in\Loneloc(\Omega)$, the Hilbert space completion of $\Coo(\Omega)$ with respect to the inner product in \eqref{eq:DWE.norm.Ww.Omega} is well-defined. Clearly, all other spaces in Definition~\ref{def:DWE.spaces} are also well-defined and Hilbert. The inclusion $\Coo(\Omega) \subset \Dd_S$ is clear from the assumption $a,q \in \Loneloc(\Omega)$; its density can be shown using standard techniques. It thus follows from the construction of $\Ww(\Omega)$ and the inequality $\norm{\cdot}_\Ww \le \norm{\cdot}_S$ that one can embed $\Dd_S \subset \Ww(\Omega)$ continuously.
		
	It remains to show that $\Dd_S\subset L^2(\Omega) \subset \Dd_S^*$, where the corresponding embeddings are continuous with dense range. The density of $\Dd_S$ in $L^2(\Omega)$ is a consequence of $\Coo (\Omega)\subset \Dd_S$, and since we have $\norm{\cdot} \le \norm{\cdot}_S$ by construction, it is continuously embedded. Upon identification of $L^2 (\Omega)$ with its anti-dual space, we further obtain
	\begin{equation}
		\label{eq:DWE.dual.inclusion}
		\Dd_S \subset L^2 (\Omega) \equiv L^2 (\Omega)^* \subset \Dd_S^*.
	\end{equation}
	The second inclusion in \eqref{eq:DWE.dual.inclusion} is realised by the continuous embedding
	\begin{equation}
		\label{eq:DWE.dual.embedding}
		L^2 (\Omega) \ni f \equiv \iprod{f}{\cdot} \mapsto \iprod{f}{\cdot}\vert_{\Dd_S} \in \Dd_S^*.
	\end{equation}
	Indeed, since $\Dd_S$ is continuously embedded in $L^2(\Omega)$, the map in~\eqref{eq:DWE.dual.embedding} is well-defined and bounded; its injectivity follows from the density of $\Dd_S$ in $L^2 (\Omega)$.
	For Assumption~\ref{asm:II}~(i), it remains to show that $L^2 (\Omega)$ is dense in $\Dd_S^*$. It suffices, however, to observe that the embedding in \eqref{eq:DWE.dual.embedding} is nothing but the adjoint operator
	\begin{equation}
		I_{\Dd_S}^*: L^2 (\Omega)^* \to \Dd_S^*
	\end{equation}
	and that $\ran I_{\Dd_S}^*$ is dense in $\Dd_S^*$ since $\ker I_{\Dd_S} = \{0\}$, cf.\ \cite[p.\ 170]{Edmunds-Evans-1987}.
\end{proof}

\begin{rem}\label{rem:DWE.distr.spaces}
	The restriction of every functional in $\Dd_S^*$ to $\Coo(\Omega) \subset \Dd_S$ is a distribution, i.e.\ one can embed $\Dd_S^*\subset \Dd'(\Omega)$; indeed, since $a,q \in \Loneloc(\Omega)$, the convergence of test functions in $\Dd (\Omega)$ clearly implies their convergence in $\Dd_S$. Moreover, the density of $\Coo (\Omega)$ guarantees injectivity of the embedding $\Dd_S^* \to \Dd'(\Omega)$. 
	
	If $n\ge 3$ then we also have the inclusion $\Ww(\Omega) \subset \Dd'(\Omega)$; this follows from
	\begin{equation}
		\Ww (\Omega) \subset \dot{H}^1  (\Omega) \subset L^{\frac{2n}{n-2}} (\Omega) \subset \Loneloc (\Omega),
	\end{equation}
	where $\dot{H}^1  (\Omega)$ denotes the completion of $\Coo(\Omega)$ with respect to $\norm{\nabla \cdot}$, i.e.\ the first order homogeneous Sobolev space, and the second inclusion is a consequence of the Gagliardo-Nirenberg-Sobolev inequality. \hfill //
\end{rem}

\subsubsection{Definition of matrix entries}

We introduce the remaining objects needed for Assumption~{\rm\ref{asm:II}}.

\begin{defi}
	\label{def:DWE.operators}
	Let Assumption~\ref{asm:III} be satisfied and define the following operators
	\begin{equation}
		\begin{array}{rcl}
			A_0 = A \defeq & 0 & \in \Bb (\Ww (\Omega)), \\[1mm]
			C \defeq & \Delta - q & \in \Bb (\Ww (\Omega), \Dd_S^*),
		\end{array} \quad
		\begin{array}{rcl}
			B \defeq & I & \in \Bb (\Dd_S, \Ww (\Omega)), \\[1mm]
			D \defeq & -2a & \in \Bb (\Dd_S, \Dd_S^*);
		\end{array}
	\end{equation}
	see Definition~\ref{def:DWE.spaces}. Here the operators $\Delta - q$ and $a$ are the unique extensions of
	\begin{equation}
	\label{eq:DWE.CD}
		\begin{aligned}
			\dprod{(\Delta - q) f}{g}_{\Dd_S^* \times \Dd_S} & \defeq - \int_\Omega \nabla f \cdot \overline{\nabla g} \d x - \int_\Omega q f \overline g \d x, \\[1.5mm]
			\dprod{af}{g}_{\Dd_S^* \times \Dd_S} & \defeq \int_\Omega a f \overline g \d x,
		\end{aligned}
			 \qquad f, g \in \Coo (\Omega),
	\end{equation}
	see Proposition~\ref{prop:DWE.entries} below for details. \hfill //
\end{defi}

\begin{prop}
	\label{prop:DWE.entries}
	Under Assumption~{\rm\ref{asm:III}}, the operators defined in Definition~{\rm\ref{def:DWE.operators}} are well-defined and satisfy Assumption~{\rm\ref{asm:II}} according to Remark~{\rm\ref{rem:second.Schur.compl}}.
\end{prop}

\begin{proof}
	Since $\Dd_S\subset \Ww (\Omega)$ is continuously embedded by Proposition~\ref{prop:DWE.spaces}, clearly	\begin{equation}
		0 \in \Bb (\Ww (\Omega)), \qquad I \in \Bb (\Dd_S, \Ww (\Omega)).
	\end{equation}
	It remains to show that $C$ and $D$ are well-defined by \eqref{eq:DWE.CD} and bounded between the claimed spaces. For $f, g \in \Coo (\Omega)$, we have the inequality
	\begin{equation}
		\label{eq:DWE.C.bounded}
		\abs{\dprod{(\Delta - q) f}{g}_{\Dd_S^* \times \Dd_S}} \le \norm{f}_\Ww \norm{g}_\Ww \le \norm{f}_\Ww \norm{g}_S.
	\end{equation}
	Taking into account \eqref{eq:DWE.C.bounded} and the density of $\Coo (\Omega)$ in $\Dd_S$, see Proposition~\ref{prop:DWE.spaces}, the formula \eqref{eq:DWE.CD} determines a unique bounded antilinear functional on $\Dd_S$
	\begin{equation}
		(\Delta - q) f \in \Dd_S^*, \qquad f \in \Coo (\Omega).
	\end{equation}
	Moreover, from \eqref{eq:DWE.C.bounded} it follows that $f \mapsto (\Delta - q) f$ is a bounded map from $\Ww (\Omega)$ to $\Dd_S^*$ and therefore, by density of $\Coo (\Omega)$ in $\Ww (\Omega)$, has a unique extension
	\begin{equation}
		\Delta - q \in \Bb (\Ww (\Omega), \Dd_S^*).
	\end{equation}
	In the same way, by deriving the estimate
	\begin{equation}
		\abs{\dprod{af}{g}_{\Dd_S^* \times \Dd_S}} \le \norm{a^\frac12 f} \norm{a^\frac12 g} \le \norm{f}_S \norm{g}_S, \quad f, g \in \Coo (\Omega),
	\end{equation}
 	one shows that $a \in \Bb (\Dd_S, \Dd_S^*)$ is well-defined.
\end{proof}

\subsubsection{Definition of matrix and Schur complement}

We proceed analogously to Definition~\ref{def:matrix.Schur.compl.family}, cf.\ Remark~\ref{rem:second.Schur.compl}.

\begin{defi}
	\label{def:DWE.Aa.S}
	Let Assumption~\ref{asm:III} be satisfied. Define the operator matrix
	\begin{equation}
		\Aa \defeq \left(
		\begin{array}{cc}
			0 & I \\
			\Delta -q & -2a
		\end{array}\right)
		\in \Bb (\Ww (\Omega) \oplus \Dd_S, \Ww (\Omega) \oplus \Dd_S^*)
	\end{equation}
	and its second Schur complement
	\begin{equation}\label{eq:DWE.def.S}
		S (\lambda) \defeq -2a - \lambda + \frac{1}{\lambda} (\Delta - q)\vert_{\Dd_S} \in \Bb (\Dd_S, \Dd_S^*), \qquad \lambda \in \Theta \defeq \C \setminus \{0\};
	\end{equation}
	see Definitions~\ref{def:DWE.spaces} and~\ref{def:DWE.operators}, as well as Definition~\ref{def:matrix.Schur.compl.family} and notice that $\Theta$ satisfies \eqref{eq:Theta} therein. Let $\Aa_0$ and $S_0 (\cdot)$, respectively, be the corresponding (family of) maximal ope\-ra\-tors in $\Ww (\Omega) \oplus L^2 (\Omega)$ and $L^2 (\Omega)$; more precisely, $\Aa_0 \defeq \Aa \vert_{\dom \Aa_0}$ and $S_0(\lambda) \defeq S(\lambda) \vert_{\dom S_0 (\lambda)}$ with their respective domains
		\begin{align*}
		\dom \Aa_0 & \defeq \set{(f, g) \in \Ww (\Omega)  \times \Dd_S}{\Aa (f, g) \in \Ww(\Omega)  \times L^2 (\Omega)}, \\
		 \dom S_0 (\lambda) & \defeq \set{f \in\Dd_S}{S (\lambda) f \in L^2 (\Omega)}.  \tag*{//}
		\end{align*}
\end{defi}

The following proposition shows that~\eqref{eq:DWE.def.S} agrees with the standard definition of the Schur complement via its quadratic form.

\begin{prop}\label{prop:DWE.S.action}
	Let Assumption~{\rm\ref{asm:III}} hold and let $S(\cdot)$ be as in Definition~{\rm\ref{def:DWE.Aa.S}}. Then, for every $\lambda \in \C \setminus \{0\}$ and $f, g \in \Dd_S$,
	\begin{equation}
		\begin{aligned}
			\dprod{S(\lambda)f}{g}_{\Dd_S^* \times \Dd_S} = - \frac{1}{\lambda} & \left( \int_{\Omega} \nabla f \cdot \overline{\nabla g} \d x + \int_{\Omega} qf\overline g \d x \right. \\
			& \qquad \qquad\qquad \qquad \left. + 2 \lambda \int_{\Omega} a f \overline{g} \d x + \lambda^2 \int_{\Omega} f\overline g \d x \right).
		\end{aligned}
	\end{equation}
\end{prop}

\begin{proof}
	Considering Definitions~\ref{def:DWE.operators} and~\ref{def:DWE.Aa.S}, it is clear that the claimed identity holds for $f, g \in \Coo (\Omega)$. Since $\Coo (\Omega)$ is dense in $\Dd_S$ and both sides of the formula are continuous with respect to convergence in $\Dd_S$, it remains valid for $f, g \in\Dd_S$.
\end{proof}

\begin{rem}\label{rem:DWE.distr.act}
	Under Assumption~\ref{asm:III}, the actions of $\Aa$ and $S(\cdot)$ introduced in Definition~\ref{def:DWE.Aa.S} coincide with their standard distributional definitions, cf.\ Remark~\ref{rem:DWE.distr.spaces}. Indeed, taking into account that $q \in\Loneloc (\Omega)$, clearly
	\begin{equation}\label{eq:DWE.C.distr.core}
		\Delta f - qf \in \Loneloc (\Omega), \qquad f\in\Coo (\Omega),
	\end{equation}
	is a regular distribution and coincides with $Cf \in \Dd_S^*$ as defined in~\eqref{eq:DWE.CD}. The action of $C$ on the completion $\Ww (\Omega)$, however, is constructed by continuous extension, see~\eqref{eq:DWE.CD}, and is thus given as a limit of regular distributions of the form \eqref{eq:DWE.C.distr.core}; notice that convergence of functionals in $\Dd_S^*$ implies their convergence in $\Dd'(\Omega)$. Finally, the following distributions
	\begin{equation}
		ag \in \Loneloc (\Omega), \quad - \frac{1}{\lambda} \left( -\Delta f + qf + 2 \lambda a f + \lambda^2 f \right) \in \Dd'(\Omega), \quad \lambda \in \C \setminus \{0\}, \quad f, g \in \Dd_S,
	\end{equation}
	are well-defined and coincide with $Dg \in \Dd_S^*$ and $S(\lambda) f \in \Dd_S^*$ as in~\eqref{eq:DWE.CD} and~\eqref{eq:DWE.def.S}, respectively; see Proposition~\ref{prop:DWE.S.action} and notice that $\Dd_S \subset \dom a^{\frac12} \cap \dom q^{\frac12}$. \hfill //
\end{rem}

\subsection{Proof of Theorem{\rm~\ref{thm:DWE}}}\label{sec:DWE.proof}

We first formulate a crucial ingredient for the proof of Theorem~\ref{thm:DWE} as a lemma. For every $\lambda \in \C \setminus (\infty, 0]$, it provides the existence of an extension as in \eqref{eq:S.gen.inv.shift.ext}; this is needed in order to apply Corollaries~\ref{cor:family.dense.domain}~(ii) and~\ref{cor:family.second.implication}.

\begin{lem}
	\label{lem:DWE.shift.extension}
	Let Assumption~{\rm \ref{asm:III}} be satisfied and let $S_0 (\cdot)$ be as in Definition {\rm\ref{def:DWE.Aa.S}}. For all $\lambda \in \C \setminus (- \infty, 0]$, the domain $\dom S_0 (\lambda)$ is dense in $\Dd_S$ and there exists $z_\lambda \in \rho (S_0(\lambda))$ such that
	\begin{equation} \label{eq:DWE.S.shift.inv}
		(S(\lambda) - z_\lambda)^{-1}  \in \Bb (\Dd_S^*, \Dd_S).
	\end{equation}
	Moreover, if $\re \lambda >0$, then one can choose $z_\lambda = 0$.
\end{lem}

\begin{proof}
	Fix $\lambda \notin (- \infty, 0]$. Analogously to the proof of \cite[Lem.\ 2.3]{Freitas-Siegl-Tretter-2018}, one can show that there exists $z_\lambda \in \C$ such that $S (\lambda) - z_\lambda$ 
	corresponds to a bounded and coercive sesquilinear form on $\Dd_S$;  cf.\ the proof of Lemma~\ref{lem:matrix.DE.S.coercive}, where the analogous statement is proven in a different setting, and also the proof of Theorem~\ref{thm:KG}. By the Lax-Milgram Theorem, see e.g.\ \cite[Cor.\ IV.1.2]{Edmunds-Evans-1987}, we conclude \eqref{eq:DWE.S.shift.inv}, $z_\lambda \in \rho (S_0 (\lambda))$ and density of the maximal domain $\dom (S_0 (\lambda) - z_\lambda) = \dom S_0 (\lambda)$ in $\Dd_S$. Finally, from the proof of \cite[Lem.\ 2.3]{Freitas-Siegl-Tretter-2018}, it is clear that one can choose $z_\lambda = 0$ in case that $\re \lambda > 0$.
	%
\end{proof}

\begin{proof}[Proof of Theorem {\rm\ref{thm:DWE}}]
	We start by pointing out that, by Propositions~\ref{prop:DWE.spaces} and~\ref{prop:DWE.entries}, Assumption~\ref{asm:II} is satisfied, the objects in Definition~\ref{def:DWE.Aa.S} are well-defined and the results of Section~\ref{sec:matrix.Schur.complement} applicable. Moreover, $\Aa_0$ and $S_0(\cdot)$ as in Definition~\ref{def:DWE.Aa.S} coincide with their definition in~\eqref{eq:DWE.def.Aa.S}, see Remark~\ref{rem:DWE.distr.act}. Considering this, the description of their domains in \eqref{eq:DWE.dom.Aa0.S0} follows immediately from Definition~\ref{def:DWE.Aa.S}.

	Let us show the identities \eqref{eq:DWE.spectra.equiv}--\eqref{eq:DWE.pspec.equiv}. To this end, we refer to Remark~\ref{rem:second.Schur.compl} and apply the results from Section~\ref{sec:matrix.Schur.complement} correspondingly. First, since $\rho (A_0) = \Theta = \C \setminus \{0\}$, Corollary~\ref{cor:point.spectrum.family} implies
	\begin{equation} \label{eq:DWE.pspec.equiv.1}
		\pspec (\Aa_0) \setminus \{0\} = \pspec (S_0 (\cdot)).
	\end{equation}
	We proceed by showing that the assumptions of Corollary~\ref{cor:family.second.implication} are satisfied. Let $\lambda\in \Sigma \defeq \C \setminus (- \infty, 0]$ be arbitrary. Since $\Dd_{-S} = \Dd_{-2}$, one clearly has
	\begin{equation}
		S(\lambda) \in \Bb (\Dd_S, \Dd_{-S}), \qquad C (\dom A_0) \subset \Dd_{-S}.
	\end{equation}
	Moreover, according to Lemma~\ref{lem:DWE.shift.extension}, $\dom S_0 (\lambda)$ is dense in $\Dd_S$ and there exists $z_\lambda \in \rho (S_0 (\lambda))$ such that
	\begin{equation}\label{eq:DWE.S.shift.ext}
		(S_0 (\lambda) - z_\lambda)^{-1} \subset S_z^\ddagger (\lambda) \defeq (S(\lambda) - z_\lambda)^{-1}  \in \Bb (\Dd_{-S}, \Dd_S).
	\end{equation}
	Considering that $\Sigma \subset \rho (A_0)$, we can thus apply Corollary~\ref{cor:family.second.implication} to conclude
	\begin{equation}
	\label{eq:DWE.ess.spec.second.inclusion}
		\sigma (\Aa_0) \setminus (- \infty, 0] \subset \sigma (S_0 (\cdot)) \setminus (- \infty, 0],
	\end{equation}
	and analogously with $\sigma_{\operatorname{e2}}$ instead of $\sigma$.	Furthermore, from Corollary~\ref{cor:family.first.implication} (i) with $\rho (A_0) = \Theta = \C \setminus \{0\}$, we obtain
	\begin{equation}\label{eq:DWE.spec.first.inclusion}
		\sigma (\Aa_0) \setminus \{0\} \supset \sigma (S_0 (\cdot)).
	\end{equation}
	Since $\Sigma \subset \rho (A_0)$ and since, for all $\lambda \in \Sigma$, Lemma~\ref{lem:DWE.shift.extension} implies that $\rho (S_0 (\lambda)) \neq \emptyset$ and thus $S_0 (\lambda) \in \Cc (\Hh_2)$, the assumptions of Corollary~\ref{cor:family.first.implication} (ii) are satisfied and we conclude
	\begin{equation}\label{eq:DWE.essspec.first.inclusion}
		\sigma_{\operatorname{e2}} (\Aa_0) \setminus (- \infty, 0] \supset \sigma_{\operatorname{e2}} (S_0 (\cdot)) \setminus (- \infty, 0].
	\end{equation}
	In summary, \eqref{eq:DWE.spectra.equiv} with $\sigma$, $\sigma_{\operatorname{p}}$ and $\sigma_{\operatorname{e2}}$ follows from \eqref{eq:DWE.pspec.equiv.1}, \eqref{eq:DWE.ess.spec.second.inclusion} with $\sigma$ and $\sigma_{\operatorname{e2}}$, \eqref{eq:DWE.spec.first.inclusion} and \eqref{eq:DWE.essspec.first.inclusion}; the inclusion \eqref{eq:DWE.spec.incl} is a consequence of \eqref{eq:DWE.spec.first.inclusion} and the identity \eqref{eq:DWE.pspec.equiv} follows from \eqref{eq:DWE.pspec.equiv.1}.
	We continue by showing that $- \Aa_0$ is accretive, i.e.\ that
	\begin{equation}
		\re \iprod{\Aa_0 (f, g)}{(f, g)}_\Hh \le 0, \qquad (f, g) \in \dom \Aa_0.
	\end{equation}
	To this end, let $(f, g) \in \dom \Aa_0$. By the density of $\Coo(\Omega)$ in $\Dd_S$, see Proposition~\ref{prop:DWE.spaces}, and by construction of $\Ww(\Omega)$, there exist $\{f_m\}_m, \{g_m\}_m \subset \Coo (\Omega)$ such that $f_m \to f$ in $\Ww (\Omega)$ and $g_m \to g$ in $\Dd_S$ as $m \to \infty$. Moreover, $\Dd_S \subset \Ww(\Omega)$ is continuously embedded, see Proposition~\ref{prop:DWE.spaces}, and we conclude
	\begin{equation}
		\label{eq:DWE.Ww.approximation}
		\iprod{g}{f}_\Ww = \lim_{m \to \infty} \left( \iprod{\nabla g_m}{\nabla f_m} + \iprod{q^\frac12 g_m}{q^\frac12 f_m} \right).
	\end{equation}
	One can easily see from the definitions in \eqref{eq:DWE.CD} that
	\begin{equation}
		\begin{aligned}
			\dprod{(\Delta - q) f - 2 a g}{g}_{\Dd_S^* \times \Dd_S} & = - \lim_{m \to \infty} \left( \iprod{\nabla f_m}{\nabla g_m} + \iprod{q^\frac12 f_m}{q^\frac12 g_m}\right. \\ 
			& \qquad \qquad \qquad \qquad \qquad \qquad \left. + 2 \iprod{a^\frac12 g_m}{a^\frac12 g_m}\right).
		\end{aligned}
	\end{equation}
	Using this, \eqref{eq:DWE.Ww.approximation} and $\dprod{h}{\cdot}_{\Dd_S^* \times \Dd_S} = \iprod{h}{\cdot}$ for $h \in L^2 (\Omega)$, we further derive
	\begin{equation}
		\begin{aligned}
			\iprod{\Aa_0 (f, g)}{(f, g)}_\Hh & = \iprod{g}{f}_\Ww + \iprod{(\Delta - q) f - 2 a g}{g} \\[1mm]
			& =  - 2 \i \lim_{m \to \infty} \im \left(\iprod{\nabla f_m}{\nabla g_m} + \iprod{q^\frac12 f_m}{q^\frac12 g_m}\right) \\
			& \qquad \qquad \qquad \qquad  \qquad \qquad - 2 \lim_{m \to \infty}\iprod{a^\frac12 g_m}{a^\frac12 g_m}.
		\end{aligned}
	\end{equation}
	Finally, the accretivity of $-\Aa_0$ then follows from
	\begin{equation}
		\re \iprod{\Aa_0 (f, g)}{(f, g)}_\Hh = - 2 \lim_{m \to \infty} \iprod{a^\frac12 g_m}{a^\frac12 g_m} \le 0.
	\end{equation}
	Next we note that, by Lemma \ref{lem:DWE.shift.extension}, we have $S(\lambda)^{-1} \in \Bb(\Dd_S^*, \Dd_S)$ and thus $0 \in \rho (S_0 (\lambda))$ whenever $\re \lambda > 0$. By taking complements in \eqref{eq:DWE.ess.spec.second.inclusion}, we conclude
	\begin{equation}
		\set{\lambda \in \C}{\re \lambda >0} \subset \rho (S_0 (\cdot)) \setminus (-\infty, 0] \subset \rho (\Aa_0).
	\end{equation}
	This implies that $- \Aa_0$ is m-accretive; the generation of a strongly continuous contraction semigroup then follows from a standard result in semigroup theory, see e.g.\ \cite[\S IX.1]{Kato-1995}.
	
	It remains to show that $\dom \Aa_0$ is dense in $\Ww(\Omega) \oplus \Dd_S$ and $\Ww(\Omega) \oplus L^2(\Omega)$. This, however, follows from $\Dd_{-S} = \Dd_S^* = \Dd_{-2}$ and Corollary~\ref{cor:family.dense.domain}~(ii), since we have shown already that the assumptions of Proposition~\ref{prop:Aa.dom.dense}~(ii) are satisfied in any point $\lambda \in \C \setminus (\infty, 0]$, see \eqref{eq:DWE.S.shift.ext}.
\end{proof}

\begin{rem}
	The density of $\dom \Aa_0$ in $\Ww(\Omega) \oplus L^2 (\Omega)$ (which already follows from its m-accretivity, see \cite[\S V.3.10]{Kato-1995}) can equally be shown by employing Corollary~\ref{cor:family.dense.domain}~(i) according to Remark~\ref{rem:second.Schur.compl}. Let therefore $(\Delta - q)_{L^2 (\Omega)}$ denote the Friedrichs extension in $L^2(\Omega)$ of $\Delta - q$ defined on $\Coo (\Omega)$. One can show that
	\begin{equation}
		\dom (\Delta - q)_{L^2 (\Omega)} \subset \dom C_0 = \dom A_0 \cap \dom C_0,
	\end{equation}
	where we used $\dom A_0 = \Ww (\Omega)$ and where $\dom C_0$ is defined analogously to \eqref{eq:B.dom}. Moreover, one can prove that $\dom (\Delta - q)_{L^2 (\Omega)}$ is dense in $\Ww (\Omega)$, which then by means of Corollary~\ref{cor:family.dense.domain}~(i) applied with any $\la_0 \in \C \setminus (-\infty, 0]$ gives the claimed density; notice that $\Dd_S$ is dense and boundedly embedded in $\Hh_2$ and thus $\dom S_0 (\la_0)$ is dense in $\Hh_2$ due to Lemma~\ref{lem:DWE.shift.extension}. \hfill //
\end{rem}

\begin{rem}
		The procedure in this section can be further extended to implement damped wave equations with accretive (differential) dampings in weighted spaces, see~\cite[Sec.~5.4]{Gerhat-Siegl-2022-preprint}. While the necessary well-behaved representation of the Schur complement could be constructed in Theorem~\ref{thm:DWE} by means of the Lax-Milgram theorem, for the described extension new Dirichlet realisations of Schr\"odinger type operators with accretive potentials in weighted spaces need to be implemented. The latter is done in~\cite{Gerhat-Siegl-2022-preprint} by establishing generalised coercivity estimates and applying a recent representation theorem in~\cite{Almog-Helffer-2015}. \hfill //
\end{rem}

\section{Singular coefficient matrix differential operators}
\label{sec:singular.matrix.DE}

We study the spectra of generic second order differential operator matrices of the form
\begin{equation}
\label{eq:matrix.DE}
	\left(
	\begin{array}{cc}
		-\Delta + q & \nabla \cdot \formb \\
		\formc \cdot \nabla & d
	\end{array}
	\right)
\end{equation}
with irregular coefficients acting in the Hilbert space $L^2 (\Omega) \oplus L^2 (\Omega)$ with $\Omega \subset\nolinebreak\Rn$. Operator matrices with this particular structure appear in several problems in mathematical physics, see e.g.\ \cite{Boegli-Marletta-2020,Ibrogimov-2017,Ibrogimov-Siegl-Tretter-2016,Ibrogimov-Tretter-2017,Konstantinov-1998,Kurasov-Lelyavin-Naboko-2008}.~Due to the apparent independence of their entries, however, their spectral analysis is not straightforward and previous results typically rely on the regularity or special form of coefficients. Our approach merges the conceptual idea of a dominant Schur complement in \cite{Ibrogimov-2017,Ibrogimov-Siegl-Tretter-2016,Ibrogimov-Tretter-2017} with a distributional framework, which enables us to substantially reduce the required regularity of the coefficients. In particular, we allow the latter to be singular and/or unbounded, as long as they satisfy certain (in some sense minimal) conditions arising from our setting in Section~\ref{sec:matrix.Schur.complement}, see Assumption~\ref{asm:IV}. 

The main results in this section are Theorems~\ref{thm:matrix.DE} and~\ref{thm:matrix.DE.accretive}; the former provides spectral equivalence of our realisation of the operator matrix \eqref{eq:matrix.DE} and its first Schur complement, whereas in the latter we show that if $q$ and $d$ are sectorial and $\formc = \overline{\formb}$, then the resulting operator matrix is m-accretive and generates a strongly continuous contraction semigroup.

\subsection{Assumptions and main results}

We impose the following \emph{low regularity assumptions} on the coefficients. 

\begin{asm} \label{asm:IV}
	Let $\Omega \subset \Rn$ be open.
	\begin{enumerate}
		\item \emph{Basic assumptions on coefficients}: Let $\formb, \formc : \Omega \to \Cn$ be measurable and
		\begin{equation}
			q \in \Loneloc (\Omega), \qquad d \in L^\infty_{\operatorname{loc}} (\Omega).
		\end{equation}
		\item \emph{Definition and regularity of $\boldpi$ on $\Theta$}: Denote by $\mathbf I \in \Cnn$ the identity matrix, let $\Theta \subset \C \setminus \essran d$ be connected and let
		\begin{equation} \label{eq:def.boldpi}
			\boldpi (\lambda) \defeq \mathbf{I} + (d - \lambda)^{-1} (\formb \otimes \formc) \in \Loneloc (\Omega)^{n \times n}, \qquad \lambda \in \Theta.
		\end{equation}
		\item \emph{Sectoriality of $q$ and $\boldpi$ on $\Phi$ {\rm(}after rotation and shift{\rm)}}: Let $\emptyset \neq \Phi \subset \Theta$.	For all $\lambda \in \Phi$, assume there exist $\omega_\lambda \in (-\pi,\pi]$ and $\gamma_\lambda \ge 0$ such that both
		\begin{equation}\label{eq:def.tilde.q.pi}
			\widetilde q(\lambda) \defeq \e^{\i \omega_\lambda} q + \gamma_\lambda, \qquad \widetilde \boldpi (\lambda) \defeq \e^{\i \omega_\lambda} \boldpi (\lambda),
		\end{equation}
		are sectorial; more precisely, a.e.\ in $\Omega$, let $\re \widetilde q (\lambda) \ge 0$, let the matrix $\re \widetilde \boldpi (\lambda)$ be positive definite and let $C_\lambda>0$ be such that
		\begin{equation}
			\begin{aligned}
				\abs{\im \widetilde q (\lambda)} & \le C_\lambda \re \widetilde q(\lambda), \\[1mm]
				\abs{\im \widetilde \boldpi (\lambda) \xi \cdot \overline{\xi}} & \le C_\lambda \re \widetilde \boldpi (\lambda) \xi \cdot \overline{\xi}, \qquad \xi \in \Cn.
			\end{aligned}
		\end{equation}
		Moreover, for all $\lambda, \mu \in \Phi$, assume there exist constants $m_{\lambda,\mu}, M_{\lambda,\mu}>0$ such that a.e.\ in $\Omega$ the following holds
		\begin{equation}
			\label{eq:matrix.DE.uniform.positivity}
			\begin{array}{rcccl}
				m_{\lambda,\mu} \re \widetilde q (\mu) & \!\!\! \le  &  \!\!\! \re \widetilde q (\lambda) & \!\!\! \le & \!\!\!  M_{\lambda,\mu} \re \widetilde q (\mu), \\[2mm]
				m_{\lambda,\mu} \re \widetilde \boldpi (\mu) & \!\!\! \le & \!\!\! \re \widetilde \boldpi (\lambda) & \!\!\! \le & \!\!\! M_{\lambda,\mu} \re \widetilde \boldpi (\mu).
			\end{array}
		\end{equation}
		\item \emph{Dominance of Schur complement}: For all $\lambda \in \Phi$, assume that
			\begin{equation}
			(d-\lambda)^{-1} \max \left(\big|(\re \widetilde \boldpi (\lambda))^{-\frac12} \formb\big|, \big|(\re \widetilde \boldpi (\lambda))^{-\frac12} \overline \formc\big|\right)  \in L^\infty (\Omega). \tag*{//}
			\end{equation}
	\end{enumerate}
\end{asm}

\begin{rem}\label{rem:matrix.DE.ass}
	\begin{enumerate}
		\item The assumptions $q \in \Loneloc(\Omega)$ and $\boldpi (\lambda) \in \Loneloc (\lambda)^{n \times n}$ are naturally minimal and guarantee that the generalised quadratic form of the Schur complement is densely defined for $\lambda \in\nolinebreak\Theta$. Assumption~\ref{asm:IV}~(iii), however, translates into sectoriality (after shift and rotation) of the Schur complement on the set of parameters $\Phi$.
		
		\item The assumption $d \in L^\infty_{\operatorname{loc}} (\Omega)$ is made for the sake of simplicity and can be relaxed. It is a sufficient condition for the following
		\begin{equation}\label{eq:matrix.DE.rem.cond.distr}
			\omega (d - \lambda)^{-1} \mathbf b \in L^2_{\operatorname{loc}}(\Omega)^n, \qquad \omega \in L^2_{\operatorname{loc}} (\Omega), \qquad (d - \lambda) \omega^{-1} \in L^2_{\operatorname{loc}}(\Omega),
		\end{equation}
		where $\lambda \in \Phi$ and $\omega$ is defined in~\eqref{eq:matrix.DE.D_2.def} below. The first two of the conditions above guarantee that our realisation of the operator matrix in~\eqref{eq:matrix.DE} and its Schur complement coincide with their standard definition in the distributional sense, see Remark~\ref{rem:matrix.DE.distr.operators}. The third condition ensures that $\Coo (\Omega) \subset \Dd_2$, where $\Dd_2$ is defined in~\eqref{eq:matrix.DE.D_2.def} below.
		
		\item Assumption~\ref{asm:IV}~(iv) is essential; it ensures that, on the set of parameters~$\Phi$, the Schur complement dominates the neighbouring factors of the Frobenius-Schur factorisation of the resolvent in a suitable way.
		
		\item By \eqref{eq:matrix.DE.uniform.positivity} and Lemma~\ref{lem:matrix.DE.d.equiv} below, it is equivalent to assume \eqref{eq:def.boldpi} or~Assumption~\ref{asm:IV}~(iv), respectively, only in an arbitrary point $\lambda_0 \in \Theta$ or $\lambda_0 \in \Phi$.\nolinebreak\hfill//
	\end{enumerate}
\end{rem}

From now on, we write $\norm{\cdot} \defeq \norm{\cdot}_{L^2 (\Omega)}$ and $\iprod{\cdot}{\cdot} \defeq \iprod{\cdot}{\cdot}_{L^2 (\Omega)}$; the same notation will be used for norm and inner product on $L^2 (\Omega)^n$ if no confusion can arise.

\subsubsection{Spectral correspondence}\label{sec:DWE.spec.corr}

Under Assumption~\ref{asm:IV}, we place the spectral problem for the operator matrix \eqref{eq:matrix.DE} in a suitable setting and apply the results from Section~\ref{sec:matrix.Schur.complement}. Fix an arbitrary point $\lambda_0 \in \Phi$ and set $q_0 \defeq \widetilde q (\lambda_0)$ and $\boldpi_0 \defeq \widetilde \boldpi (\lambda_0)$, see~\eqref{eq:def.tilde.q.pi}. Let $\Dd_S$ be the closure of $\Coo (\Omega)$ with respect to the inner product
\begin{equation}\label{eq:matrix.DE.D_S.iprod}
		\iprod{f}{g}_S \defeq \int_{\Omega} \re \boldpi_0 \nabla f \cdot \overline{\nabla g} \d x  + \int_\Omega \re q_0 f \overline g \d x + \int_{\Omega} f \overline g \d x, \quad f, g\in \Coo (\Omega);
\end{equation}
recall that $q_0 \ge 0$ and $\re \boldpi_0>0$ is positive definite a.e.\ in $\Omega$. Moreover, define
\begin{equation}\label{eq:matrix.DE.D_2.def}
	\Dd_2 \defeq L^2 (\Omega, \abs{d-\lambda_0}^2 \omega^{-2}), \qquad \omega \defeq \max \left(1,\abs{(\re \boldpi_0)^{-\frac12} \overline \formc}\right).
\end{equation}
We emphasise that, as topological spaces,  $\Dd_S$ and $\Dd_2$ do not depend on the choice of~$\lambda_0$. Indeed, by \eqref{eq:matrix.DE.uniform.positivity} and Lemma~\ref{lem:matrix.DE.d.equiv} below, their respective inner products generate equivalent norms for distinct choices of $\lambda_0 \in \Phi$.

Let $\Aa_0$ and $S_0 (\cdot)$ be the (family of) maximal operators in $L^2 (\Omega) \oplus L^2 (\Omega)$ and $L^2 (\Omega)$, respectively, corresponding to the differential expressions
\begin{equation}\label{eq:matrix.DE.Aa.S}
	\Aa \defeq \begin{pmatrix}
		-\Delta + q & \nabla \cdot \formb \\
		\formc \cdot \nabla & d
	\end{pmatrix}, \qquad 
	S(\lambda) \defeq -\nabla \cdot \boldpi (\lambda) \nabla + q - \lambda, \quad \lambda \in \Theta,
\end{equation}
see \eqref{eq:def.boldpi}, on their respective maximal domains
\begin{equation}\label{eq:matrix.DE.max.dom}
	\begin{aligned}
		\dom \Aa_0 & \defeq \left\{(f, g) \in \Dd_S \times \Dd_2 \, : \, (\Delta - q) f - \nabla \cdot \mathbf b g \in L^2 (\Omega),\right. \\
		& \qquad \qquad \qquad \qquad \qquad \qquad \qquad \qquad \left.\mathbf c \cdot \nabla f + d g \in L^2 (\Omega)\right\}, \\
		\dom S_0 (\lambda) & \defeq \set{f \in \Dd_S}{(\nabla \cdot \boldpi (\lambda) \nabla - q) f \in L^2 (\Omega)}.
	\end{aligned}
\end{equation}
Here the above operations are understood in a standard (antilinear) distributional sense, see Definitions~{\rm\ref{def:matrix.DE.spaces}}, {\rm\ref{def:matrix.DE.operators}}, {\rm\ref{def:matrix.DE.Aa.S}} and Remark~\ref{rem:matrix.DE.distr.operators} below for details. Note that $S_0 (\cdot)$ coincides with the standard definition of the Schur complement by means of its quadratic form on the set of parameters $\Phi$ were the latter is sectorial.

The first main result in this section, the spectral correspondence between operator matrix and its Schur complement, reads as follows; its proof can be found in Section~\ref{sec:DE.thms.proofs}.

\begin{thm}
	\label{thm:matrix.DE}
	Let Assumption~{\rm \ref{asm:IV}} be satisfied and let $\Aa_0$ and $S_0 (\cdot)$ be as in~\eqref{eq:matrix.DE.Aa.S} and \eqref{eq:matrix.DE.max.dom}. Then the equivalence of {\rm(}point and essential{\rm)} spectra between operator matrix and Schur complement holds on $\Phi$, i.e.
	\begin{equation}
		\label{eq:matrix.DE.spec.equivalent}
		\sigma (\Aa_0) \cap \Phi  = \sigma (S_0 (\cdot)) \cap \Phi,
	\end{equation}
	and analogously with $\sigma_{\operatorname{p}}$ and $\sigma_{\operatorname{e2}}$ instead of $\sigma$. Moreover, on the remaining set $\Theta \setminus \Phi$, one has
	\begin{align}
			\label{eq:matrix.DE.spec.first.incl} \sigma (\Aa_0) \cap (\Theta \setminus \Phi) & \supset \sigma (S_0 (\cdot)) \cap (\Theta \setminus \Phi), \\
			\label{eq:matrix.DE.pspec.equiv} \pspec (\Aa_0) \cap (\Theta \setminus \Phi) & = \pspec (S_0 (\cdot)) \cap (\Theta \setminus \Phi).
	\end{align}
	Finally, if
	$\dom B_0 \cap \dom d$ is dense in $L^2(\Omega)$, where
	\begin{equation}
		\dom B_0 = \set{f \in \Dd_2}{\nabla \cdot \formb f \in L^2 (\Omega)},
	\end{equation}
	then $\dom \Aa_0$ is dense in $L^2 (\Omega) \oplus L^2 (\Omega)$.
\end{thm}

\begin{rem}
	If e.g.\ $\formb \in W^{1,2}_{\operatorname{loc}} (\Omega)$, then both $\dom d$ and $\dom B_0$ contain $\Coo (\Omega)$, thus $\dom B_0 \cap \dom d$ is dense in $L^2 (\Omega)$ and $\dom\Aa_0$ is dense in $L^2 (\Omega) \oplus L^2 (\Omega)$. \hfill //
\end{rem}

\subsubsection{Generation of contraction semigroup}

When imposing additional assumptions on the structure of the problem, the resulting operator matrix proves to be m-accretive and thus generates a strongly continuous contraction semigroup, see Theorem~\ref{thm:matrix.DE.accretive} below.

More precisely, besides Assumption~\ref{asm:IV}~(i), we require that $\formc = \overline{\formb}$ and that $q$ and $d$ are sectorial, i.e.\ that there exist $0 \le \theta_q, \theta_d < \pi/2$ such that a.e.\ in $\Omega$
	\begin{equation}
		\label{eq:matrix.DE.q.d.sectorial}
		\begin{aligned}
			\re q & \ge 0, & \qquad \abs{\im q} & \le \tan \theta_q \re q, \\[0.5mm]
			\re d & \ge 0, & \qquad \abs{\im d} & \le \tan \theta_d \re d.
		\end{aligned}
	\end{equation}
The above structural assumptions imply Assumption~\ref{asm:IV}~(iii) with certain sets
\begin{equation}
	\label{eq:matrix.DE.Theta.Phi}
	\Theta  \defeq \set{\lambda \in \C }{|\arg \lambda | > \theta_d}, \quad 
	\Phi \defeq \set{\lambda \in \C }{|\arg \lambda |> \max (\theta_q, \theta_d)},
\end{equation}
where $\arg : \C \setminus \{0\} \to (-\pi,\pi]$ and $\arg 0 = 0$. Moreover, the regularity and dominance Assumptions~\ref{asm:IV}~(iii) and~(iv), respectively, reduce to
\begin{equation}\label{eq:matrix.DE.accr.regularity}
	\boldpi_0 \defeq \mathbf I + (d+1)^{-1} (\mathbf b \otimes \overline{\mathbf b}) \in \Loneloc(\Omega)^{n \times n}, \quad (d+1)^{-1} (\re \boldpi_0)^{-\frac12} \formb \in L^\infty (\Omega)^n;
\end{equation}
see Remark~\ref{rem:matrix.DE.ass}~(iv) and notice that $\re \boldpi_0 \ge \mathbf I$ is positive definite, see~\eqref{eq:matrix.DE.accretive.pi.lambda.positive.definite} below. Applying Theorem~\ref{thm:matrix.DE}, we thereby obtain the following theorem; its proof can be found in Section~\ref{sec:matrix.DE.accr.proof}.

\begin{thm}
	\label{thm:matrix.DE.accretive}
	Let Assumption~{\rm\ref{asm:IV}~(i)} hold, let $\mathbf c = \overline{\mathbf b}$, let $q$ and $d$ be sectorial with semi-angle $\theta_q, \theta_d \in [0,\pi/2)$, see ~\eqref{eq:matrix.DE.q.d.sectorial}, and let~\eqref{eq:matrix.DE.accr.regularity} be satisfied. Let $\Theta$ and $\Phi$ be defined as in~\eqref{eq:matrix.DE.Theta.Phi} and let $\Aa_0$ and $S_0 (\cdot)$ be as in~\eqref{eq:matrix.DE.Aa.S} and~\eqref{eq:matrix.DE.max.dom}, where $\boldpi$ is as in \eqref{eq:def.boldpi}, and we set $\lambda_0\defeq-1$, $q_0 \defeq q$ in the definitions~\eqref{eq:matrix.DE.D_S.iprod},~\eqref{eq:matrix.DE.D_2.def}. Then \eqref{eq:matrix.DE.spec.equivalent}--\eqref{eq:matrix.DE.pspec.equiv} hold and $-\Aa_0$ is m-accretive, thus $\Aa_0$ generates a strongly continuous contraction semigroup on $L^2 (\Omega) \oplus L^2 (\Omega)$. Its domain is dense in $L^2 (\Omega) \oplus L^2 (\Omega)$ and satisfies
	\begin{equation}
		\label{eq:matrix.DE.accretive.domain}
		\pi_2 \dom \Aa_0 \subset \dom |d|^\frac12.
	\end{equation}
\end{thm}

\subsection{Realisation of matrix and Schur complement}
\label{sec:matrix.DE.constr}
In the more general setting of Assumption~\ref{asm:IV}, we provide the appropriate framework to the spectral problem for the operator matrix~\eqref{eq:matrix.DE} such that the results of Section~\ref{sec:matrix.Schur.complement} apply. We therefore need to define the objects in Assumption~\ref{asm:II} in a suitable way.

\subsubsection{Definition of spaces}

We introduce the spaces needed for Assumption~\ref{asm:II}~(i).

\begin{defi}
	\label{def:matrix.DE.spaces}
	Let Assumption~\ref{asm:IV} be satisfied and let $\Hh_1 \defeq \Hh_2 \defeq L^2 (\Omega)$. Moreover, let $\Dd_2$ and $\omega$ be as in \eqref{eq:matrix.DE.D_2.def} and let $\Dd_S$ and $\Dd_1$, respectively, be the closure of $\Coo (\Omega)$ with respect to the inner products in \eqref{eq:matrix.DE.D_S.iprod} and
	\begin{equation}\label{eq:matrix.DE.norm.Dd.1}
		\begin{aligned}
			\iprod{f}{g}_1 & \defeq \int_\Omega |d-\lambda_0|^{-2} \omega^2  (\mathbf b \otimes \overline{\mathbf{b}})  \nabla f \cdot \overline{\nabla g}   \d x  \\
			&  \qquad \quad \qquad \qquad \qquad   + \int_{\Omega} \Delta f \overline{\Delta g} \d x + \iprod{f}{g}_S , \qquad  f, g \in \Coo (\Omega);
		\end{aligned}
	\end{equation}
	notice that $(\mathbf b \otimes \overline{\mathbf{b}}) \ge 0$ is positive semi-definite a.e.\ in $\Omega$. Finally, define
	\begin{equation}
		\Dd_{-S} \defeq \Dd_S^*, \qquad \Dd_{-1} \defeq  \Dd_1^*, \qquad \Dd_{-2} \defeq L^2 (\Omega, \omega^{-2}).\tag*{//}
	\end{equation}
\end{defi}

\begin{prop}
	\label{prop:matrix.DE.spaces}
	Under Assumption~{\rm \ref{asm:IV}}, the spaces in Definition {\rm\ref{def:matrix.DE.spaces}} are well-defined and satisfy Assumption~{\rm \ref{asm:II}~(i)}.
\end{prop}

\begin{proof}
	By assumption, $\re \boldpi_0 \in \Loneloc (\Omega)^{n \times n}$ is positive definite and $\re q_0 \in \Loneloc (\Omega)$ is non-negative a.e.\ in $\Omega$. Considering this, it is easy to see that $\Dd_S$ is a well-defined Hilbert space. Moreover, from $d \in L^\infty_{\operatorname{loc}} (\Omega)$ and Assumption~\ref{asm:IV}~(iv) it follows that
	\begin{equation}\label{eq:matrix.DE.omega.loc.bdd}
		(\re \boldpi_0)^{-\frac12} \overline c = (d - \lambda_0) (d - \lambda_0)^{-1} (\re \boldpi_0)^{-\frac12} \overline c \in L^\infty_{\operatorname{loc}} (\Omega)^n,
	\end{equation}
	which in turn implies $\omega \in L^\infty_{\operatorname{loc}} (\Omega)$. From $\boldpi_0 \in \Loneloc (\Omega)^{n\times n}$ and Assumption~\ref{asm:IV}~(iv), we further derive that
	\begin{equation}\label{eq:matrix.DE.rem.D1.norm.L2loc}
		\omega (d - \lambda_0)^{-1} \mathbf b = \omega (\re \boldpi_0)^{\frac12} (d - \lambda_0)^{-1} (\re \boldpi_0)^{-\frac12} \mathbf b \in L^2_{\operatorname{loc}} (\Omega)^n,
	\end{equation}
	and thus also $\Dd_1$ is a well-defined Hilbert space.
	
	Analogously to the proof of Proposition~\ref{prop:DWE.spaces}, from the density of $\Coo (\Omega)$ in $L^2 (\Omega)$ and since $\norm{\cdot} \le \norm{\cdot}_S$ by construction, we obtain the inclusions
	\begin{equation}
		\Dd_S \subset L^2 (\Omega) \subset \Dd_S^*
	\end{equation}
	where the corresponding embeddings are continuous and have dense range. Similarly, it follows that $\Dd_1 \subset \Dd_S$, and thus $\Dd_S^* \subset \Dd_1^*$, are dense and continuously embedded. 

	Clearly, the weighted spaces $\Dd_2$ and $\Dd_{-2}$ are Hilbert spaces. From $\lambda_0 \notin \essran d$ and Assumption~\ref{asm:IV}~(iv), it follows that
	\begin{equation}
		\abs{d - \lambda_0}^2 \omega^{-2} = \min\left(\abs{d - \lambda_0}^2, \abs{(d-\lambda_0)^{-1} (\re \boldpi_0)^{-\frac12} \overline \formc}^{-2}\right) \gtrsim 1,
	\end{equation}
	thus $\Dd_2\subset L^2 (\Omega)$ is continuously embedded. Moreover, our assumptions imply
		\begin{equation}\label{eq:matrix.DE.Dd2.dense}
			|d-\la_0|^2\omega^{-2} \in L^\infty_{\rm loc} (\Omega),
		\end{equation}
	and it follows that $\Dd_2$ contains $\Coo(\Omega)$ and  is therefore dense in $L^2(\Omega)$. Since $\omega \ge 1$, also the embedding $L^2 (\Omega) \subset \Dd_{-2}$ is continuous and, as the weighted Lebesgue measure $\omega^{-2} \d\lambda^n$ is a Borel measure due to the essential boundedness of $\omega^{-2}$, the embedding also has dense range,  see~\cite[Ex.~1.5.3~(c)]{Blank-Exner-Havlicek-2008}.
\end{proof}

\begin{rem} \label{rem:matrix.DE.distr.spaces}
	Due to Assumption~{\rm \ref{asm:IV}}, the spaces $\Dd_1^*$ and $\Dd_{-2}$ in Definition~{\rm\ref{def:matrix.DE.spaces}} can both be understood as subspaces of $\Dd'(\Omega)$, cf.\ Remark~\ref{rem:DWE.distr.spaces}. Indeed, since $\omega$ is locally bounded by~\eqref{eq:matrix.DE.omega.loc.bdd}, we have
	\begin{equation}
		\Dd_{-2} = L^2 (\Omega, \omega^{-2}) \subset L^2_{\rm loc} (\Omega) \subset \Dd'(\Omega).
	\end{equation}
	The embedding of $\Dd_1^*$ in $\Dd'(\Omega)$ can be shown analogously to $\Dd_S^* \subset \Dd'(\Omega)$ in Remark~\ref{rem:DWE.distr.spaces}. \hfill //
\end{rem}

\subsubsection{Definition of matrix entries}

We introduce the operators needed for Assumption~\ref{asm:II}~(ii) and (iii).

\begin{defi}
	\label{def:matrix.DE.operators}
	Let Assumption~{\rm \ref{asm:IV}} be satisfied. Define the following operators
	\begin{equation}
		\begin{array}{rcl}
			A \defeq & - \Delta + q & \in \Bb (\Dd_S, \Dd_1^*), \\[1mm]
			C \defeq & \formc \cdot \nabla f & \in \Bb (\Dd_S, \Dd_{-2}),
		\end{array} \quad
		\begin{array}{rcl}
			B \defeq & \nabla \cdot \formb & \in \Bb (\Dd_2, \Dd_1^*), \\[1mm]
			D \defeq & d & \in \Bb (\Dd_2, \Dd_{-2}),
		\end{array}
	\end{equation}
	see Definition~\ref{def:matrix.DE.spaces} for the spaces involved. Here $A$ and $B$ are determined uniquely by the identities
	\begin{equation}
		\label{eq:matrix.DE.AB}
		\begin{aligned}
			\dprod{A f}{g}_{\Dd_1^* \times \Dd_1} & \defeq \int_\Omega \nabla f \cdot \overline{\nabla g} \d x + \int_\Omega q f \overline g \d x, && \quad \!\!\! f \in \Coo (\Omega), & g \in \Coo (\Omega), \\
			\dprod{B f}{g}_{\Dd_1^* \times \Dd_1} & \defeq - \int_\Omega  \formb f \cdot \overline{\nabla g} \d x, && \quad \!\!\! f \in \Dd_2, & g\in \Coo (\Omega),
		\end{aligned}
	\end{equation}
	and the operator $C$ is the unique extension in $\Bb (\Dd_S, \Dd_{-2})$ of
	\begin{equation}
		\label{eq:matrix.DE.C}
		C f \defeq \formc \cdot \nabla f, \qquad f \in \Coo (\Omega),
	\end{equation}
	see Proposition~\ref{prop:matrix.DE.operators} below for details. Finally, define $D_0 \defeq D\vert_{\dom d}$ as the maximal multiplication operator by $d$ in $L^2 (\Omega)$. \hfill //
\end{defi}

\begin{prop}
	\label{prop:matrix.DE.operators}
	Let Assumption~{\rm \ref{asm:IV}} hold. Then the operators introduced in Definition~{\rm\ref{def:matrix.DE.operators}} are well-defined and satisfy Assumption~{\rm \ref{asm:II}~(ii) and (iii)}. Moreover,
	\begin{equation}
		\label{eq:matrix.DE.D.invertible}
		(D - \lambda)^{-1} \in \Bb (\Dd_{-2}, \Dd_2), \qquad \lambda \in \Theta,
	\end{equation}
	where $\Theta \subset \C \setminus \essran d$ is the connected set in Assumption~{\rm \ref{asm:IV}~(ii)}.
\end{prop}

The following elementary lemma will be needed several times, including for the proof of \eqref{eq:matrix.DE.D.invertible} in Proposition~\ref{prop:matrix.DE.operators}.

\begin{lem}
	\label{lem:matrix.DE.d.equiv}
	Let $\Theta \subset \C \setminus \essran d$ be connected and $\lambda, \mu \in \Theta$. Then there exist constants $C_{\lambda, \mu}, C_{\lambda, \mu}' > 0$ such that a.e.\ in $\Omega$
	\begin{equation}
		C_{\lambda, \mu} \abs{d - \mu} \le \abs{d - \lambda} \le C_{\lambda,\mu}' \abs{d - \mu}.
	\end{equation}
\end{lem}

\begin{proof}
	For fixed $\la, \mu \in \Theta$, clearly a.e.~in $\Omega$ it holds that
	\begin{equation}
		\frac{|d-\mu|}{|d-\la|} \le 1 + \frac{|\la - \mu|}{|d - \la|} \le 1 + \frac{|\la-\mu|}{\dist (\la, \essran d)} =: C_{\la, \mu}^{-1}
	\end{equation}
	and it follows that the claimed inequalities are satisfied with $C'_{\la, \mu} := C_{\mu,\la}^{-1}$.
\end{proof}

\begin{proof}[Proof of Proposition {\rm\ref{prop:matrix.DE.operators}}]
	Integration by parts and sectoriality of $q_0$ give
	\begin{equation}
		\label{eq:matrix.DE.A.bounded}
		\abs{\dprod{A f}{g}_{\Dd_1^* \times \Dd_1}} \lesssim \norm{f} \norm{\Delta g} + \norm{(\re q_0)^{\frac12} f} \norm{(\re q_0)^{\frac12} g} + \norm{f} \norm{g} \le \norm{f}_S \norm{g}_1
	\end{equation}
	for $f,g \in \Coo (\Omega)$, see Assumption~\ref{asm:IV}~(iii). Analogously to the proof of Proposition~\ref{prop:DWE.entries}, this implies that $A \in \Bb (\Dd_S, \Dd_1^*)$ is uniquely well-defined by \eqref{eq:matrix.DE.AB}. Similarly, for $f \in \Dd_2$ and $g \in \Coo (\Omega)$,
	\begin{equation}
		\abs{\dprod{B f}{g}_{\Dd_1^* \times \Dd_1}} \le \norm{f (d - \lambda_0) \omega^{-1}} \norm{(d - \lambda_0)^{-1} \omega  \overline \formb \cdot \nabla g} \le \norm{f}_{\Dd_2} \norm{g}_1
	\end{equation}
	implies that $B \in \Bb (\Dd_2, \Dd_1^*)$ is uniquely well-defined by \eqref{eq:matrix.DE.AB}. For any $f \in \Coo (\Omega)$, we moreover obtain the chain of inequalities
	\begin{equation}
		\label{eq:matrix.DE.C.bounded}
		\begin{aligned}
			\norm{Cf}_{\Dd_{-2}}^2  & \le \int_\Omega \abs{\iprod{\overline \formc}{\nabla f}_{\Cn}}^2 \omega^{-2} \d x \\
			& \le \int_\Omega \abs{(\re \boldpi_0)^{-\frac12} \overline \formc}^2 \abs{(\re \boldpi_0)^{\frac12} \nabla f}^2 \omega^{-2} \d x \\[1.5mm]
			& \le \norm{\omega^{-1} (\re \boldpi_0)^{-\frac12} \overline \formc}_{L^\infty (\Omega)^n}^2 \norm{f}_S^2;
		\end{aligned}		
	\end{equation}
	notice that the right hand side of the above inequality is finite by definition of~$\omega$. It now follows from the density of $\Coo (\Omega)$ in $\Dd_S$ that the operator defined in \eqref{eq:matrix.DE.C} has a unique extension $C \in \Bb (\Dd_S, \Dd_{-2})$.
	We next show that $D$ and $D_0$ satisfy Assumption~\ref{asm:II}~(ii). From the definition of $\Dd_2$ and $\Dd_{-2}$, it is obvious that $D \in \Bb (\Dd_2, \Dd_{-2})$ and, as maximal multiplication operator in $L^2 (\Omega)$, its restriction $D_0$ is closed in $L^2 (\Omega)$. Moreover, since the weighted Lebesgue measure $|d-\la_0|^2\omega^{-2} \d\la^n$ is a Borel measure due to~\eqref{eq:matrix.DE.Dd2.dense} and since $d \in L^\infty_{\rm loc} (\Omega)$, it follows that $\Coo(\Omega)$ is contained in $\dom d$ and the latter is consequently dense in $\Dd_{-2}$, see~\cite[Ex.~1.5.3~(c)]{Blank-Exner-Havlicek-2008}. Finally, for fixed $\lambda \in\Theta$, relation \eqref{eq:matrix.DE.D.invertible} holds since by Lemma~\ref{lem:matrix.DE.d.equiv} with $\mu = \lambda_0$ for all $f \in \Dd_{-2}$ we have
	\begin{equation*}
		\norm{(d - \lambda)^{-1} f}_{\Dd_2}^2 = \int_\Omega \abs{d - \lambda}^{-2} \abs{f}^2 \abs{d - \lambda_0}^2 \omega^{-2} \d x \lesssim \int_\Omega \abs{f}^2 \omega^{-2} \d x = \norm{f}_{\Dd_{-2}}^2. \qedhere
	\end{equation*}
\end{proof}

\subsubsection{Definition of matrix and Schur complement}

Analogously to Definition~\ref{def:matrix.Schur.compl.family}, we now introduce our realisation of the operator matrix \eqref{eq:matrix.DE} and its Schur complement. Note that, by Proposition~\ref{prop:matrix.DE.operators}, the set $\Theta$ satisfies inclusion \eqref{eq:Theta}.

\begin{defi}
	\label{def:matrix.DE.Aa.S}
We define the operator matrix
	\begin{equation}
		\Aa \defeq \left(
		\begin{array}{cc}
			- \Delta + q & \nabla \cdot \formb \\
			\formc \cdot \nabla & d
		\end{array}
		\right) \in \Bb (\Dd_S \oplus \Dd_2, \Dd_1^* \oplus \Dd_{-2})
	\end{equation}
	and its first Schur complement
	\begin{equation}
		S (\lambda) \defeq -\Delta + q - \lambda - (\nabla \cdot \formb) (d - \lambda)^{-1} (\formc \cdot \nabla) \in \Bb (\Dd_S, \Dd_1^*), \qquad \lambda \in \Theta,
	\end{equation}
	see Definitions~\ref{def:matrix.DE.spaces} and \ref{def:matrix.DE.operators} for the spaces and operators involved. Let the (family of) maximal operators $\Aa_0$ in $L^2 (\Omega) \oplus L^2 (\Omega)$ and $S_0 (\cdot)$ in $L^2 (\Omega)$, respectively, be defined as the restrictions of $\Aa$ and $S(\cdot)$ to their respective maximal domains
		\begin{align}\label{eq:matrix.DE.def.dom.Aa.S}
			\dom \Aa_0 & \defeq \set{(f, g) \in \Dd_S \times \Dd_2}{\Aa (f,g) \in L^2 (\Omega) \times L^2 (\Omega)}, \\
			\dom S_0 (\lambda) & \defeq \set{f \in \Dd_S}{S(\lambda) f \in L^2 (\Omega)}. \tag*{//}
		\end{align}
\end{defi}

The proposition below shows that the Schur complement acts according to the formula obtained from naively integrating by parts. On the set of parameters $\Phi$ where it is sectorial, our definition of $S(\la)$ coincides with its standard definition by means of its quadratic form with form domain $\Dd_S$ and core $\Coo (\Omega)$ and is bounded between $\Dd_S$ and its anti-dual $\Dd_S^*$.

\begin{prop}
	\label{prop:matrix.DE.S.bounded}
	Let Assumption~{\rm \ref{asm:IV}} be satisfied and let $S(\cdot)$ be as in Definition~{\rm\ref{def:matrix.DE.Aa.S}}. Then, for all $\lambda \in \Theta$ and $f, g \in \Coo (\Omega)$,
	\begin{equation}
		\label{eq:matrix.DE.S.formula}
		\dprod{S (\lambda) f}{g}_{\Dd_1^* \times \Dd_1} = \int_\Omega \boldpi (\lambda) \nabla f \cdot \overline{\nabla g} \d x + \int_\Omega q f \overline{g} \d x - \lambda \int_\Omega f \overline{g} \d x.
	\end{equation}
	Moreover, if $\lambda \in \Phi$, then $S (\lambda) \in \Bb (\Dd_S, \Dd_S^*)$.
\end{prop}

\begin{proof}
	Fix $\lambda \in \Theta$ and let $f, g \in \Coo (\Omega)$. Since $(d - \lambda)^{-1} \formc \cdot \nabla f \in \Dd_2$, it follows by definition that
	\begin{equation}
		\begin{aligned}
			\dprod{S (\lambda) f}{g}_{\Dd_1^* \times \Dd_1} & = \int_\Omega \nabla f \cdot \overline{\nabla g} \d x + \int_\Omega q f \overline{g} \d x - \lambda \int_\Omega f \overline{g} \d x \\
			& \qquad \qquad \qquad \qquad \qquad + \int_\Omega (d - \lambda)^{-1} (\formc \cdot \nabla f) (\formb \cdot \overline{\nabla g}) \d x. \\
		\end{aligned}
	\end{equation}
	Formula~\eqref{eq:matrix.DE.S.formula} is then easily derived from the definition of $\boldpi (\lambda)$ and 
	\begin{equation}
		(\formc \cdot \xi) (\formb \cdot \overline{\eta}) = (\formb \otimes \formc)  \xi \cdot \overline{\eta}, \qquad \xi, \eta \in \Cn.
	\end{equation}
	Now let $\lambda \in \Phi$ be fixed. Using formula \eqref{eq:matrix.DE.S.formula}, multiplying it by the unimodular factor $\e^{\i \omega_\lambda}$ and employing Assumption~\ref{asm:IV}~(iii), we conclude
	\begin{equation}
		\abs{\dprod{S (\lambda)f}{f}_{\Dd_1^* \times \Dd_1}} \lesssim \int_\Omega 
		\re \widetilde \boldpi (\lambda) \nabla f \cdot \overline{\nabla f} \d x + \int_\Omega \re \widetilde q (\lambda) \abs{f}^2 \d x + \int_\Omega \abs{f}^2 \d x \lesssim \norm{f}_S^2
	\end{equation}
	for $f \in  \Coo (\Omega)$. By a straightforward polarisation argument, it follows that
	\begin{equation}
		\label{eq:matrix.DE.S.bounded}
		\abs{\dprod{S (\lambda)f}{g}_{\Dd_1^* \times \Dd_1}} \lesssim \norm{f}_S \norm{g}_S , \quad f, g \in \Coo (\Omega).
	\end{equation}
	Since $\Coo (\Omega)$ is dense both in $\Dd_1$ and in $\Dd_S$, $S (\lambda) \in \Bb (\Dd_S, \Dd_1^*)$ and $\Dd_1 \subset \Dd_S$ is continuously embedded, the inequality \eqref{eq:matrix.DE.S.bounded} remains valid for $f \in \Dd_S$ and $g \in \Dd_1$. However, this implies $S (\lambda) \in \Bb (\Dd_S, \Dd_S^*)$ for all $\lambda \in \Phi$.
\end{proof}

\begin{rem} \label{rem:matrix.DE.distr.operators}
	The actions of $\Aa$ and $S(\cdot)$ can be understood in a standard distributional sense, cf.\ Remark~\ref{rem:matrix.DE.distr.spaces} and also Remark~\ref{rem:DWE.distr.act} for the damped wave equation. More precisely, by Assumption~\ref{asm:IV}~(i), the distributions
	\begin{equation}\label{eq:matrix.DE.A.S.distr.Coo}
		-\Delta f + qf \in \Loneloc (\Omega), \quad \nabla \cdot \boldpi (\lambda) \nabla f + q f - \lambda f \in \Dd'(\Omega), \qquad f \in \Coo(\Omega),
	\end{equation}
	are well-defined and coincide with the functionals $Af \in \Dd_1^*$ and $S(\lambda) f \in \Dd_1^*$, respectively, see definitions of the latter in~\eqref{eq:matrix.DE.AB} and formula~\eqref{eq:matrix.DE.S.formula}. Moreover, 
	\begin{equation}\label{eq:matrix.DE.C.distr.Coo}
		Cf = \mathbf c \cdot \nabla f \in \Dd_{-2} \subset \Loneloc(\Omega), \qquad f \in \Coo(\Omega).
	\end{equation}
	Since the actions of $A$, $C$ and $S(\lambda)$ on $\Dd_S$ are obtained by continuous extension, see~\eqref{eq:matrix.DE.AB} and~\eqref{eq:matrix.DE.C}, they are given as limits of distributions of the form~\eqref{eq:matrix.DE.A.S.distr.Coo} and~\eqref{eq:matrix.DE.C.distr.Coo}; notice that the convergence of functionals in $\Dd_1^*$ or $\Dd_{-2}$ (which can be understood as $L^2 (\Omega, w^2)^*$) implies their convergence in $\Dd'(\Omega)$. From~\eqref{eq:matrix.DE.rem.D1.norm.L2loc}, we see that
	\begin{equation}
		\mathbf b g =  \omega (d - \lambda_0)^{-1} \mathbf b g (d - \lambda_0) \omega^{-1} \in \Loneloc (\Omega)^n, \quad g \in \Dd_2 = L^2 (\Omega, |d - \lambda_0|^2 \omega^{-2}).
	\end{equation}
	Hence, the distributional divergence $\nabla \cdot \mathbf b g \in \Dd'(\Omega)$ is well-defined and clearly coincides with $Bg \in \Dd_1^*$ as in~\eqref{eq:matrix.DE.AB}. Finally, it is clear that $D = d$ is a standard multiplication operator between the weighted spaces $\Dd_2$ and $\Dd_{-2}$. \hfill //
\end{rem}

\subsection{Proof of Theorem~{\rm\ref{thm:matrix.DE}}}
\label{sec:DE.thms.proofs}

The statement of Theorem~\ref{thm:matrix.DE} can be obtained from the results in Section~\ref{sec:matrix.Schur.complement}; the following lemma is needed in order to apply Corollaries~\ref{cor:family.first.implication} and~\ref{cor:family.second.implication} therein.

\begin{lem}
	\label{lem:matrix.DE.S.coercive}
	Let Assumption~{\rm \ref{asm:IV}} be satisfied and $S_0 (\cdot)$ as in Definition {\rm\ref{def:matrix.DE.Aa.S}}. Then, for every $\lambda \in \Phi$, there exists $z_\lambda \in \rho (S_0 (\lambda))$ such that
	\begin{equation}
		(S (\lambda) - z_\lambda)^{-1}  \in \Bb (\Dd_S^*, \Dd_S)
	\end{equation}
	and $\dom S_0 (\lambda) = \dom (S_0(\lambda) - z_\lambda)$ is dense in $\Dd_S$.
\end{lem}

\begin{proof}
	Let $\lambda \in \Phi$ be arbitrary but fixed, then $S (\lambda) \in \Bb (\Dd_S, \Dd_S^*)$ by Proposition~\ref{prop:matrix.DE.S.bounded}. Moreover, multiplying \eqref{eq:matrix.DE.S.formula} with the unimodular factor $\e^{\i \omega_\lambda}$, we obtain
	\begin{equation}
	\begin{aligned}
		\re \dprod{\e^{\i\omega_\lambda}S (\lambda) f}{f}_{\Dd_S^* \times \Dd_S} & \ge \int_{\Omega} \re \widetilde \boldpi (\lambda) \nabla f \cdot \overline{\nabla f} \d x \\
		& \qquad \qquad \qquad + \int_{\Omega} \re \widetilde q (\lambda) |f|^2 \d x - |\gamma_\lambda + \e^{\i \omega_\lambda}\lambda| \norm{f}^2 \\[1.5mm]
	\end{aligned}
	\end{equation}
	for $f \in \Coo (\Omega)$. Using \eqref{eq:matrix.DE.uniform.positivity}, we further derive
	\begin{equation}
		\re \dprod{\e^{\i\omega_\lambda}S (\lambda) f}{f}_{\Dd_S^* \times \Dd_S} \ge C_1 \norm{f}_S^2 - C_2 \norm{f}^2, \qquad f \in\Coo(\Omega),
	\end{equation}
 	where $C_1, C_2 > 0$ depend on $\lambda$ but not on $f$. From this it is easy to see that there exists $z_\lambda \in \C$ such that
 	\begin{equation}\label{eq:matrix.DE.S.coercive}
 		\abs{\e^{\i\omega_\lambda} \dprod{(S (\lambda) -z_\lambda)f}{f}_{\Dd_S^* \times \Dd_S}} \gtrsim \norm{f}^2_S, \qquad f\in\Coo(\Omega).
 	\end{equation}
 	By the density of $\Coo (\Omega)$ in $\Dd_S$, the continuity of the embedding $\Dd_S \subset \Dd_S^*$ and $S (\lambda) \in \Bb (\Dd_S, \Dd_S^*)$, the above inequality remains valid for $f \in \Dd_S$, i.e.\ $S (\lambda) - z_\lambda$ corresponds to a bounded and coercive sesquilinear form on $\Dd_S$. By the Lax-Milgram Theorem, see e.g.\ \cite[Thm.\ IV.1.1]{Edmunds-Evans-1987}, we conclude $z_\lambda \in \rho(S_0 (\lambda))$,
	\begin{equation}
		(S (\lambda) - z_\lambda)^{-1} \in \Bb (\Dd_S^*, \Dd_S)
	\end{equation}
	and density of the maximal domain $\dom (S_0 (\lambda) - z_\lambda) = \dom S_0 (\lambda)$ in $\Dd_S$.
\end{proof}

\begin{proof}[Proof of Theorem {\rm~\ref{thm:matrix.DE}}]
	By Propositions~\ref{prop:matrix.DE.spaces} and~\ref{prop:matrix.DE.operators}, Assumption~\ref{asm:II} is satisfied, the objects in Definition~\ref{def:matrix.DE.Aa.S} are well-defined and the results of Section~\ref{sec:matrix.Schur.complement} are applicable. We point out that our realisations of $\Aa_0$ and $S_0(\cdot)$ in Definition~\ref{def:matrix.DE.Aa.S} coincide with their standard distributional definitions in~\eqref{eq:matrix.DE.Aa.S}, see Remark~\ref{rem:matrix.DE.distr.operators}. The description of their domains in~\eqref{eq:matrix.DE.max.dom} then follows from their definition in~\eqref{eq:matrix.DE.def.dom.Aa.S}.
	We show the claims in \eqref{eq:matrix.DE.spec.equivalent}--\eqref{eq:matrix.DE.pspec.equiv}. Since $\Theta \subset \C \setminus \essran d = \rho (D_0)$, the identities~\eqref{eq:matrix.DE.pspec.equiv} and~\eqref{eq:matrix.DE.spec.equivalent} with $\sigma_{\operatorname{p}}$ are a direct consequence of Corollary~\ref{cor:point.spectrum.family}. Let us proceed with using Corollary~\ref{cor:family.second.implication} in order to show that the left hand side of~\eqref{eq:matrix.DE.spec.equivalent}, both with $\sigma$ and $\sigma_{\operatorname{e2}}$, is included in its right hand side.	To this end, let $f \in \dom d$ and $g \in \Coo (\Omega)$. We can estimate
	\begin{equation}
	\label{eq:matrix.DE.B.dom.D}
		\begin{aligned}
			\abs{\dprod{B f}{g}_{\Dd_1^* \times \Dd_1}} & \le \int_{\Omega} |f| |\iprod{\mathbf b}{\nabla g}_{\Cn}| \d x \\[1mm]
			& \le \int_\Omega |(d - \lambda_0)f| |d - \lambda_0|^{-1} |(\re \boldpi_0)^{-\frac12} \formb||(\re \boldpi_0)^{\frac12} \nabla g| \d x \\[2mm]
			& \le \norm{(d- \lambda_0)f} \norm{(d- \lambda_0)^{-1} (\re \boldpi_0)^{-\frac12} \formb}_{L^\infty (\Omega)^n} \norm{g}_S;
		\end{aligned}
	\end{equation}
	notice that the right hand side of the last inequality is finite due to Assumption~\ref{asm:IV}~(iv). By the density of $\Coo (\Omega)$ in $\Dd_1$ and since $\Dd_1 \subset \Dd_S$ is continuously embedded, both left and right hand side of \eqref{eq:matrix.DE.B.dom.D} are continuous in $g$ with respect to $\norm{\cdot}_1$. Hence,~\eqref{eq:matrix.DE.B.dom.D} holds for all $f \in \dom d$ and $g \in \Dd_1$, implying $B (\dom d) \subset \Dd_S^*$. The remaining assumptions of Corollary~\ref{cor:family.second.implication} with $\Sigma = \Phi \subset \Theta$ are a consequence of $\Phi \subset \rho (D_0)$, Proposition~\ref{prop:matrix.DE.S.bounded}, Lemma~\ref{lem:matrix.DE.S.coercive} and
	\begin{equation}
		(S_0(\lambda)-z_\lambda)^{-1} \subset S_z^\ddagger(\lambda) \defeq (S(\lambda)-z_\lambda)^{-1} \in \Bb (\Dd_S^*,\Dd_S).
	\end{equation}
	The claimed inclusion in~\eqref{eq:matrix.DE.spec.equivalent}, i.e.~left hand side in right hand side, with $\sigma$ and $\sigma_{\operatorname{e2}}$, respectively, follows from Corollary~\ref{cor:family.second.implication} (i) and (ii).
	The reverse inclusion in~\eqref{eq:matrix.DE.spec.equivalent} with $\sigma$ and~\eqref{eq:matrix.DE.spec.first.incl} are obtained from $\Theta \subset \rho (D_0)$ and Corollary~\ref{cor:family.first.implication}~(i). The remaining inclusion in~\eqref{eq:matrix.DE.spec.equivalent} with $\sigma_{\rm e2}$
	follows from Corollary~\ref{cor:family.first.implication} (ii) with $\Sigma = \Phi$; here it suffices to note that $\Phi \subset \rho (D_0)$ and that, for every $\lambda \in \Phi$, $S_0 (\lambda)$ has non-empty resolvent set by Lemma~\ref{lem:matrix.DE.S.coercive} and is therefore closed in $L^2 (\Omega)$. The claims \eqref{eq:matrix.DE.spec.equivalent}--\eqref{eq:matrix.DE.pspec.equiv} are thus proven.
	Finally, assume that $\dom B_0 \cap \dom d$ is dense in $L^2 (\Omega)$. Then Lemma~\ref{lem:matrix.DE.S.coercive} with $\lambda = \lambda_0$ implies that $\dom S_0 (\lambda_0)$ is dense in $\Dd_S$. Consequently, since $\Dd_S$ is dense and continuously embedded in $L^2 (\Omega)$, it follows from Corollary~\ref{cor:family.dense.domain} that $\dom \Aa_0$ is dense in $L^2 (\Omega) \oplus L^2 (\Omega)$.
\end{proof}

\subsection{Proof of Theorem~{\rm\ref{thm:matrix.DE.accretive}}}
\label{sec:matrix.DE.accr.proof}
In preparation for the proof of Theorem~\ref{thm:matrix.DE.accretive}, we show that the hypotheses of the latter imply Assumption~\ref{asm:IV}.

\begin{lem}
	\label{lem:matrix.DE.accretive}
	Let the assumptions of Theorem{\rm~\ref{thm:matrix.DE.accretive}} be satisfied. Then Assumption~{\rm\ref{asm:IV}} holds with $\Theta$, $\Phi$ as in \eqref{eq:matrix.DE.Theta.Phi}.
\end{lem}

\begin{proof}
	We only need to show Assumption~\ref{asm:IV}~(ii)--(iv). However, (ii) is clearly satisfied since the set $\Theta \subset \C \setminus \essran d$ in~\eqref{eq:matrix.DE.Theta.Phi} is connected and~\eqref{eq:def.boldpi} follows from the first assumption in~\eqref{eq:matrix.DE.accr.regularity}, see Remark~\ref{rem:matrix.DE.ass}~(iv). Recall that in this case
	\begin{equation}
		\boldpi (\lambda) = \mathbf I + (d - \lambda)^{-1} \mathbf B, \qquad \lambda\in \Theta,
	\end{equation}
	where $\mathbf B \defeq (\mathbf b \otimes \overline{\mathbf b}) \ge 0$ is easily verified to be positive semi-definite a.e.\ in $\Omega$.
	
	For (iii), let $\lambda \in \Phi \subset \Theta$ be arbitrary but fixed. Set $\omega_\lambda \defeq 0$ if $|\arg\lambda| > \pi / 2$ and
	\begin{equation}
		\omega_\lambda \defeq - \frac{\sgn (\arg \lambda)}{2} (\pi - \max (\theta_q, \theta_d) - |\arg \lambda|), \qquad |\arg \lambda| \le \pi / 2,
	\end{equation}
	where the above is well-defined since $\lambda \neq 0 \notin \Phi$; recall that $\arg : \C\setminus\{0\} \to (-\pi,\pi]$ in our convention and notice that $|\omega_\lambda| < \pi/2 - \max (\theta_q, \theta_d) \le \pi / 2$.
	It then follows from elementary geometrical considerations that
	\begin{equation}
		\widetilde d (\lambda) \defeq \e^{\i \omega_\lambda}(d - \lambda)^{-1}, \qquad \widetilde q (\lambda) \defeq \e^{\i \omega_\lambda} q,
	\end{equation}
	are both sectorial; in particular, we can set $\gamma_\lambda \defeq 0$ in view of Assumption~\ref{asm:IV}~(iii). Moreover, by the sectoriality of $\widetilde d (\lambda)$ and $|\omega_\lambda| < \pi/2$, also
	\begin{equation}\label{eq:matrix.DE.tilde.pi.accr}
		\widetilde \boldpi (\lambda) \defeq \e^{\i \omega_\lambda} \boldpi (\lambda) = \e^{\i \omega_\lambda} \mathbf I + \widetilde d (\lambda) \mathbf B
	\end{equation}
	is sectorial and its real part is positive definite a.e.\ in $\Omega$ since we have
	\begin{equation}
		\label{eq:matrix.DE.accretive.pi.lambda.positive.definite}
		\re \widetilde \boldpi (\lambda) \ge \cos \omega_\lambda \mathbf {I} > 0.
	\end{equation}
	It remains to show~\eqref{eq:matrix.DE.uniform.positivity}. The first chain of inequalities therein, i.e.\ the equivalence of $\re \widetilde q (\lambda)$ for different $\lambda\in \Phi$, is derived from the sectoriality of $q$ and the relation
	\begin{equation}\label{eq:matrix.DE.accr.re.q}
		\re \widetilde q(\lambda) = \re q \cos \omega_\la - \im q \sin \omega_\la.
	\end{equation}
	Since $|\omega_\la| < \pi/2 - \theta_q$, it indeed follows that 
	\begin{equation}
		\cos \omega_\la - \tan \theta_q |\sin \omega_\la| >0,
	\end{equation}
	such that, with another $\mu \in \Phi$, one further concludes using~\eqref{eq:matrix.DE.accr.re.q} that
	\begin{equation}
		\frac{\re \widetilde q(\mu)}{\re \widetilde q(\la)} \le \frac{\cos \omega_\mu + \tan \theta_q |\sin \omega_\mu|}{\cos \omega_\la - \tan \theta_q |\sin \omega_\la|}.
	\end{equation}
	This readily implies the sought chain of inequalities. Employing Lemma~\ref{lem:matrix.DE.d.equiv}, one obtains the analogous inequalities for $\widetilde d(\lambda)$ in a similar way. The second line in~\eqref{eq:matrix.DE.uniform.positivity} then follows easily from the decomposition~\eqref{eq:matrix.DE.tilde.pi.accr} and $|\omega_\lambda| < \pi /2$.
	
	Finally, Assumption~\ref{asm:IV}~(iv) follows from the second assumption in~\eqref{eq:matrix.DE.accr.regularity}, see Remark~\ref{rem:matrix.DE.ass}~(iv).
\end{proof}

By the lemma above, Theorem~\ref{thm:matrix.DE} is applicable in the present setting and can be used in order to prove Theorem~\ref{thm:matrix.DE.accretive}. Before doing so, we first describe $\Dd_S$ and certain actions defined on it.

\begin{lem}\label{lem:matrix.DE.accr.Dd.S}
	Let the assumptions of Theorem~{\rm\ref{thm:matrix.DE.accretive}} be satisfied. Then
	\begin{equation}
		\label{eq:matrix.DE.Dd.S.accretive}
		\Dd_S = \set{f \in H_0^1 (\Omega)}{(\re \boldpi_0)^{\frac12} \nabla f \in L^2 (\Omega)^n, ~ (\re q)^\frac12 f \in L^2 (\Omega)}.
	\end{equation}
	The formulas \eqref{eq:matrix.DE.C} and \eqref{eq:matrix.DE.S.formula} remain valid with $f = g \in \Dd_S$ and $\lambda = -1$. Moreover, for $f \in \dom d$ and $g \in \Dd_S$, the action of $Bg \in \Dd_S^*$ is given by
	\begin{equation}
		\label{eq:matrix.DE.B.dom.D.formula}
		\dprod{Bf}{g}_{\Dd_S^* \times \Dd_S} = - \int_{\Omega} f \, \formb \cdot \overline{\nabla g} \d x.
	\end{equation}
\end{lem}

\begin{proof}	
	In order to show \eqref{eq:matrix.DE.Dd.S.accretive}, we point out that, by \eqref{eq:matrix.DE.accretive.pi.lambda.positive.definite}, we have
	\begin{equation}
	\label{eq:matrix.DE.Dd.S.subset.H1}
		\norm{\nabla f}^2 + \norm{(\re q)^\frac12 f}^2 + \norm{f}^2 \lesssim \norm{f}_S^2, \qquad f \in \Coo (\Omega).
	\end{equation}
	Since $\Coo (\Omega)$ is dense in $H_0^1 (\Omega)$, this implies $\Dd_S \subset H_0^1 (\Omega)$. The claims 
	\begin{equation}
		(\re \boldpi_0)^\frac12 \nabla f \in L^2 (\Omega)^n, \qquad (\re q)^\frac12 f \in L^2 (\Omega), \qquad f \in \Dd_S,
	\end{equation}
	then follow by standard arguments from the definition of $\Dd_S$; indeed, if
	\begin{equation}
	\label{eq:matrix.DE.Dd.S.approx.sequ}
		f \in \Dd_S, \qquad \{f_m\} _m\subset \Coo(\Omega), \qquad \norm{f_m - f}_S \to 0, \qquad m \to \infty,
	\end{equation}
	then by \eqref{eq:matrix.DE.Dd.S.subset.H1} also $f_m \to f$ in $H^1(\Omega)$ as $m \to \infty$. Since $\{(\re \boldpi_0)^\frac12 \nabla f_m\}_m$ is Cauchy in $L^2 (\Omega)^n$ by \eqref{eq:matrix.DE.Dd.S.approx.sequ}, it has a limit $h$ in $L^2 (\Omega)^n$. For any test function $\varphi \in \Coo (\Omega)^n$, we derive the following
	\begin{equation}
		\begin{aligned}
			\iprod{h}{\varphi} = & \lim_{m \to \infty} \iprod{(\re \boldpi_0)^\frac12 \nabla f_m}{\varphi} \\
			= & \lim_{m \to \infty} \iprod{\nabla f_m}{(\re \boldpi_0)^\frac12\varphi} = \iprod{\nabla f}{(\re \boldpi_0)^\frac12 \varphi} = \iprod{(\re \boldpi_0)^\frac12 \nabla f}{\varphi}.
		\end{aligned}
	\end{equation}
	This in turn implies $(\re \boldpi_0)^\frac12 \nabla f = h\in L^2 (\Omega)$ by a density argument. Analogously, one shows that $(\re q)^\frac12 f_m \to (\re q)^\frac12 f \in L^2 (\Omega)$ converges in $L^2 (\Omega)$ as $m \to \infty$. Formula~\eqref{eq:matrix.DE.S.formula} with $\lambda = -1$ then immediately extends to $f=g \in \Dd_S$, noting that the sectoriality of $\boldpi_0$ and $q$ implies continuity of both sides in $f$ with respect to convergence in $\Dd_S$.
	In order to show that \eqref{eq:matrix.DE.C} holds with $f \in \Dd_S$, consider a sequence as in \eqref{eq:matrix.DE.Dd.S.approx.sequ}. By construction of $C$ and since $(\re \boldpi_0)^\frac12 \nabla f_m \to (\re \boldpi_0)^\frac12 \nabla f$ in $L^2 (\Omega)^n$ as $m \to \infty$, for all test functions $\varphi \in \Coo (\Omega)$ we have
	\begin{equation}
		\begin{aligned} 
			\iprod{\omega^{-1}Cf}{\varphi} = \lim_{m \to \infty} \iprod{\omega ^{-1}\overline \formb \cdot \nabla f_m}{\varphi} & = \lim_{m \to \infty} \iprod{(\re \boldpi_0)^\frac12 \nabla f_m}{(\re \boldpi_0)^{-\frac12}\formb \omega^{-1} \varphi} \\
			& = \iprod{(\re \boldpi_0^{\frac12} \nabla f}{(\re \boldpi_0)^{-\frac12}\formb \omega^{-1} \varphi} \\[1.5mm]
			& = \iprod{\omega^{-1}\overline \formb \cdot \nabla f}{\varphi}.
		\end{aligned}
	\end{equation}
	For the third identity above, note that
	\begin{equation}
		(\re \boldpi_0)^{-\frac12}\formb \omega^{-1} \varphi = (d-\la)^{-1} (\re \boldpi_0)^{-\frac12}\formb (d-\la) w^{-1} \varphi \in L^2(\Omega), 
	\end{equation}
	see Assumption~\ref{asm:IV}~(i),~(iv) and recall that $w \ge 1$. From density arguments, one concludes $\omega^{-1} Cf = \omega^{-1}\overline \formb \cdot \nabla f$, i.e.\ that \eqref{eq:matrix.DE.C} holds. The claim \eqref{eq:matrix.DE.B.dom.D.formula} can be shown similarly, cf.\ \eqref{eq:matrix.DE.B.dom.D}.
\end{proof}

\begin{proof}[Proof of Theorem {\rm~\ref{thm:matrix.DE.accretive}}]
	By Lemma~\ref{lem:matrix.DE.accretive}, Assumption~\ref{asm:IV} is fully satisfied and the claims \eqref{eq:matrix.DE.spec.equivalent}--\eqref{eq:matrix.DE.pspec.equiv} follow from Theorem~\ref{thm:matrix.DE}. It only remains to show that $\Aa_0$ is m-accretive in $L^2 (\Omega) \oplus L^2 (\Omega)$ with~\eqref{eq:matrix.DE.accretive.domain}; the density of $\dom \Aa_0$ and the generation of the semigroup then follow from classical results, see e.g.\ \cite[\S V.3.10, \S IX.1]{Kato-1995}.
	In order to show that $\Aa_0$ is accretive, let $(f, g) \in \dom \Aa_0$. From the sectoriality of $\boldpi_0 = \mathbf I + (d+1)^{-1} \mathbf B$, from $f \in \Dd_S$ and \eqref{eq:matrix.DE.Dd.S.accretive}, it then follows that
	\begin{equation}
	\label{eq:matrix.DE.accr.S.L1}
			| (d + 1)^{-1} \mathbf B \nabla f\cdot \overline {\nabla f} | \lesssim  \re \boldpi_0 \nabla f \cdot \overline {\nabla f}  + |\nabla f|^2 \in L^1 (\Omega).
	\end{equation}
	Since $(D + 1)^{-1}Cf + g \in \dom d$ as $(f, g) \in \dom \Aa_0$, applying formula \eqref{eq:matrix.DE.B.dom.D.formula} and \eqref{eq:matrix.DE.C} according to Lemma~\ref{lem:matrix.DE.accr.Dd.S}, we further conclude 
	\begin{equation}
	\label{eq:matrix.DE.Cfg.L1}
		\begin{aligned}
			\dprod{B((D + 1)^{-1}Cf + g)}{f}_{\Dd_S^* \times \Dd_S} & = - \int_\Omega ((d + 1)^{-1} \overline \formb \cdot \nabla f + g) \formb \cdot \overline{\nabla f} \d x \\
			&  = - \int_{\Omega} (d + 1)^{-1} \mathbf B \nabla f \cdot \overline{\nabla f} \d x + \int_{\Omega} g \formb \cdot \overline{\nabla f} \d x;
		\end{aligned}
	\end{equation}
	here in the second equality we used \eqref{eq:matrix.DE.accr.S.L1}, which then also gives $g \formb \cdot \overline{\nabla f} \in L^1 (\Omega)$. Consequently, formula \eqref{eq:matrix.DE.S.formula} according to Lemma~\ref{lem:matrix.DE.accr.Dd.S} gives
		\begin{align}
			\iprod{Af + Bg}{f} & = \dprod{S(-1)f}{f}_{\Dd_S^* \times \Dd_S} + \dprod{B((D + 1)^{-1}Cf + g)}{f}_{\Dd_S^* \times \Dd_S} - \|f\|^2\\
			\label{eq:matrix.DE.accr.AB} & = \norm{\nabla f}^2 + \int_\Omega q |f|^2 \d x + \int_{\Omega} g \formb \cdot \overline{\nabla f} \d x.
		\end{align}
	Using  $g \formb \cdot \overline{\nabla f} \in L^1 (\Omega)$, it follows from \eqref{eq:matrix.DE.C} with Lemma~\ref{lem:matrix.DE.accr.Dd.S} that
	\begin{equation}
	\label{eq:matrix.DE.accr.CD}
		\iprod{Cf + Dg}{g} = \int_\Omega  \overline \formb \cdot \nabla f \overline g \d x + \int_\Omega d |g|^2 \d x;
	\end{equation}
	this in turn gives $d |g|^2 \in L^1 (\Omega)$, thus $g \in \dom (\re d)^\frac12$ by the sectoriality of $d$ and \eqref{eq:matrix.DE.accretive.domain} is shown. Combining \eqref{eq:matrix.DE.accr.AB} and \eqref{eq:matrix.DE.accr.CD}, we obtain accretivity of $\Aa_0$ as follows
	\begin{equation}
		\re \iprod{\Aa_0 (f, g)}{(f, g)}_\Hh = \norm{\nabla f}^2 + \int_{\Omega} \re q |f|^2 \d x + \int_{\Omega} \re d |g|^2 \d x \ge 0.
	\end{equation}
	In order to conclude m-accretivity of $\Aa_0$, we show that $\rho (\Aa_0) \cap (-\infty, 0) \neq 0$, see e.g.\ \cite[\S V.3.10]{Kato-1995}. This however, follows from taking complements in \eqref{eq:matrix.DE.spec.equivalent} and the proof of Lemma~\ref{lem:matrix.DE.S.coercive}, where one can easily show that, if $\lambda < 0$ has large enough modulus, then \eqref{eq:matrix.DE.S.coercive} is satisfied with $z_\lambda = 0$, i.e.\ that $S(\lambda)$ is coercive on $\Dd_S$ and thus $0 \in \rho (S_0 (\lambda))$ for  $\lambda \in (-\infty, 0)$ with $|\la|$ sufficiently large.
\end{proof}

\section{Dirac operators with Coulomb type potentials}
\label{sec:Dirac}

We apply our abstract method and use a dominant first Schur complement to construct a self-adjoint realisation of the Dirac operator
	\begin{equation}
		\label{eq:Dirac}
		\Aa = \begin{pmatrix}
			V + 1 & - \i \sigma \cdot \nabla \\[1mm]
			- \i \sigma \cdot \nabla & V - 1
		\end{pmatrix}
	\end{equation}
	acting in $L^2 (\R^3)^2 \oplus L^2 (\R^3)^2$, where as usual we denote 
	\begin{equation}
		\sigma \cdot \nabla = \sum_{j=1}^3 \sigma_j \partial_j
	\end{equation}
	with $\sigma_j$, $j = 1,2,3$, being the Pauli matrices
	\begin{equation}
		\sigma_1 = \begin{pmatrix}
			0 & 1 \\
			1 & 0
		\end{pmatrix}, \qquad
		\sigma_2 = \begin{pmatrix}
			0 & -\i \\
			\i & 0
		\end{pmatrix}, \qquad
		\sigma_3 = \begin{pmatrix}
			1 & 0 \\
			0 & -1
		\end{pmatrix}.
	\end{equation}
	Our assumptions on the real-valued potential are minimal with respect to the employed form methods. Indeed, $V$ is only assumed locally square integrable such that it satisfies the Hardy type inequality~\eqref{eq:Hardy.Dirac} below, which ensures the coercivity of the Schur complement on its form domain for a suitable set of spectral parameters.
	
	\begin{asm}
		Let $V \in L^2_{\rm loc} (\R^3, \R)$ be bounded above with
		\begin{equation}
			\Gamma \defeq \esssup_{x \in \R^3} V (x) < \infty
		\end{equation}
		and assume there exists $\Lambda > \Gamma - 1$ such that
		\begin{equation}
			\label{eq:Hardy.Dirac}
			\int_{\R^3} \frac{|\sigma \cdot \nabla f|^2}{1 + \Lambda - V} \d x + \int_{\R^3} (V + 1 - \Lambda) |f|^2 \d x \ge 0, \qquad f \in  C_0^\infty (\R^3)^2.
		\end{equation}
	\end{asm} 

	\begin{rem}
		It was shown in \cite{Dolbeault-Esteban-Loss-Vega-2004,Dolbeault-Esteban-Sere-2000} that the above Hardy-Dirac inequality is satisfied for certain Coulomb type potentials, in particular for $V(x) = -\nu / |x|$ with $\nu \in (0,1]$; in the latter case, $\Gamma = 0$ and $\Lambda = \sqrt{1-\nu}$, see~\cite[Thm.\ 1, Cor.\ 3]{Dolbeault-Esteban-Loss-Vega-2004}. \hfill //
	\end{rem}

	By means of the construction in Section~\ref{sec:matrix.Schur.complement}, a self-adjoint realisation of the operator matrix \eqref{eq:Dirac} with a spectral gap $(\Gamma - 1, \Lambda)$ can be obtained, see Theorem~\ref{thm:Dirac} below. Moreover, the spectral equivalence to its Schur complement on $(\Gamma-1, \infty)$ can be established, see Remark~\ref{rem:Dirac} below. 

	In~\cite{Esteban-Loss-2007}, the same self-adjoint operator matrix was constructed (and the same spectral gap was proven); the authors later presented an abstract version of this method in~\cite{Esteban-Loss-2008}. Even though their approach is similar to ours and uses both a distributional framework and a dominant Schur complement due to the Hardy-Dirac inequality~\eqref{eq:Hardy.Dirac}, it requires the additional assumption
	\begin{equation}\label{eq:EL.grad}
		\frac{\nabla V}{(\la_0+1-V)^2} \in L^2_{\rm loc} (\R^3, \R)
	\end{equation}
	for some $\la_0$ in the spectral gap. Inspired by~\cite{Esteban-Loss-2007,Esteban-Loss-2008}, in the recent work~\cite{Schimmer-Solovej-Tokus-2020} distinguished self-adjoint extensions of symmetric operators with spectral gaps were constructed and variational principles for their eigenvalues were derived. Applied to Dirac operators subject to Coulomb potentials $V(x) = -\nu / |x|$ with $\nu \in (0,1]$, this construction in particular allows to establish self-adjointness without the additional assumption~\eqref{eq:EL.grad}. We mention that while the elaborate abstract method in~\cite{Schimmer-Solovej-Tokus-2020} carefully defines suitable Hilbert spaces of test functions and distributions (which allow to rigorously implement the Frobenius-Schur factorisation of the resolvent), unlike our approach it heavily relies on the symmetry of the underling operator.

	In the sequel, $\iprod{\cdot}{\cdot}$ and $\norm{\cdot}$ shall denote the inner product and norm on $L^2 (\R^3)^2$.
	
	\begin{thm}\label{thm:Dirac}
		Let $\Aa$ be the matrix differential expression \eqref{eq:Dirac} understood in the standard distributional sense, let $\Aa_0 \defeq \Aa\vert_{\dom\Aa_0}$ be its action on the maximal domain
		\begin{equation}
			\label{eq:Dirac.dom}
			\begin{aligned}
				\dom \Aa_0 & \defeq \{ \, (f, g) \in \Dd_S \times L^2 (\R^3)^2\, : \\ 
				&\qquad \qquad \quad  Vf -\i \sigma \cdot \nabla g, ~-\i \sigma \cdot \nabla f + Vg \in L^2 (\R^3)^2\, \}.
			\end{aligned}
		\end{equation}
		Here the Hilbert space $\Dd_S$ is the closure
		\begin{equation}
			\Dd_S \defeq \overline{\Coo (\R^3)^2}^{\norm{\cdot}_S}, \qquad \norm{f}_S^2 \defeq \forms_{\lambda_0}[f] + \norm{f}^2, \qquad f \in \Coo (\R^3)^2,
		\end{equation}
		where $\lambda_0 \in (\Gamma - 1, \Lambda)$ is arbitrary and
		\begin{equation}
			\label{eq:Dirac.forms.lambda.0}
			\forms_{\lambda_0} [f] \defeq \int_{\R^3} \frac{|\sigma \cdot \nabla f|^2}{1 + \lambda_0 - V} \d x + \int_{\R^3} (V + 1 - \lambda_0) |f|^2 \d x \ge 0.
		\end{equation}
		Then $\Aa_0$ is independent of the choice of $\lambda_0$ in~\eqref{eq:Dirac.forms.lambda.0}. Moreover, it is densely defined and self-adjoint in $L^2 (\R^3)^2 \oplus L^2 (\R^3)^2$ and has a spectral gap
		\begin{equation}
			\label{eq:Dirac.spec.gap}
			\sigma (\Aa_0) \cap (\Gamma -1, \Lambda) = \emptyset.
		\end{equation}
	\end{thm}
	
	\begin{proof}
		For any $\lambda \in (\Gamma - 1,\infty)$, let the quadratic form $\forms_\lambda$ be defined on $\Coo (\R^3)^2$ analogously to \eqref{eq:Dirac.forms.lambda.0}. We first note that $\forms_\la$ is decreasing in $\la$ and that~\eqref{eq:Hardy.Dirac} translates into $\forms_\lambda$ being non-negative for $\lambda \in (\Gamma-1,\Lambda]$. Moreover, for $\lambda \in (\Gamma - 1, \Lambda)$, the forms $\forms_\lambda + \norm{\cdot}^2$ are pairwise equivalent norms. This can be shown similarly to Lemma~\ref{lem:matrix.DE.d.equiv}, by proving that for any $\lambda, \mu > \Gamma-1$ there exist $m_{\lambda,\mu}, M_{\lambda, \mu}>0$ such that a.e.\ in $\R^3$ and for all $f\in\Coo(\R^3)^2$ 
		\begin{equation}
			\label{eq:Dirac.slambda.loc.equiv}
			m_{\lambda, \mu} \forms_\mu [f] + (\mu-\lambda ) \norm{f}^2 \le \forms_\lambda [f] \le M_{\lambda, \mu} \forms_\mu [f] + (\mu-\lambda) \norm{f}^2.
		\end{equation}

		In order to apply the results in Section~\ref{sec:matrix.Schur.complement}, we proceed by indicating the remaining spaces and operators needed for Assumption~\ref{asm:II}. The following constructions are similar to Propositions~\ref{prop:DWE.spaces} and~\ref{prop:DWE.entries} for the wave equation and Propositions~\ref{prop:matrix.DE.spaces} and~\ref{prop:matrix.DE.operators} for the singular coefficient matrix differential operators, thus analogous parts will be merely sketched here. 
		
		We start by defining the spaces and showing that (as topological spaces) they are independent of $\la_0$. This in particular implies that $\Aa_0$ does not depend on $\la_0$. Clearly, $\Hh_1 := \Hh_2 := L^2 (\R^3)^2$ and we set $\Dd_{-S} \defeq \Dd_S^*$ and $\Dd_{-1} \defeq \Dd_1^*$ with $\Dd_1$ defined as the closure
		\begin{equation}
			\Dd_1 \defeq \overline{C_0^\infty (\R^3)^2}^{\norm{\cdot}_1}, \quad \norm{f}_1^2 \defeq \norm{f}_S^2 + \norm{V f}^2 + \norm{\sigma \cdot \nabla f}^2, \quad f \in C_0^\infty (\R^3)^2.
		\end{equation}
		We further introduce the (weighted) spaces
		\begin{equation}
			\label{eq:Dirac.Dd.2}
			\Dd_2 \defeq L^2 (\R^3)^2, \qquad \Dd_{-2} \defeq  L^2 (\R^3; (1 + \lambda_0 - V)^{-2})^2.
		\end{equation}
		One can show that the above spaces are well-defined and satisfy Assumption~\ref{asm:I}~(i), cf.~Propositions~\ref{prop:DWE.spaces} and~\ref{prop:matrix.DE.spaces}. Notice therefore that the weighted Lebesgue measure $(1 + \lambda_0 - V)^{-2} \d \la^3$ is a Borel measure, which implies the density of $L^2(\R^3)^2$  in $\Dd_{-2}$, see~\cite[Ex.~1.5.3~(c)]{Blank-Exner-Havlicek-2008}. The $\la_0$-independence of $\Dd_S$ and $\Dd_{-2}$, respectively, follow from the chain of inequalities~\eqref{eq:Dirac.slambda.loc.equiv} and a similar argument as in the proof of Lemma~\ref{lem:matrix.DE.d.equiv}.
		We continue with defining the matrix entries. To this end, the operators
		\begin{equation}
			A \defeq V+1 \in \Bb (\Dd_S, \Dd_1^*), \qquad B \defeq - \i \sigma \cdot \nabla \in \Bb (L^2(\R^3)^2, \Dd_1^*),
		\end{equation}
		shall be defined as unique bounded extensions of
		\begin{equation}\label{eq:Dirac.AB}
			\begin{aligned}
				\dprod{Af}{g}_{\Dd_1^* \times \Dd_1} & \defeq \int_{\R^3} (V + 1)f \cdot \overline{g} \d x, \\
				\dprod{Bf}{g}_{\Dd_1^* \times \Dd_1} & \defeq \i \int_{\R^3} f \cdot \overline{(\sigma \cdot \nabla g)} \d x,
			\end{aligned}
			\qquad \quad f,g \in C_0^\infty (\R^3)^2,
		\end{equation}
		cf.~the constructions \eqref{eq:DWE.CD} and \eqref{eq:matrix.DE.AB}. This is possible due to the inequalities 
		\begin{equation}\label{eq:Dirac.AB.bdd}
				|(Af,g)_{\Dd_1^* \times \Dd_1}| \le \|f\|_S \|g\|_1, \qquad |(Bf,g)_{\Dd_1^* \times \Dd_1}| \le \|f\| \|g\|_1,
		\end{equation}
		which hold for $f, g \in \Coo(\R^3)^2$ by the construction of $\|\cdot\|_1$ and since $\| \cdot \| \le \| \cdot \|_S$. Moreover, the entry 
		\begin{equation}
			C \defeq -\i \sigma \cdot \nabla \in \Bb (\Dd_S, \Dd_{-2})
		\end{equation}
		shall be constructed as unique bounded extension of
		\begin{equation}
			Cf \defeq - \i \sigma \cdot \nabla f, \qquad f \in C_0^\infty (\R^3)^2,
		\end{equation}
		similarly to~\eqref{eq:matrix.DE.C} and~\eqref{eq:matrix.DE.C.bounded}. To this end, see~\cite[Prop.\ 6]{Esteban-Loss-2007} for the proof of the inequality
		\begin{equation}\label{eq:Dirac.C.bdd}
			\|Cf \|^2_{\Dd_{-2}} = \int_{\R^3} \frac{|\sigma \cdot \nabla f|^2}{(1+\la_0-V)^2} \d x \lesssim  \|f\|_S
		\end{equation}
		for $f \in \Coo(\R^3)^2$. Finally,  let
		\begin{equation}
			D \defeq (V - 1) \in \Bb (L^2(\R^3)^2, \Dd_{-2}), \quad  D_0 \defeq D\vert_{\dom D_0}, \quad \dom D_0 \defeq (\dom V)^2.
		\end{equation}

		It follows that the operators defined above satisfy Assumptions~\ref{asm:I}~(ii) and \ref{asm:II}~(ii), cf.\ Propositions~\ref{prop:DWE.entries} and~\ref{prop:matrix.DE.operators}, thus Assumption~\ref{asm:II} holds. Moreover, we set
		\begin{equation}
			D^\ddagger (\lambda) \defeq (D - \lambda)^{-1} \in \Bb (\Dd_{-2}, L^2(\R^3)^2), \qquad \lambda \in \Theta \defeq (\Gamma - 1, \infty) \subset \rho(D_0),
		\end{equation}
		and point out that $\Theta$ satisfies the inclusion \eqref{eq:Theta}, cf.\ Lemma~\ref{lem:matrix.DE.d.equiv} regarding the claimed boundedness of $D^\ddagger(\lambda)$.
		The matrix $\Aa_0$ and its first Schur complement $S_0 (\cdot)$ are well-defined as in Definition~\ref{def:matrix.Schur.compl.family} with $\dom \Aa_0$ as in \eqref{eq:Dirac.dom}. It is obvious that the action of $\Aa_0$ coincides with its standard definition by distributional operations, where $A$ and $C$ are extended by continuity to $\Dd_S$ analogously as in Remarks~\ref{rem:DWE.distr.act} and~\ref{rem:matrix.DE.distr.operators}. It thus remains to explain that $\Aa_0$ is densely defined, self-adjoint and that \eqref{eq:Dirac.spec.gap} holds. 

		Let $\lambda \in (\Gamma - 1, \infty)$ be arbitrary but fixed. Considering the definition of $S(\la)$ and the inequalities in~\eqref{eq:Dirac.slambda.loc.equiv}, it follows that
		\begin{equation}
			\label{eq:Dirac.S.bdd}
			S(\lambda) \in \Bb (\Dd_S, \Dd_S^*), \qquad (S(\lambda)f, f)_{\Dd_S^* \times \Dd_S} = \forms_\lambda[f], \qquad f \in C_0^\infty(\R^3)^2,
		\end{equation}
		cf.\ Proposition~\ref{prop:matrix.DE.S.bounded}. Moreover, it is not difficult to use~\eqref{eq:Dirac.slambda.loc.equiv} to show that there exists a shift $z_\lambda \in \R$ such that
		\begin{equation}
			\label{eq:Dirac.shift.coercive}
			\abs{\dprod{(S (\lambda)-z_\lambda)f}{f}_{\Dd_S^* \times \Dd_S}} \gtrsim \norm{f}_S^2, \qquad f\in C_0^\infty (\R^3)^2,
		\end{equation}
		which means that $S (\lambda)-z_\lambda$ is coercive on $\Dd_S$. Hence, by the Lax-Milgram Theorem, see e.g.\ \cite[Thm.\ IV.1.1]{Edmunds-Evans-1987}, we conclude
		\begin{equation}
			S_z^\ddagger (\lambda) \defeq (S(\lambda)-z_\lambda)^{-1} \in \Bb (\Dd_S^*, \Dd_S), \qquad z_\lambda\in\rho(S_0(\lambda)),
		\end{equation}
		and that $\dom S(\lambda)$ is dense in $\Dd_S$ and thus also in $L^2 (\R^3)^2$. In particular, one can choose $z_\lambda=0$ if $\lambda \in (\Gamma-1,\lambda_0)$ and thus
		\begin{equation}\label{eq:Dirac.S.gap}
			(\Gamma-1,\lambda_0) \subset \rho(S_0).
		\end{equation}
		
		We proceed to show that the numerical range of $\Aa_0$ is a subset of $\R$, cf.~the proofs of \cite[Thm.\ 4]{Esteban-Loss-2007} and~\cite[Thm.~1]{Esteban-Loss-2008}. To this end, let $(f,g) \in \dom \Aa_0$. Then $S(\la_0) f \in \Dd_S^*$ by~\eqref{eq:Dirac.S.bdd} and hence
		\begin{equation}
			\langle Af + Bg, f \rangle = (S(\la_0) f, f)_{\Dd_S^*\times \Dd_S} + \la_0 \|f\|^2 + (B ((D-\la_0)^{-1} Cf + g), f)_{\Dd_S^*\times \Dd_S}.
		\end{equation}
		Moreover, with $f_n \in \Coo(\R^3)^2$ converging to $f$ in $\Dd_S$, we have
		\begin{equation}
			\begin{aligned}
				(B ((D-\la_0)^{-1} Cf + g), f_n)_{\Dd_S^*\times \Dd_S} & = \i \int_{\R^3} ( (D-\la_0)^{-1}Cf + g) \overline{(\sigma \cdot\nabla f_n)} \d x \\
				& = \int_{\R^3} (Cf + (D-\la_0) g) \frac{\overline{C f_n}}{V - 1 - \la_0} \d x \\[3mm]
				& = \langle Cf + (D-\la_0) g, (D-\la_0)^{-1} Cf_n \rangle;
			\end{aligned}
		\end{equation}
		notice that the defining formula~\eqref{eq:Dirac.AB} for $B$ holds also for $f \in L^2(\R^3)^2$ by inequality~\eqref{eq:Dirac.AB.bdd}. Considering~\eqref{eq:Dirac.C.bdd}, the above identity is continuous with respect to convergence in $\Dd_S$ and remains valid in the limit $n\to \infty$; notice that $(f, g) \in \dom \Aa_0$ implies that $Cf + (D-\la_0) g \in L^2 (\R^3)^2$. In total, we have
		\begin{equation}
			\begin{aligned}
				\langle \Aa_0 (f, g), (f, g) \rangle_{\Hh} & = (S(\la_0) f, f)_{\Dd_S^*\times \Dd_S} + \la_0 \|f\|^2  \\[1mm]
				& \qquad \qquad  + \langle Cf + (D-\la_0) g, (D-\la_0)^{-1} Cf \rangle \\[1mm]
				& \qquad \qquad \qquad \qquad + \langle Cf + (D- \la_0) g, g \rangle + \la_0 \|g\|^2 \\[1mm]
				& = (S(\la_0) f, f)_{\Dd_S^*\times \Dd_S} + \la_0 (\|f\|^2 + \|g\|^2) \\[1mm]
				& \qquad \qquad  + \langle Cf + (D-\la_0) g, (D-\la_0)^{-1} (Cf + (D-\la_0) g) \rangle.
			\end{aligned}
		\end{equation}
		The last expression above, however, is clearly real.
		
		It remains to employ Corollary~\ref{cor:family.second.implication}~(i) in order to conclude~\eqref{eq:Dirac.spec.gap}; indeed, it then follows from $W(\Aa_0) \subset \R$ and~\eqref{eq:Dirac.spec.nonempty} below that $\Aa_0$ is densely defined and self-adjoint. (Note that the density of $\dom \Aa_0$ can also be shown using Corollary~\ref{cor:family.dense.domain}~(i).) Considering the above, it suffices to prove $B ((\dom V)^2) \subset \Dd_S^*$. The latter follows from a density argument and~\eqref{eq:Dirac.C.bdd}, which leads to the estimate
		\begin{equation}
			|(Bf, g)_{\Dd_1^*\times \Dd_1}| \le \int_{\R^3} |(V-1-\la_0) f| \frac{|\sigma \cdot \nabla g|}{|V-1-\la_0|} \d x \le \|(V-1-\la_0) f\| \|g\|_S
		\end{equation}
		for $f \in (\dom V)^2$ and $g \in \Coo(\R^3)^2$, cf.\ \cite[Lem.\ 7]{Esteban-Loss-2007} and~\cite[Lem.~4]{Esteban-Loss-2008}. Corollary~\ref{cor:family.second.implication}~(i) is then applicable with $\Sigma \defeq (\Gamma -1, \lambda_0)$ and gives
		\begin{equation}\label{eq:Dirac.spec.nonempty}
			\sigma (\Aa_0) \cap (\Gamma -1, \lambda_0) \subset \sigma (S_0) \cap (\Gamma -1, \lambda_0) = \emptyset,
		\end{equation}
		see~\eqref{eq:Dirac.S.gap}. Since $\lambda_0$ was arbitrary and $\Aa_0$ is independent of $\lambda_0$, \eqref{eq:Dirac.spec.gap} follows.
	\end{proof}
	
	\begin{rem}\label{rem:Dirac}
		Besides the spectral gap \eqref{eq:Dirac.spec.gap} found in \cite{Esteban-Loss-2007} leading to the self-adjoint\-ness of $\Aa_0$, our method provides the equivalence of (point and essential) spectra of $\Aa_0$ and its first Schur complement on $[\Lambda, \infty)$, i.e.
		\begin{equation}
			\sigma (\Aa_0) \cap [\Lambda, \infty) = \sigma (S_0 (\cdot)) \cap [\Lambda, \infty)
		\end{equation}
		and analogously with $\sigma_{\rm p}$ and $\sigma_{\operatorname{e2}}$ instead of $\sigma$, see Corollaries~\ref{cor:point.spectrum.family}--\ref{cor:family.second.implication} and the proof of Theorem~\ref{thm:Dirac}. (Due to the self-adjointness of $\Aa_0$, it suffices to consider real spectral parameters.) For $\lambda \in (\Gamma -1, \infty)$, our realisation of the Schur complement is defined as the maximal restriction $S_0 (\lambda)\defeq S(\lambda)\vert_{\dom S_0 (\lambda)}$ where
		\begin{equation}
			\begin{aligned}
				S(\lambda) & \defeq V + 1 - \lambda + (\sigma \cdot \nabla) (V-1-\lambda)^{-1} (\sigma \cdot \nabla), \\
				\dom S_0(\lambda) & \defeq \set{f \in \Dd_S}{Vf + (\sigma \cdot \nabla) (V-1-\lambda)^{-1} (\sigma \cdot \nabla) f \in L^2 (\R^3)^2}.
			\end{aligned}
		\end{equation}
		By construction, the above operations are understood in the standard distributional sense. More precisely, the action of $S(\lambda)$ on $\Dd_S$ is obtained by continuous extension, i.e.\ as a limit in $\Dd'(\R^3)^2$ of distributions of the form
		\begin{equation}
			(V + 1 -\la) f + (\sigma \cdot \nabla) (V-1-\lambda)^{-1} (\sigma \cdot \nabla) f \in \Dd' (\R^3)^2, \quad f\in\Coo(\R^3)^2,
		\end{equation}
		cf.\ Remarks~\ref{rem:DWE.distr.act} and~\ref{rem:matrix.DE.distr.operators}. \hfill //
	\end{rem}


\section{Klein-Gordon equation with purely imaginary potential}
\label{sec:KG}

We consider a Klein-Gordon equation on $\Rn$ with potential $V$ and mass $m>0$
\begin{equation}
	(\partial_t - \i V (x))^2 u(x,t) - \Delta_x u(x,t) + m^2 u(x,t) = 0, \qquad x \in \Rn, \quad t \ge 0;
\end{equation}
here the involved physical constants are normalised for the sake of simplicity, see e.g.~\cite{Veselic-1991} for the full generality. The equation above has been studied in a large number of works, employing various (operator theoretic) approaches, see e.g.~\cite{Langer-Najman-Tretter-2006,Langer-Najman-Tretter-2008,Veselic-1991}. After suitable transformations, one arrives at the following first order Cauchy problem 
\begin{equation}
	\label{eq:KG.Cauchy}
	\partial_t \left(
	\begin{array}{c}
		u_1 (t,x) \\
		u_2 (t,x)
	\end{array} \right) = \i \left(
	\begin{array}{cc}
		0 & 1 \\
		- \Delta_x + m^2 - V(x)^2 & 2V(x)
	\end{array} \right)\left(
	\begin{array}{c}
		u_1 (t,x) \\
		u_2 (t,x)
	\end{array} \right).
\end{equation}
We mention that also another system of equations arising by means of different transformations has been of interest, for instance in~\cite{Langer-Najman-Tretter-2008,Veselic-1991}. Motivated by the underlying physical problem, the potential is assumed to be real-valued in all works above. This results in a certain indefiniteness of the problem, which makes its spectral analysis less straightforward, see e.g.~\cite{Langer-Najman-Tretter-2006} where Krein spaces together with a smallness condition for the potential were employed.

For purely imaginary potentials $V = \i W$ with real-valued $W$, however, the problem can be reduced to a suitable wave equation, see~\eqref{KG.DWE} below; note that the latter has a special structure since the damping $W$ is relatively bounded with respect to $-\Delta + m^2 + W^2$. Assuming only that the potential is locally square integrable, we define the matrix expression
\begin{equation}\label{eq:KG.Aa}
	\Aa : = \begin{pmatrix}
		0 & I \\
		- \Delta + m^2 + W^2 & 2\i W
	\end{pmatrix}
\end{equation}
on the right hand side of~\eqref{eq:KG.Cauchy} as a densely defined, boundedly invertible operator in a suitable Hilbert space and show spectral equivalence to its second Schur complement
\begin{equation}\label{eq:KG.Ss}
	S(\la) := \frac 1\la \left( -\Delta + m^2 + (W + \la \i )^2 \right), \qquad \la \in \C \setminus \{0\}, 
\end{equation}
see Theorem~\ref{thm:KG} below. Moreover, in Example~\ref{ex:KG}, we consider the special case $W(x) =  x$ in one dimension and show that the spectrum of the resulting operator matrix is empty, which is in line with the analogous result for the Airy operator in case of the Schr\"odinger equation.

We denote $\norm{\cdot} \defeq \norm{\cdot}_{L^2 (\Rn)}$ and  $\iprod{\cdot}{\cdot} \defeq \iprod{\cdot}{\cdot}_{L^2 (\Rn)}$ and, if its meaning is clear from the context, adapt the same notation for norm and inner product in $L^2 (\Rn)^n$. The operator matrix and its Schur complement are defined as follows
\begin{equation}\label{eq:KG.Aa0.S0}
	\Aa_0 := \Aa\vert_{\dom \Aa_0}, \qquad S_0 (\la) := S(\la)\vert_{\dom S_0}, \qquad \la \in \C \setminus \{0\},
\end{equation}
acting in the underlying Hilbert spaces $\Ww (\Rn) \oplus L^2 (\Rn)$ and $L^2 (\Rn)$, respectively, with $\Aa$ and $S(\la)$ as in~\eqref{eq:KG.Aa} and~\eqref{eq:KG.Ss} understood in the standard distributional sense, on  their respective domains
\begin{equation}\label{eq:KG.dom}
	\begin{aligned}
		\dom \Aa_0 & := \dom S_0 \times \Ww (\Rn),  \\
		\dom S_0 & := \set{ f \in \Ww (\Rn)}{ (\Delta  - W^2) f \in L^2 (\Rn) };
	\end{aligned}
\end{equation}
notice that the domain of $S_0 (\la)$ is independent of $\la \in \C\setminus \{0\}$ and that the domain of $\Aa_0$ is diagonal due to the previously mentioned relative boundedness. In the above, the first component of the product space is
\begin{equation}\label{eq:KG.Ww.Rn}
	\Ww (\Rn) := H^1 (\Rn) \cap \dom W
\end{equation}
considered as Hilbert space equipped with the inner product
\begin{equation}\label{eq:KG.Ww}
	\langle f,g\rangle_\Ww^2 := \int_{\Rn} \nabla f \cdot\overline{\nabla g} \d x + \int_{\Rn} W^2 f \overline g \d x + \int_{\Rn}  f \overline g \d x,	\qquad f, g \in \Ww(\Rn).
\end{equation}

\begin{thm}\label{thm:KG}
	Assume that $W \in L^2_{\operatorname{loc}} (\Rn,\R)$ and let $\Aa_0$ and $S_0(\cdot)$ be as in~\eqref{eq:KG.Aa0.S0}, \eqref{eq:KG.dom}. Then $\Aa_0$ is closed with $0 \in \rho (\Aa_0)$ and its domain is dense in $\Ww (\Rn) \oplus \Ww (\Rn)$ and in $\Ww (\Rn) \oplus L^2 (\Rn)$.	Moreover, the {\rm(}point and essential{\rm)} spectra of $\Aa_0$ and $S_0 (\cdot)$ are equivalent, more precisely,
	\begin{equation}\label{eq:KG.equiv}
		\sigma (\Aa_0)  = \sigma (S_0 (\cdot)), \quad \sigma_{\rm p} (\Aa_0)  = \sigma_{\rm p} (S_0 (\cdot)), \quad \sigma_{\rm e2} (\Aa_0)  = \sigma_{\rm e2} (S_0 (\cdot)).
	\end{equation}
\end{thm}

\begin{proof}
	We relate the problem to a suitable damped wave equation and apply the results in Section~\ref{sec:DWE}. In detail, for $\la \in \C$ we have
	\begin{equation} 
		\Aa_0 - \la = \i \diag (I, \i)^{-1} (\widetilde \Aa_0 -  \i\la ) \diag (I, \i)
	\end{equation}
	with $\widetilde \Aa_0$ being the following linear operator in $\Ww (\Rn) \oplus L^2 (\Rn)$
	\begin{equation}\label{KG.DWE}
		\widetilde \Aa_0 := \begin{pmatrix}
			0 & I \\
			\Delta - (m^2 + W^2) & - 2 W
		\end{pmatrix}, \qquad \dom \widetilde \Aa_0 := \dom \Aa_0.
	\end{equation}
	This clearly gives the equivalence
	\begin{equation}\label{eq:KG.DWE.equiv}
		\la \in \sigma (\Aa_0) \iff  \i\la \in \sigma (\widetilde \Aa_0),
	\end{equation}
	and analogously with $\sigma_{\rm p}$ and $\sigma_{\operatorname{e2}}$ instead of $\sigma$.

	We apply Theorem~\ref{thm:DWE} to $\widetilde \Aa_0$ and its second Schur complement $\widetilde S_0 (\cdot)$ in a suitable way, i.e.\ with $\Omega := \Rn$, the potential $q := m^2 + W^2$ and the damping $a:= W$ therein. Even though the latter might be indefinite, it is relatively bounded with bound zero with respect to the potential in the sense of quadratic forms, and the spectral equivalence can thus be implemented analogously. Indeed, merely the following adjustments have to be made.
	
	Instead of~\eqref{eq:DWE.norm.Ww.Omega},~\eqref{eq:DWE.Dd.S} and~\eqref{eq:DWE.norm.S}, define $\Ww (\Omega) := \Ww (\Rn)$ as in~\eqref{eq:KG.Ww.Rn}, \eqref{eq:KG.Ww}, and set $\Dd_S := \Ww (\Rn)$; notice that~\eqref{eq:DWE.norm.Ww.Omega} and~\eqref{eq:KG.Ww} give equivalent norms and, as a topological space, $\Ww (\Rn)$ is in fact independent of choosing either one of them. Taking into account these modifications, Propositions~\ref{prop:DWE.spaces},~\ref{prop:DWE.entries} and~\ref{prop:DWE.S.action} remain valid. Moreover, since $\Ww (\Rn) \subset \dom W$, the domain of $\widetilde \Aa_0$ defined as in~\eqref{eq:DWE.dom.Aa0.S0} indeed coincides with~$\dom  \Aa_0$ in~\eqref{eq:KG.dom}. In view of Lemma~\ref{lem:DWE.shift.extension}, we consider the Schur complement
	\begin{equation}
		\widetilde S(\la) =  - \frac1 \la \left( -\Delta + m^2 + W^2 + 2\la W + \la^2 \right) \in \Bb (\Ww (\Rn), \Ww (\Rn) ^*),
	\end{equation}
	and the restriction $\widetilde S_0 (\cdot)$ to its maximal domain in $L^2 (\Rn)$, see~\eqref{eq:DWE.def.Aa.S} and~\eqref{eq:DWE.dom.Aa0.S0}, which since $\Ww (\Rn) \subset \dom W$, is given by
	\begin{equation}\label{eq:KG.dom.S0}
		\dom \widetilde S_0 (\la) = \dom S_0, \quad \la \in \C \setminus \{0\},
	\end{equation}
	where $\dom S_0$ is as in~\eqref{eq:KG.dom}. In order to prove the claims analogously to Theorem~\ref{thm:DWE}, it thus suffices to justify that $\dom S_0$ is dense in $L^2 (\Rn)$ and  that, for every $\la \neq 0$, there exists $z_\la \in \C$ such that~\eqref{eq:DWE.S.shift.inv} holds and $z_\la = 0$ can be chosen if $\la >0$ is sufficiently large. This, however, follows in a straightforward way from the order zero form-relative boundedness of the damping with respect to the potential. More precisely, using Cauchy-Schwarz' and Young's inequalities, one shows that for every $\delta >0$ it holds that
	\begin{equation}
		\int_{\Rn} |W| |f|^2 \d x \le \frac\delta 2  \int_{\Rn} W^2 |f|^2 \d x + \frac{1}{2\delta} \int_{\Rn}|f|^2 \d x, \quad f \in \dom (W^2).
	\end{equation}
	Hence, for $\la \in \C \setminus \{0\}$ and $f \in \Ww (\Rn)$, the above gives
	\begin{equation}
		\begin{aligned}
			\re ( -\la \widetilde S(\la) f, f)_{\Ww(\Rn) \times \Ww(\Rn)^*} & \ge \norm{\nabla f}^2 + (1 -  |\la|\delta) \norm{Wf}^2 \\
			& \qquad \qquad \qquad  + (m^2 + \re (\la^2) - |\la|\delta^{-1}) \norm{f}^2,
		\end{aligned}
	\end{equation}
	implying the existence of $z_\la\in\C$ such that $\widetilde S(\la) - z_\la $ is coercive on $\Ww(\Rn)$ and also that $z_\la = 0$ is a possible choice for any $\la >0$.
	The claimed invertibility in~\eqref{eq:DWE.S.shift.inv} now follows from the Lax-Milgram-Theorem, see~\cite[Cor.\ IV.1.2]{Edmunds-Evans-1987}; cf.\ the proofs of Lemmas~\ref{lem:DWE.shift.extension} and~\ref{lem:matrix.DE.S.coercive}. Combining the above, Corollaries~\ref{cor:family.dense.domain}~(ii) and~\ref{cor:point.spectrum.family}--\ref{cor:family.second.implication}, we obtain the density of $\dom \Aa_0 = \dom \widetilde \Aa_0$ in the claimed spaces, as well as
	\begin{equation}\label{eq:KG.tilde.equ}
		0 \neq \la \in \sigma (\widetilde \Aa_0) \iff \la \in \sigma (\widetilde S_0 (\cdot)),
	\end{equation}  
	and analogously with $\sigma_{\rm p}$ and $\sigma_{\operatorname{e2}}$ instead of $\sigma$.
	
	It remains to show that $\Aa_0$ is boundedly invertible. The equivalence~\eqref{eq:KG.equiv} then follows from \eqref{eq:KG.DWE.equiv},~\eqref{eq:KG.tilde.equ} and since for $\la \neq 0$ clearly 
	\begin{equation}
		0 \in \sigma (\widetilde S_0 (\i \la)) \iff 0 \in \sigma (S_0 (\la)),
	\end{equation}
	and analogously with $\sigma_{\rm p}$ and $\sigma_{\operatorname{e2}}$ instead of $\sigma$.
	To show $0 \in \rho (\Aa_0)$, in view of~\eqref{eq:KG.DWE.equiv}, it suffices to prove that
	\begin{equation}
		\widetilde \Aa_0^{-1 } = \Rr := \begin{pmatrix}
			- 2  T_0 ^{-1} W & - T_0^{-1}\\
			I & 0
		\end{pmatrix} \in \Bb (\Ww (\Rn), L^2 (\Rn))
	\end{equation}
	where $T_0$ is the linear operator in $L^2 (\Rn)$ given by
	\begin{equation}
		T_0 := - \Delta + m^2 + W^2, \qquad \dom T_0 := \dom S_0.
	\end{equation}
	However, the claimed boundedness of $\Rr$ readily follows from
	\begin{equation}
		W \in \Bb(\Ww(\R^n), L^2 (\R^n)), \qquad T_0^{-1} \in \Bb(L^2 (\R^n), \Ww (\R^n)),
	\end{equation}
	and the proof of $\widetilde \Aa_0 \Rr = I$ and of $\Rr \widetilde \Aa_0 = I$ on $\dom \Aa_0$ is straightforward.
\end{proof}

\begin{exple}\label{ex:KG}
	For the one dimensional purely imaginary potential $V = \i x$, our realisation of the Klein-Gordon Cauchy problem~\eqref{eq:KG.Cauchy} has empty spectrum. More precisely, we show that if $n=1$ and  $W (x) := x$, $x \in \R$, then $\sigma (\Aa_0) = \emptyset$.
	
	By Theorem~\ref{thm:KG}, we have $0 \in \rho (\Aa_0)$ and that $\sigma (\Aa_0)  = \sigma (S_0 (\cdot))$, where
	\begin{equation}
		S_0 (\la)  = \frac{1}{\la} \left( -\partial_x^2 + m^2 + (x + \i  \la)^2 \right), 
		\qquad \la \in \C \setminus\{0\},
	\end{equation}
	is the operator family in $L^2 (\R)$ on the ($\la$-independent) domain
	\begin{equation}
		\dom S_0 = H^2 (\R) \cap \dom (x^2);
	\end{equation}
	see ~\cite[Prop.\ 2.6.\ (i)]{Boegli-Siegl-Tretter-2017} for the domain separation property and note that
	\begin{equation}
		\dom (m^2 + (x + \i \la)^2) = \dom (x^2).
	\end{equation}
	The claim thus follows if we show $\sigma(S_0 (\cdot)) = \emptyset$. It is easy to see that 
	\begin{equation}
		T(\la) := \la S_0(\la) , \quad T(0) := - \partial_x^2 + m^2 + x^2, \quad \dom T (\la):= \dom S_0, \quad \la \in \C,
	\end{equation}
	is a holomorphic family of type (A) in the sense of~\cite[Sec.\ VII.2]{Kato-1995}; notice that in fact $T (\cdot)$ is also type (B) holomorphic, see~\cite[Sec.\ VII.4]{Kato-1995}, with form domain $\Ww(\R)$ as in \eqref{eq:KG.Ww.Rn},~\eqref{eq:KG.Ww}.	Since for every $\la \neq 0$, the operator $T (\la)$ is a bound zero perturbation of $T(0)$ and the latter has compact resolvent, see~\cite[Thm.\ XIII.67]{Reed-Simon-1978}, also $T (\la)$ has compact resolvent; see~\cite[Thm.\ IV.1.16]{Kato-1995} and note that
	\begin{equation}
		\norm{(T(0)-\mu)^{-1}}\to 0, \qquad \mu \to -\infty.
	\end{equation}
	Due to the analyticity of $T (\cdot)$, the isolated eigenvalues (of finite multiplicity) of $T (\la)$ depend analytically on $\la \in \C $, see~\cite[Sec.\ VII.1.3, Thm.\ VII.1.8]{Kato-1995}. Considering that, for $\la \in \i \R_+$, by unitary equivalence we have
	\begin{equation}
		\sigma (T (0)) =  \sigma (-\partial_x^2 + m^2 + (x - |\la|)^2 ) = \sigma (T (\la)),
	\end{equation}
	the spectrum of $T(\la)$ remains unchanged in $\la \in \C$ and thus, for $\la \neq 0$,
	\begin{equation}
		\la \sigma (S_0 (\la)) = \sigma (T(\la)) =  \sigma (T(0)) .
	\end{equation}
	Since it is well-known that $\sigma(T(0))$ does not contain zero, this gives $\sigma(S_0 (\cdot)) = \emptyset$.
\end{exple}

\section{An illustrative constant coefficient problem}
\label{sec:Laplace.powers}

In many applications, in particular in Sections~\ref{sec:DWE}--\ref{sec:KG}, the spaces $\Dd_S$ and $\Dd_{-S}$ are given by the form domain of the Schur complement and its anti-dual space. However, the latter is not always the case, as we demonstrate in a model problem where $\Dd_S = H^1 (\Rn)$ and $\Dd_{-S} = H^{-2} (\Rn)$, while $H^\frac32 (\Rn)$ is the form domain of the Schur complement. We point out that this example is of illustrative purpose and chosen as simple as possible; examples of similar structure can be found e.g.\ in \cite{Ibrogimov-Tretter-2017} where the entries are more general pseudodifferential operators, cf.~\cite[Sec.~5.2]{Gerhat-Siegl-2022-preprint} where $\Dd_S$ is the operator domain of the Schur complement, which is defined in a suitable weighted $L^2$-space $\Dd_{-2}$, and also Remark~\ref{rem:diag.dom}.

\begin{exple}
	In the Hilbert space $\Hh = \Hh_1 \oplus \Hh_2 = L^2 (\Rn) \oplus L^2 (\Rn)$, we consider the operator matrix
	\begin{equation}
	\label{eq:matrix.Laplace.powers}
		\Aa = \begin{pmatrix}
			\Delta & - \Delta^2 \\
			\sqrt{- \Delta} & \Delta
		\end{pmatrix}
	\end{equation}
	and the corresponding first Schur complement
	\begin{equation}
	\label{eq:Schur.compl.Laplace.powers}
		S(\lambda) = \Delta - \lambda + \Delta^2 (\Delta - \lambda)^{-1} \sqrt{- \Delta}, \qquad \lambda \in \C \setminus (-\infty, 0],
	\end{equation}
	acting in $L^2 (\Rn)$. Note that in this particular case, it is not difficult to explicitly determine the spectra of the operator matrix and its Schur complement above; however, the goal of this example is not their spectral analysis, but the illustration of spaces and operators underlying the spectral correspondence developed in Section~\ref{sec:matrix.Schur.complement}.

	\begin{prop}
		Let $\Aa_0 \defeq \Aa\vert_{\dom \Aa_0}$ with $\Aa$ as in \eqref{eq:matrix.Laplace.powers} and
		\begin{equation}
		\label{eq:dom.Aa.0.Laplace.powers}
			\dom \Aa_0 \defeq \set{(f, g) \in H^1 (\Rn) \times H^2 (\Rn)}{\Delta f - \Delta^2 g \in L^2 (\Rn)}.
		\end{equation}
		Moreover, let $S_0 (\lambda) \defeq S (\lambda)\vert_{\dom S_0 (\lambda)}$, $\lambda \in \C \setminus (- \infty, 0]$, be the family of maximal operators in $L^2 (\Rn)$ where $S(\lambda)$ is as in \eqref{eq:Schur.compl.Laplace.powers} and
		\begin{equation}
		\label{eq:dom.S.0.Laplace.powers}
			\dom S_0 (\lambda) \defeq \set{f \in H^1 (\Rn)}{\Delta f + \Delta^2 (\Delta - 1)^{-1} \sqrt{- \Delta} f \in L^2 (\Rn)}.
		\end{equation}
		In the above, the square root of $-\Delta$ is defined via functional calculus. Then $\Aa_0$ is densely defined and closed in $L^2 (\Rn) \oplus L^2 (\Rn)$ with
		\begin{equation}
		\label{eq:Laplace.powers.resolvent.half.plane}
			\set{\lambda \in \C}{\re \lambda > 0} \subset \rho (\Aa_0)
		\end{equation}
		and for the {\rm(}point and essential{\rm)} spectra of $\Aa_0$ and $S_0(\cdot)$ it holds that
		\begin{equation}
		\label{eq:Laplace.powers.spectra}
			\sigma (\Aa_0) \setminus (- \infty, 0] = \sigma  (S_0 (\cdot))
		\end{equation}
		and analogously with $\sigma_{\rm p}$ and $\sigma_{\operatorname{e2}}$ instead of $\sigma$.
	\end{prop}
	
	\begin{proof}[Sketch of proof]
		We indicate the objects needed for Assumption~\ref{asm:II}. Let
		\begin{equation}\label{eq:const.coeff.spaces}
			\begin{aligned}
				\Dd_S & \defeq H^1 (\Rn), \\
				\Dd_2 & \defeq H^2 (\Rn),
			\end{aligned}
			\qquad \quad 
			\begin{aligned}
				\Dd_{-S} = \Dd_{-1} & \defeq H^{-2} (\Rn), \\
				\Dd_{-2} & \defeq L^2 (\Rn).
			\end{aligned}
		\end{equation}
		Using the distributional Fourier transform, one can easily check that the following operators are bounded between the claimed spaces
		\begin{equation}
		\begin{array}{rcl}
		A \defeq & \Delta & \in \Bb (H^1 (\Rn), H^{-1} (\Rn)), \\[2mm]
		C \defeq & \!\!\sqrt{- \Delta}\!\! & \in \Bb (H^1 (\Rn), L^2 (\Rn)),
		\end{array}
		\quad \!
		\begin{array}{rcl}
		B \defeq & \!\!- \Delta^2\!\! & \in \Bb (H^2 (\Rn), H^{-2} (\Rn)), \\[2mm]
		D \defeq & \Delta & \in \Bb (H^2 (\Rn), L^2 (\Rn)),
		\end{array}	
		\end{equation}
		see e.g.\ \cite[Thm.\ 3.41]{Abels-2012}. Since $H^{-1} (\Rn)$ embeds continuously in $H^{-2} (\Rn)$, Assumption~\ref{asm:II} is satisfied with $D_0 \defeq D$ and we can define $\Aa_0$ and $S_0 (\cdot)$ as in Definition~\ref{def:matrix.Schur.compl.family} with
		\begin{equation}
		D^\ddagger(\lambda) \defeq (D - \lambda)^{-1}, \qquad \lambda \in \Theta \defeq \rho (D_0) = \C \setminus (- \infty, 0].
		\end{equation}
		One easily verifies that $\dom \Aa_0$ and $\dom S_0(\lambda)$ coincide with \eqref{eq:dom.Aa.0.Laplace.powers} and \eqref{eq:dom.S.0.Laplace.powers}.
		
		For the proof of \eqref{eq:Laplace.powers.spectra}, we outline the assumptions of Corollaries~\ref{cor:family.first.implication} and~\ref{cor:family.second.implication}. Let therefore $\lambda \in \C \setminus (- \infty, 0]$ be arbitrary but fixed. It is not difficult to see that there exists $z_\lambda > 0$ such that the inverse of $S(\lambda) - z_\lambda$ is bounded on $H^{-2} (\Rn)$; this follows from the fact that $S(\lambda)$ is unitarily equivalent to the multiplication operator by the symbol
		\begin{equation}
			s_\lambda (\xi) \defeq - |\xi|^2 - \lambda - \frac{|\xi|^5}{|\xi|^2 + \lambda}, \quad \xi \in \Rn,
		\end{equation}
		in the Fourier space. Indeed, it is elementary to prove the lower bound
		\begin{equation}
			|s_\lambda (\xi) - z_\lambda| \gtrsim |\xi|^3 + 1, \qquad \xi \in \Rn,
		\end{equation}
		if $z_\lambda > 0$ is large enough, which implies
		\begin{equation}
			S_z^\ddagger(\lambda) \defeq (S(\lambda) - z_\lambda)^{-1} \in \Bb (H^{-2}(\Rn), H^1(\Rn)).
		\end{equation}
		Moreover, $\dom S_0 (\lambda)$ contains $\Coo (\Rn)$ and is thus dense in $H^1 (\Rn)$. The relations in \eqref{eq:Laplace.powers.spectra} then follow from Corollary~\ref{cor:point.spectrum.family} and Corollaries~\ref{cor:family.first.implication} and~\ref{cor:family.second.implication} with
		\begin{equation}
			\Sigma \defeq \Theta = \C \setminus (- \infty, 0]=\rho(D_0).
		\end{equation}
		The density of $\dom \Aa_0$ in $L^2 (\Rn) \oplus L^2 (\Rn)$ is a consequence of Corollary~\ref{cor:family.dense.domain}~(ii). Finally, if $\re \lambda >0$, then one can easily derive the following lower bound
		\begin{equation}
			\abs{s_\lambda(\xi)} \ge \re \lambda>0, \qquad \xi \in \Rn,
		\end{equation}
		 which implies $0 \in \rho (S_0 (\lambda))$ and in turn \eqref{eq:Laplace.powers.resolvent.half.plane} by \eqref{eq:Laplace.powers.spectra}.
	\end{proof}

\end{exple}

\bibliographystyle{acm} 
\bibliography{references}

\end{document}